\documentclass[a4paper]{article}

\usepackage[american]{babel}
\usepackage{amsmath,amssymb,amsthm}
\usepackage{centernot,braket}
\usepackage{cancel}
\usepackage{lipsum}
\usepackage{calc}
\usepackage{xcolor}
\usepackage{csquotes}
\usepackage{comment}
\usepackage{fancyhdr}
\usepackage{titling}
\usepackage[textwidth=4cm]{todonotes}
\usepackage{hyperref}
\usepackage{nicematrix}

\usepackage{tikz}
\usetikzlibrary{cd}
\usetikzlibrary{calc}
\def\temp{&} \catcode`&=\active \let&=\temp

\newcommand{\TAB}{\quad}
\newcommand{\from}{\leftarrow}
\newcommand{\NOT}{\centernot}
\newcommand{\dd}{\,\mathrm{d}}
\newcommand{\CSTAR}{\ensuremath{C^*}}

\newcommand{\STARHOMS}{$*$-homomorphisms}

\DeclareMathSymbol{\mlq}{\mathord}{operators}{'134}
\DeclareMathSymbol{\mrq}{\mathord}{operators}{'42}
\newcommand{\mathquote}[1]{\mlq #1 \mrq}

\newcommand{\BLANK}{{-}}

\newenvironment{TIKZCD}{\[\begin{tikzcd}}{\end{tikzcd}\]\ignorespacesafterend}
\newenvironment{smalltikzcd}{\begin{tikzcd}[column sep=small]}{\end{tikzcd}}
\newcommand{\SHORTTIKZCD}[1]{\begin{tikzcd} #1 \end{tikzcd}}

\newcommand{\MATRIX}[1]{\left(\hskip\arraycolsep\begin{matrix}#1\end{matrix}\hskip\arraycolsep\right)}
\newcommand{\mediumMATRIX}[1]{
  \begingroup
  \setlength\arraycolsep{2pt}
  \renewcommand*{\arraystretch}{0.6}
  \left(\hskip\arraycolsep\begin{matrix}#1\end{matrix}\hskip\arraycolsep\right)
  \endgroup}
\newcommand{\rowMATRIX}[1]{
    \begingroup
    \setlength\arraycolsep{1pt}
    \renewcommand*{\arraystretch}{0.6}
    \left(\hskip\arraycolsep\begin{matrix}#1\end{matrix}\hskip\arraycolsep\right)
    \endgroup}

\renewcommand{\subset}{\subseteq}
\renewcommand{\supset}{\supseteq}
\renewcommand{\phi}{\varphi}
\renewcommand{\epsilon}{\varepsilon}
\renewcommand{\emptyset}{\varnothing}

\makeatletter
\newcommand{\xRightarrow}[2][]{\ext@arrow 0359\Rightarrowfill@{#1}{#2}}
\makeatother

\newcommand{\NATURALS}{\mathbb{N}}
\newcommand{\INTEGERS}{\mathbb{Z}}
\newcommand{\RATIONALS}{\mathbb{Q}}
\newcommand{\REALS}{\mathbb{R}}
\newcommand{\COMPLEX}{\mathbb{C}}
\newcommand{\TORUS}{\mathbb{T}}

\newcommand{\inv}{^{-1}}

\DeclareMathOperator{\coker}{coker}

\DeclareMathOperator{\im}{im}

\DeclareMathOperator{\coequalizer}{coeq}

\newcommand{\ketbra}[2]{\mathord{|#1\rangle\!\langle #2|}}

\DeclareMathOperator{\supp}{supp}

\DeclareMathOperator*{\CLOSURE}{closure}
\DeclareMathOperator*{\CLOSESPAN}{\overline{span}}

\DeclareMathOperator{\ver}{ver}
\DeclareMathOperator{\edges}{edges}

\DeclareMathOperator{\her}{her}

\DeclareMathOperator{\source}{source}
\newcommand{\SPACE}{\mathrm{space}}
\newcommand{\vertices}{\mathrm{vertices}}
\newcommand{\hereditary}{\mathrm{hereditary}}
\newcommand{\regular}{\mathrm{regular}}
\newcommand{\singular}{\mathrm{singular}}
\newcommand{\sources}{\mathrm{sources}}
\newcommand{\finRECEIVERS}{\mathrm{fin~receivers}}
\DeclareMathOperator{\range}{range}

\DeclareMathOperator{\reg}{reg}
\DeclareMathOperator{\cov}{cov}

\DeclareMathOperator{\pos}{pos}

\newcommand{\ILEQ}{\trianglelefteq}

\newcommand{\act}{\curvearrowright}

\newcommand{\DISK}{\mathbb{D}}

\newcommand{\ADJOINTS}{\mathcal{L}}
\newcommand{\COMPACTS}{\mathbb{K}}

\newcommand{\FOCK}{\mathcal{F}}
\newcommand{\TOEPLITZ}{\mathcal{T}}
\newcommand{\PIMSNER}{\mathcal{O}}

\newcommand{\CONTINUOUS}{C}

\newcommand{\tensor}{\mathbin{\otimes}}

\usepackage{setspace}
\hypersetup{colorlinks,citecolor=blue}
\setuptodonotes{color=blue!10,bordercolor=blue!10}

\tikzset{every loop/.style={ distance=1cm, out=120, in=60, -> }}
\tikzcdset{row sep/normal=1cm}
\tikzcdset{column sep/normal=1cm}
\tikzcdset{row sep/small=0.5cm}
\tikzcdset{column sep/small=0.5cm}
\tikzcdset{row sep/large=1.5cm}
\tikzcdset{column sep/large=1.5cm}
\tikzset{help lines/.style={very thin, color=lightgray, dashed}}
\tikzset{axis lines/.style={very thin, color=lightgray}}

\makeatletter
\def\mathcolor#1#{\@mathcolor{#1}}
\def\@mathcolor#1#2#3{%
  \protect\leavevmode
  \begingroup
    \color#1{#2}#3%
  \endgroup}
\makeatother

\let\originalleft\left
\let\originalright\right
\renewcommand{\left}{\mathopen{}\mathclose\bgroup\originalleft}
\renewcommand{\right}{\aftergroup\egroup\originalright}

\newtheoremstyle{MYBREAK}
  {}          
  {}          
  {\itshape}  
  {}          
  {\bfseries} 
  {:}         
  {\newline}  
  {}          
\newtheoremstyle{MYPLAIN}
  {\topsep}   
  {\topsep}   
  {\itshape}  
  {15pt}          
  {\bfseries} 
  {}         
  {5pt plus 1pt minus 1pt} 
  {}          

\theoremstyle{definition}
\newtheorem{DEF}{Definition}[section]

\theoremstyle{plain}
\newtheorem{PROP}[DEF]{Proposition}
\newtheorem{LMM}[DEF]{Lemma}
\newtheorem{THM}[DEF]{Theorem}
\newtheorem{COR}[DEF]{Corollary}

\newtheorem{EXM}[DEF]{Example}

\theoremstyle{MYBREAK}
\newtheorem{BREAKPROP}[DEF]{Proposition}
\newtheorem{BREAKLMM}[DEF]{Lemma}
\newtheorem{BREAKTHM}[DEF]{Theorem}
\newtheorem{BREAKCOR}[DEF]{Corollary}
\newtheorem{BREAKREM}[DEF]{Remark}
\newtheorem{BREAKEXM}[DEF]{Example}

\title{Relative Cuntz--Pimsner algebras:\\Classification of gauge-invariant ideals:\\a simple and complete picture}
\author{Alexander Frei}
\date{\today}

\begin{document}
\maketitle

We give a simple and complete picture on the classification of relative \mbox{Cuntz--Pimsner algebras}
(and so also of gauge--equivariant representations)
using their intuitive \textbf{parametrisation by kernel--covariance pairs.}

For this we first present a \textbf{classification of kernel and cokernel} morphisms (in the general category of correspondences)
which builds on the concept of invariant ideals as realised and coined by Kajiwara--Pinzari--Watatani (and implicitely by Pimsner).
The existence of all such kernel and cokernel morphisms then enable us to reduce the general classification problem to the faithful case of correspondences within ambient operator algebras.

The second component arises from an \textbf{observation made by Katsura}:\linebreak
We unravel Katsura's observation as an obstruction on the range of covariance ideals for correspondences embedded in ambient operator algebras, which comprises the second component of kernel--covariance pairs:
As such our parametrisation runs over \textbf{invariant ideals} (as a discrete range of kernel ideals)
and on the other hand over \textbf{bounded ideals} below some maximal covariance (as an upper bound on the range of covariance ideals).

We then illustrate the \textbf{lattice of relative Cuntz--Pimsner algebras} (and so also of every gauge--equivariant representation)
along the range of kernel--covariance pairs.
Following, we provide the general version of the gauge-invariant uniqueness theorem by its reduction to the faithful case,\linebreak
for which we further recall a \textbf{simplified proof by Evgenios Kakariadis.}

This establishes the first half in our classification:
Every gauge--equivariant representation arises as a relative Cuntz--Pimsner algebra (for its own kernel--covariance pair) and whence the class of gauge--equivariant representations coincides with the class of relative Cuntz--Pimsner algebras. As such the kernel--covariance pairs \textbf{exhaust the gauge--equivariant representations.}

For the second half in our classification we aim to uniquely determine the relative \mbox{Cuntz--Pimsner algebras} by their parametrising kernel--covariance pairs:
More precisely, we will recover every abstract kernel--covariance pair as the actual kernel and covariance from its relative Cuntz--Pimsner algebra, and as such \textbf{our kernel--covariance pairs are also classifying.}

With our found classification by kernel--covariance pairs we then further investigate the lattice structure of gauge--equivariant representations.\linebreak
In particular, we elaborate the existence of \textbf{connecting morphisms} between cokernel strands (given by the kernel component of kernel--covariance pairs) and illustrate our results on examples of graph algebras.

Along this discussion we further clarify Katsura's description (using T-pairs) as a simple translation of kernel--covariance pairs
(which had been already covered by Katsura himself)
with the second component however not taken and further pursued as describing the range of covariance ideals.

Altogether \textbf{our classification is a simple reduction} of the gauge--invariant uniqueness theorem (along cokernel morphisms) together with the identification of kernel and covariance for relative Cuntz--Pimsner algebras.

Finally we provide a realisation of relative Cuntz--Pimsner algebras as absolute Cuntz--Pimsner algebra by maximal dilation and reveal Katsura's construction as the canonical dilation given by the maximal covariance.\linebreak
We further reveal the canonical dilation as a familiar construction from graph algebras and provide an example to illustrate \textbf{the lack of minimal dilations.}

\vspace{0.5\baselineskip}

As an application we provide a systematic approach for an earlier pullback result by Robertson--Szymanski which we extend to the general context in upcoming work \textbf{with Mariusz Tobolski and Piotr~M.~Hajac.}

\section{Correspondences}
\label{sec:CORRESPONDENCES}

We begin with an introduction to correspondences and their representations.
In particular, we provide a less formal and more abstract angle.
This seeks to help to understand their gauge-equivariant representations from an abstract perpsective --- and so also to classify the entire lattice of gauge-invariant ideals.
Further, this allows one to better understand dilations and in further work the shift equivalence problem from an abstract angle.
A Hilbert module is a right module over an operator algebra (seen as coefficient algebra) that comes equipped with a pairing (compatible with the coefficient algebra) which renders the right module complete with respect to the induced norm:
\[
  \braket{\BLANK|\BLANK}: X\times X\to A:\TAB\|x\|^2:=\|\braket{x|x}\|.
\]
Given a pair of Hilbert modules over a common coefficient algebra one may introduce the notion of adjointable operators as those which admit an adjoint:
\[
  T:X\to Y,\TAB T^*:Y\to X:\TAB\braket{T\BLANK|\BLANK}=\braket{\BLANK|T^*\BLANK}.
\]
We note that such adjointable operators are automatically continuous which may be seen most easily via the closed graph theorem.
Moreover the class of adjointable operators over pairs of Hilbert modules defines a \enquote{consistent system} of Banach spaces with composition and involution
\[
  \ADJOINTS(Y|Z)\circ\ADJOINTS(X|Y) \subset\ADJOINTS(X|Z),\TAB\ADJOINTS(X|Y)^*=\ADJOINTS(Y|X)
\]
satisfying the generalized \CSTAR-identity
\[
  T\in\ADJOINTS(X|Y):\TAB\|T\|^2=\|T^*T\|.
\]
The vertical separators hereby seek to convey the \enquote{Dirac bra--ket notation},
that is we have the following identification which we term Dirac calculus:
\begin{PROP}\label{DIRAC-CALCULUS}
  The identification (and its conjugate)
  \begin{gather*}
    X=\ket X\subset\ADJOINTS(A|X),\TAB X^*=\bra X\subset\ADJOINTS(X|A):\\[2\jot]
    \ket{x}a:= \ket{xa},\TAB\bra{x}y:= \braket{x|y}:\TAB \bra{x}^*=\ket{x}
  \end{gather*}
  define an isometric embedding (and its conjugate),\\
  which renders the pairing as
  \[
    \braket{x|y}=\bra{x}\circ\ket{y}\in\ADJOINTS(X|A)\circ\ADJOINTS(A|X).
  \]
  This further renders the notion of compact operators as
  \[
    \COMPACTS(X,Y):=\CLOSESPAN\ket{Y}\bra{X}\subset\ADJOINTS(X|Y)
  \]
  and the notion of adjointable operators as
  \[
    Tx=T\circ\ket{x},\TAB\braket{Tx|y}=\bra{x}\circ T^*\circ\ket{y}=\braket{x|T^*y}.
  \]
  All of the above enables one to split expressions as composition of operators.
\end{PROP}
\begin{proof}
  We first note that the coefficient algebra itself defines a Hilbert module:
  \[
    \braket{\BLANK|\BLANK}:A\times A\to A:\TAB\braket{x|y}=x^*y.
  \]
  One may now verify that the assignments define mutual adjoints
  \[
    \bra{x}^*=\ket{x}\in\ADJOINTS(A|X):\TAB
    \big\langle\braket{x|y}\big|a\big\rangle = \ldots = \big\langle{y}\big|xa\big\rangle = \big\langle{y}\big|\ket{x}a\big\rangle
  \]
  and that the identification defines an isometric embedding
  \[
    X=\ket{X}\subset\ADJOINTS(A|X):\TAB
    \|\ket{x}\|^2=\|\bra{x}\circ\ket{x}\|=\ldots=\|\braket{x|x}\|.
  \]
  We leave these as an instructive exercise for the reader.
\end{proof}
With this identification at hand, a Hilbert module conveniently reads as nothing but a right module together with an abstract pairing given by involution
\[
  (X\curvearrowleft A)\TAB XA\subset X,\TAB\TAB(X\times X\to A)\TAB X^*X\subset A
\]
which we from now on simply indicate as such formal inclusions.\\
We may now turn our attention to the notion of \CSTAR-correspondence:
Formally these are given as Hilbert modules together with a representation of the coefficient algebra as adjointable operators.
With the viewpoint from above we however obtain the alternative description as a bimodule over the coefficient algebra together with some compatible inner product pairing
\begin{equation}
  \label{CORRESPONDENCE=BIMODULE}
  AX\subset X,\TAB\TAB XA\subset X,\TAB\TAB X^*X\subset A
\end{equation}
satisfying some relations such as (now evident in Dirac formalism)
\[
  (ax)^*y=x^*(a^*y),\TAB\TAB x^*(ya) = (x^*y)a,\TAB\TAB\mathrm{etc.}
\]
We meanwhile note that the notion of correspondences has an intrinsic asymmetry by the pairing.
More precisely, the more traditional notion of Hilbert modules comes equipped with a dual pairing which renders the notion symmetric:
\begin{gather*}
  X^*X\subset A,\TAB\TAB XA\subset X,\\
  XX^*\subset A,\TAB\TAB AX\subset X.
\end{gather*}
In fact, this covers the main objective of covariant representations. We return to this aspect in the following section.
We continue with a swift introduction to the (internal) tensor product of correspondences.
Using our Dirac formalism from above we may now simply introduce those as formal powers such as
\[
  A(XY\ldots Z)\subset (AX)Y\ldots Z,\TAB\TAB (XY\ldots Z)A\subset XY\ldots(ZA)
\]
and where the inner product pairing now simply reads
\begin{gather*}
  (XY\ldots Z)^*(XY\ldots Z)
  = Z^*\ldots Y^*\Big(X^*X\subset A\Big)Y\ldots Z\\
  \subset Z^*\ldots\Big(Y^*Y\subset B\Big)\ldots Z\subset\ldots\subset\Big(Z^*Z\subset A\Big).
\end{gather*}
Moreover, this automatically entails the balanced relation as for example
\begin{gather*}
  \Big((xa)y-x(ay)\Big)^*\Big((xa)y-x(ay)\Big) = \\
  = y^*(a^*x^*)(xa)y - \ldots + (y^*a^*)(x^*x)(ay) = 0.
\end{gather*}

Let us give an illustrative example for such tensor products:
\begin{BREAKEXM}[Graph correspondences: Tensor powers]
  \label{GRAPH-CORRESPONDENCES:TENSOR-POWERS}
  Consider a directed graph and regard its graph correspondence
  \[
    X=\ell^2\Bigl(~E=\edges~\Bigr),\TAB A=c_0\Bigl(~V=\mathrm{vertices}~\Bigr)
  \]
  with action and pairing given by range and source say
  \[
    \begin{gathered}
    \begin{tikzcd}[row sep=huge, column sep=huge]
      a\rar{z} & b\arrow[loop,"w"]
    \end{tikzcd}
    \end{gathered}:\TAB\TAB
    \begin{gathered}
      z^*z=a,\TAB za=z=bz,\\
      z^*z=b,\TAB wb=w=bw.
    \end{gathered}
  \]
  On the other hand we know by \cite{KALISZEWSKI-PATANI-QUIGG} that every nondegenerate correspondence over a \enquote{direct sum over a discrete set as vertices} arises as a graph correspondence:
  \[
    A=c_0\Big(\mathrm{vertices}=\mathrm{some~set}\Big)\implies X=\ell^2(\edges).
  \]
  Indeed this follows by some simple counting argument:
  \[
    a,b\in\mathrm{vertices}:\TAB bXa = \ell^2\Big(\edges:b\from a\Big)=\ell^2(bEa).
  \]
  From this argument it furthermore follows that degenerate correspondences arise the same way when allowing for edges with heads pointing into the void such as
  \begin{TIKZCD}[row sep=1cm/4]
    \phantom{\emptyset} &[-0.5cm]&&[-0.5cm] \phantom{\emptyset}\\
    & a\rar{z}\ular\dlar & b\arrow[loop, out=120, in=60, distance=1cm,"w"]\drar\\
    \phantom{\emptyset}&&&\phantom{\emptyset}
  \end{TIKZCD}
  and furthermore also any power of a graph correspondence arises as a graph correspondence itself and indeed is given by its paths of according length:
  \[
    XX\ldots X=\ell^2(EE\ldots E),\TAB A=c_0(\mathrm{vertices}).
  \]
  In other words, that is by concatenation of edges.
\end{BREAKEXM}

As such the tensor product may be seen as nothing but a formal power in words and one may equivalently consider also mixed powers with the dual space and its dual pairing (though formally only in the context of operator spaces).

We finish this section with an intrinsic characterization of Hilbert bimodules.
For this we first require the following well-known observation:

\begin{PROP}\label{MAXIMAL-FAITHFUL-ACTION}
  Consider a correspondence (in Dirac braket notation)
  \[
    \braket{X|X}\subset A,\TAB XA\subset X,\TAB AX\subset A
  \]
  and regard the orthogonal complement for the kernel (by left action),
  \begin{gather*}
    \ker(A\act X)=\{aX=0\}\subset A:\\[2\jot]
    \ker(A\act X)^\perp = \{a\ker(A\act X)=0\}\subset A.
  \end{gather*}
  This defines the largest ideal that renders the left action faithful.\\
  In particular one may simultaneously identify the orthogonal complement with its isometric image (by left action) within the space of adjointable operators.
\end{PROP}
\begin{proof}
  Before we begin with the actual argument let us make the useful observation that the kernel defines an ideal (closed and two-sided) since
  \begin{gather*}
    A\ker(A\act X)X=0,\TAB \ker(A\act X)AX=0,\\[2\jot]
    \overline{\ker(A\act X)}X\subset\overline{\ker(A\act X)X}=0
  \end{gather*}
  and so also a selfadjoint one (which one may also easily observe by hand)
  \[
    \Big\langle X\Big|\ker(A\act X)^*X\Big\rangle = \Big\langle \ker(A\act X)X\Big|X\Big\rangle=0\implies \ker(A\act X)^*X=0.
  \]
  This kernel determines now precisely those ideals that render the action faitfhul:
  \[
    K\ker(A\act X)=K\cap\ker(A\act X)=0\iff K\restriction X=0.
  \]
  As such its orthogonal complement defines the largest such ideal:
  \[
    K\ker(A\act X) = 0\iff K\subset\ker(A\act X)^\perp.
  \]
  Note that the orthogonal complement (as we have defined from the left only) defines itself an ideal (closed and two-sided)
  \begin{gather*}
    A\ker(\ldots)^\perp\ker(\ldots)=0,\TAB \ker(\ldots)^\perp A\ker(\ldots)=0,\\[2\jot]
    \overline{\ker(\ldots)^\perp}\ker(\ldots)\subset \overline{\ker(\ldots)^\perp\ker(\ldots)}0
  \end{gather*}
  and so also a selfadjoint one (which this time is not too obvious).
\end{proof}

Building on this, we may now further give an \textbf{intrinsic characterization} of Hilbert bimodules
(which we promote as theorem since it's widely non-standard even until now) and the author would like to thank \textbf{Adam Skalski and Bartosz Kwasniewski} for helping with the key detail for this idea:

\begin{THM}\label{HILBERT-BIMODULES}
  Consider a correspondence as seen in \eqref{CORRESPONDENCE=BIMODULE}\\
  and regard the orthogonal complement from Proposition \ref{MAXIMAL-FAITHFUL-ACTION}
  \[
    \ker(A\act X)^\perp = \{a\ker(A\act X)=0\}\subset A
  \]
  and simultaneously identified as in that proposition by their action
  \[
    A\supset\ker(A\act X)^\perp=\Bigl(~\ker(A\act X)^\perp\act X~\Bigr)\subset \ADJOINTS(X|X).
  \]
  Then the correspondence defines a Hilbert bimodule if and only if the compact operators (see Proposition \ref{DIRAC-CALCULUS}) all lie within the orthogonal complement
  \[
    \ketbra{X}{X}\subset
    \ker(A\act X)^\perp\subset\ADJOINTS(X|X)
  \]
  while the dual pairing is always given by
  \[
    \ketbra{\BLANK}{\BLANK}:X\times X\to A:\TAB \ketbra{X}{X}\subset\ker(A\act X)^\perp\subset A
  \]
  whence there is no ambiguity left anymore.
  As such the notion of Hilbert bimodules \textbf{defines a pure property} without \textbf{any additional structure.}
\end{THM}

Before we begin with the proof, we note that the situation above may be easiest pictured in mind (and remembered) by the following illustration:
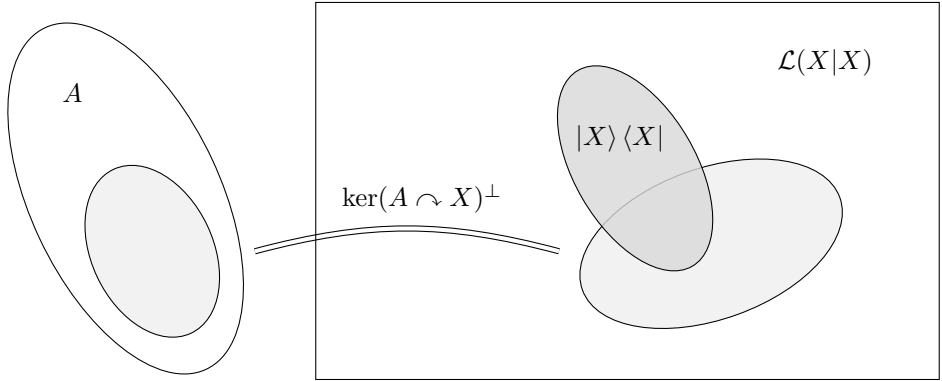
\begin{center}
  \hspace{-0.4cm}
  \begin{tikzpicture}[scale=1]
    \draw (-2.7,2.4) ellipse [x radius=5/2,y radius=5/4, rotate=-65];
    \draw[fill=gray!10] (-2.35,1.7) ellipse [x radius=1.2,y radius=0.8, rotate=-65];
    \draw[double equal sign distance] (-1,1.7) to[bend left=15] (3,1.7);
    \draw (-0.2,0) rectangle (8,5);
    \draw[fill=gray!10] (5,1.8) ellipse [x radius=1.8,y radius=1, rotate=20];
    \path[fill=gray!30,opacity=0.8] (4,2.8) ellipse [x radius=1.5,y radius=0.8, rotate=120];
    \draw (4,2.8) ellipse [x radius=1.5,y radius=0.8, rotate=120];
    \node at (-3.4,3.8) {$A$};
    \node at (1.2,2.4) {$\ker(A\act X)^\perp$};
    \node at (6.5,4.2) {$\ADJOINTS(X|X)$};
    \node at (3.8,3.2) {$\ket{X}\bra{X}$};
  \end{tikzpicture}
\end{center}
As such a correspondence in general may be seen as a Hilbert bimodule with a partial dual pairing.
With this picture in mind let us get to the proof.

\begin{proof}
  Instead of Hilbert bimodules, it suffices to consider the case of Hilbert modules (meaning the case of a single action and pairing) say given by
  \[
    \braket{X|X}\subset A,\TAB XA\subset X.
  \]
  The kernel for the single action (in this case from the right) agrees with the orthgonal complement for the pairing (and so its support ideal):
  \[
    \ker(X\curvearrowleft A) = \braket{X|X}^\perp \left(= \supp(X)^\perp\right).
  \]
  Indeed using Blanchard factorization \cite[Lemma~1.3]{BLANCHARD} this immediately follows from
  \[
    \braket{X|X}a=0\iff\Big(\ket{X}=\ket{X}\braket{X|X}\Big)a=0.
  \]
  In turn this observation tells us where to search for the pairing:
  \[
    \braket{X|X}\subset\braket{X|X}^{\perp\perp}=\ker(X\curvearrowleft A)^\perp
  \]
  Replacing the right action from our consideration above by the left action from the proposition we thus just revealed the condition from the proposition.
  Conversely, given the condition from the proposition we may simply retrieve the dual pairing using the isometric image of the orthogonal complement.
  Using Dirac calculus from Proposition \ref{DIRAC-CALCULUS} we finally obtain (simply as composition)
  \[
    \ket{x}\circ\Big(\bra{y}\circ\ket{z}\Big) = \Big(\ket{x}\circ\bra{y}\Big)\circ\ket{z}
  \]
  and so the traditional compatibility (between pairings) holds trivially.
\end{proof}

\section{Representations}
\label{sec:REPRESENTATIONS}

Consider an abstract correspondence as introduced in the first section:
\[
  X^*X\subset A,\TAB XA\subset X,\TAB AX\subset X.
\]
This structure is in some sense freely floating,
and so we wish to embed this structure as a whole into an ambient operator algebra as illustrated:
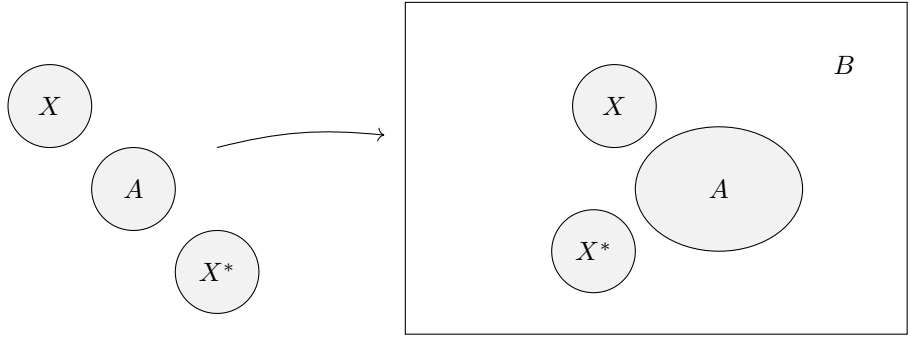
\begin{center}
\begin{tikzpicture}[scale=0.55]
    \draw[fill=gray!10] (1,2) circle (1) node {$X$};
    \draw[fill=gray!10] (3,0) circle (1) node {$A$};
    \draw[fill=gray!10] (5,-2) circle (1) node {$X^*$};
    \draw (5,1) edge[bend left =10,->] (9,1.3);
    \draw (9.5,-3.5) rectangle (21.5,4.5);
    \node at (20,3) {$B$};
    \draw[fill=gray!10] (14.5,2) circle (1) node {$X$};
    \draw[fill=gray!10] (17,0) ellipse (2 and 1.5) node {$A$};
    \draw[fill=gray!10] (14,-1.5) circle (1) node {$X^*$};
\end{tikzpicture}
\end{center}
A good analogy here is the embedding of Fell bundles into any crossed product.
This is what we understand as a representation. More precisely, that is a representation of both the correspondence and the coefficient algebra into some ambient operator algebra say
\[
  \begin{tikzcd}[column sep=1cm]
    (X,A)\rar& B
  \end{tikzcd}:\TAB\TAB\tau: X\to B,\TAB\TAB\phi:A\to B
\]
so the former being a morphism of vector spaces and the latter a morphism of operator algebras
and such that the pair is coherent with the structure in between, which now reads in Dirac formalism (see Proposition \ref{DIRAC-CALCULUS}):
\[
  \phi(x^*y)=\tau(x)^*\tau(y),\TAB\TAB\tau(xa)=\tau(x)\phi(a),\TAB\TAB\tau(ax)=\phi(a)\tau(x).
\]
It is well-known that the latter follows automatically from the former two:
Indeed using the \CSTAR-identity we find (written in Dirac formalism)
\begin{gather*}
  \Big[\phi(a)\tau(x)-\tau(ax)\Big]^*\Big[\phi(a)\tau(x)-\tau(ax)\Big] = \\[\jot]
  = \phi(a)^*\tau(x)^*\tau(x)\phi(a) - \ldots + \tau(ax)^*\tau(ax)\\[\jot]
  = \phi(x^*a^*ax) - \ldots + \phi(x^*a^*ax)=0.
\end{gather*}
Consider now the structure as embedded within the ambient operator algebra
\[
  X=\tau (X)\subset B,\TAB A=\phi(A)\subset B
\]
then we may also view those as subspace and subalgebra which renders the abstract inclusions from the previous section into actual inclusions
\[
  X^*X\subset A\subset B,\TAB XA\subset X\subset B,\TAB AX\subset X\subset B.
\]
We note that while these representations are evidently not faitful in general, one may always pass to its quotient correspondence which renders the representation faithful. This will define the first parameter for the classification.

We now wish to extend a representation to higher and mixed tensor powers:
For this we recall the following result by Kajiwara--Pinzari--Watatani:
\begin{PROP}[{\cite[Lemma 2.1]{KAJIWARA-PINZARI-WATATANI}}]\label{HIGHER-MIXED}
  Representations canonically extend to the tensor product and compact operators (denoted in Dirac formalism):
  \begin{gather*}
    \left(\TAB
    \begin{gathered}
      \tau:X\to Y\\
      \phi:A\to B
    \end{gathered}
    \TAB\right)
    \implies
    \left(\TAB
    \begin{gathered}
      XX\to B:\\
      XX^*\to B:
    \end{gathered}\TAB\TAB
    \begin{aligned}
      \tau(xy)&:=\tau(x)\tau(y)\\
      \tau(xy^*)&:=\tau(x)\tau(y)^*
    \end{aligned}
    \TAB\right)
  \end{gather*}
  The latter further satisfies the relations
  \begin{gather*}
    \begin{gathered}
      KL\in (XX^*)(XX^*)\subset(XX^*):\\[2\jot]
      aKb\in A(XX^*)A\subset(XX^*):
    \end{gathered}\TAB\TAB
    \begin{gathered}
      \tau(K)\tau(L)=\tau(KL)\\[\jot]
      \phi(a)\tau(K)\phi(b)=\tau(aKb)
    \end{gathered}
  \end{gather*}
  and so defines in particular a morphism of operator algebras.\\
  Furthermore, suppose the morphism is isometric on the coefficient algebra\\
  then so also on the correspondence and on compact operators:
  \begin{align*}
    \|(A\to B)\BLANK\|=\|\BLANK\|\TAB&\implies\TAB\|(X\to B)\BLANK\|=\|\BLANK\|,\\
    \|(A\to B)\BLANK\|=\|\BLANK\|\TAB&\implies\TAB\|(XX^*\to B)\BLANK\|=\|\BLANK\|.
  \end{align*}
  By iteration the proposition includes all higher and any other mixed powers.
\end{PROP}
\begin{proof}
  We provide the extension to compact operators since it demonstrates the use of Dirac calculus with a \textbf{neat trick by Kajiwara--Pinzari--Watatani:}
  We need to verify that the formal linear assignment on elementary compact operators remains bounded
  \[
    \hspace{2cm}\left\|\tau\left(\sum_nx_ny_n^*\right)\right\|\leq\left\|\sum_nx_ny_n^*\right\|,\TAB\TAB\forall x_n,y_n\in X,
  \]
  whence the assignment allows an (a posteriori well-defined) extension to the completion of all compact operators.
  For this we invoke matrix calculus to reformulate the linear sum as a product of matrices
  \[
     (x)(y^*):=\sum_nx_ny_n^*=\MATRIX{x_1&\cdots&x_N}\MATRIX{y_1^*\\\vdots\\y_N^*}.
  \]
  Reformulated, the quite clever trick by Kajiwara Pinzari and Watatani is now to use the \CSTAR-identity (generalized to matrices of adjointable operators):
  \begin{gather*}
    \|xy^*\|^2 = \Big\| (xy^*)(yx^*)=x(y^*y)x\Big\| = \Big\|x\sqrt{y^*y}\Big\|=\ldots= \Big\|\sqrt{x^*x}\sqrt{y^*y}\Big\|.
  \end{gather*}
  Note that this automatically invokes matrix inflations since for example
  \[
    x^*x = \MATRIX{\bra{x_1}\\\vdots\\\bra{x_N}}\MATRIX{\ket{x_1}&\cdots&\ket{x_N}} =\MATRIX{\braket{x_1|x_1} & \cdots & \braket{x_1|x_N}\\\vdots&\ddots&\vdots\\\braket{x_N|x_1}&\cdots&\braket{x_N|x_N}}.
  \]
  On the other hand, this equally applies when invoking the representation and so we obtain the desired bound
  \[
    \|\tau(xy^*)\| = \Big\|\phi\Big(\sqrt{x^*x}\sqrt{y^*y}\Big)\Big\|
    \leq \Big\|\sqrt{x^*x}\sqrt{y^*y}\Big\|=\|xy^*\|.
  \]
  In turn this trick also applies for the correspondence and tensor product
  \begin{gather*}
    X\to B:\TAB\|\tau(x)\| = \|\phi(x^*x)\|\leq \|x^*x\|=\|x\|,\\[2\jot]
    XX\to B:\TAB\|\tau(xy)\| = \|\phi(y^*x^*xy)\|\leq \|y^*x^*xy\|=\|xy^*\|
  \end{gather*}
  as well as on every other type of tensor product such as of operator algebras.\\
  Meanwhile, this moreover infers that isometricity passes from the coefficient algebra to the correspondence (and tensor product) and to compact operators.
  Moreover, the morphism satisfies the relations on compact operators since one easily verifies on elementary compact operators:
  \begin{gather*}
    \phi(a)\tau(xy^*)\phi(b)=\phi(a)\tau(x)\tau(y^*)\phi(b)=\tau(axy^*b), \\[\jot]
    \tau(xy^*)\tau(zw^*)=\tau(x)\phi(y^*z)\tau(w^*)=\tau(xy^*zw^*).
  \end{gather*}
  We therefore found the desired relations and the proof is complete.
\end{proof}

With this at hand we may now introduce the notion of covariances.
We begin for this with the observation that (as in the previous proposition) we could have also defined the morphism (somewhat senseless)
\[
  A\supset\braket{X|X}\to B:\TAB\tau(x^*y):=\tau(x)^*\tau(y)
\]
which agrees with the oringal one on the coefficient algebra (somewhat trivially):
\begin{TIKZCD}[column sep=5cm]
  \braket{X|X}\dar\drar[phantom]{\tau(\BLANK)=\phi(\BLANK)} \rar & B\dar[equal] \\
  A\rar & B.
\end{TIKZCD}
Suppose on the other hand that some elements also act as compact operators:
\[
  A\cap\ketbra{X}{X}:=\Big\{~a\in A~\Big|~ a\in \ketbra{X}{X}~\Big\}\subset A.
\]
Then there is no evidence to believe that the induced morphism from the previous proposition would agree with the morphism for the algebra:
\begin{TIKZCD}[column sep=5cm]
  A\cap\ketbra{X}{X}\dar\rar\drar[phantom]{\tau(\BLANK)\neq\phi(\BLANK)} & B\dar[equal] \\
  \ketbra{X}{X}\rar & B.
\end{TIKZCD}
Nevertheless the difference of morphisms surprisingly defines a morphism of operator algebras since (using the property from Proposition \ref{HIGHER-MIXED})
\begin{align*}
  (\phi-\tau)a(\phi-\tau)b
  &= \phi(a)\phi(b)-\phi(a)\tau(b)-\tau(a)\phi(b) + \tau(a)\tau(b) \\
  &= \phi(ab)-\tau(ab)-\tau(ab) + \tau(ab) =  \phi(ab) - \tau(ab).
\end{align*}
and so our representation decomposes into the sum of morphisms
\begin{align}
  &\left(\begin{tikzcd}[column sep=1cm]
    A\cap\ketbra{X}{X}\rar&B
  \end{tikzcd}\right)= \label{DIFFERENCE-MORPHISM}\\[2\jot]
  &= \left(\begin{tikzcd}[column sep=0.2cm,row sep=0.4cm]
    A\cap\ketbra{X}{X}\dar \\
    \ketbra{X}{X}\rar& B
  \end{tikzcd}\right)\: +\:
  \left(\begin{tikzcd}[column sep=0.5cm,row sep=0.4cm]
    A\cap\ketbra{X}{X}\rar &B\dar[equal] \\
    & B
  \end{tikzcd}~-~
  \begin{tikzcd}[column sep=0.2cm,row sep=0.4cm]
    A\cap\ketbra{X}{X}\dar \\
    \ketbra{X}{X}\rar& B
  \end{tikzcd}\right).\notag
\end{align}
As such we may capture the domain of equality by the covariance ideal:
\begingroup
\setlength{\abovedisplayskip}{\baselineskip}
\setlength{\belowdisplayskip}{\baselineskip}
\begin{equation}\label{COVARIANCE:DEFINITION}
  \cov\Big(\begin{tikzcd}[column sep=0.3cm]
    (X,A)\to B
  \end{tikzcd}\Big):=
  \ker\bigg(~\rule[1ex/2]{2cm}{0.1pt}~''~\rule[1ex/2]{2cm}{0.1pt}~\bigg)\ILEQ A.
\end{equation}
\endgroup
We will often drop the dependency on the coefficient algebra for simplicity.\\
The covariance and kernel will classify the gauge-equivariant representations.
Let us thus take a closer look at kernel morphisms and possible covariances:

\section{Kernel and Covariance}
\label{sec:COVARIANCES}

We begin this section with a characterization of kernel morphisms in the category of correspondences.
For this let us introduce the general notion of morphisms between correspondences.
As for representations, these are given by a pair of morphisms on the correspondence and the coefficient algebra
\[
  \begin{tikzcd}[column sep=1cm]
    (X,A)\rar& (Y,B)
  \end{tikzcd}:\TAB\TAB\tau:X\to Y,\TAB\TAB \phi:A\to B
\]
where the former defines a linear morphism and the latter a morphism of operator algebras and such that the pair is coherent with the structure in between, which conveniently reads in Dirac formalism:
\[
  \phi(x^*y)=\tau(x)^*\tau(y),\TAB\TAB \tau(ax)=\phi(a)\tau(x),\TAB\TAB \tau(xa)=\tau(x)\phi(a).
\]
While the resulting category fails to be abelian (basically due to the algebraic morphism on the coefficient algebra) it still possesses all kernels and cokernels.
As the notion of kernels and cokernels is nonstandard however in categories beyond abelian ones, let us give a quick introduction.
Our catgory of correspondences and their morphisms has zero morphisms in the sense:
\begin{gather*}
  \begin{tikzcd}[column sep=1cm]
    (X,A)\rar & (Y,B)\rar{0} & (Z,C)
  \end{tikzcd}
  =
  \begin{tikzcd}[column sep=1cm]
    (X,A)\rar{0} & (Z,C),
  \end{tikzcd}\\
  \begin{tikzcd}[column sep=1cm]
    (X,A)\rar{0} & (Y,B)\rar & (Z,C)
  \end{tikzcd}
  =
  \begin{tikzcd}[column sep=1cm]
    (X,A)\rar{0} & (Z,C).
  \end{tikzcd}
\end{gather*}
A kernel for a morphism is the universal annihilating morphism:
\[
  \begin{tikzcd}[column sep=1cm]
    (?,?)\rar{\ker} & (X,A)\rar & (Y,B)
  \end{tikzcd}
  =0
\]
That is any other annihilating morphism factors uniquely over the kernel:
\begin{gather*}
  \begin{tikzcd}[column sep=1cm]
    (X',A')\rar & (X,A)\rar & (Y,B)
  \end{tikzcd}=0\\[\jot]
  \implies\TAB
  \begin{tikzcd}[column sep=1cm]
    (X',A')\rar[dashed]{\exists!} & (?,?)\rar{\ker} & (X,A).
  \end{tikzcd}
\end{gather*}
Dually one may define the cokernel of morphisms. Moreover, we define a short exact sequence denoted by
\begin{TIKZCD}[column sep=1cm]
  0\rar& (?,?)\rar& (X,A)\rar& (?,?) \rar&0
\end{TIKZCD}
whenever each side is the kernel respectively cokernel of the other:
\begin{gather*}
  \begin{tikzcd}[column sep=1cm]
    (?,?)\rar & (X,A)
  \end{tikzcd}=\ker\Big(
  \begin{tikzcd}[column sep=0.5cm]
    (X,A)\rar & (?,?)
  \end{tikzcd}\Big)\\[2\jot]
  \begin{tikzcd}[column sep=1cm]
    (X,A)\rar & (?,?)
  \end{tikzcd}=\coker\Big(
  \begin{tikzcd}[column sep=0.5cm]
    (?,?)\rar & (X,A)
  \end{tikzcd}\Big)
\end{gather*}
We will fill the question marks in the proposition below.\\
But before we note the following equivalent notions of invariant and hereditary ideals
which date all the way back to Pimsner and as explicitely coined by Kajiwara--Pinzari--Watatani
(we refrain from the notion of negatively invariant ideals as we will find those from a different perspective later on):

\begin{BREAKLMM}[{\cite[Lemma~3.5 following]{PIMSNER-1997} and \cite[section 4]{KAJIWARA-PINZARI-WATATANI};\\
  see further also \cite[section 2]{FOWLER-MUHLY-RAEBURN} and similarly also \cite[section 1]{KATSURA-GAUGE-IDEALS}}]
  The notion of invariant and hereditary ideals coincide.
  More precisely, one has the characterization for ideals in the coefficient algebra $K\ILEQ A,$
  \begin{equation}\label{INVARIANCE-IMAGE}
    XK=\{x\in X\mid X^*x\in K\}=\{x\in X\mid x^*x\in K\}
  \end{equation}
  and so it furthermore holds the equivalence
  \begin{equation}\label{INVARIANT=HEREDITARY}
    X^*KX\subset K\iff KX\subset XK.
  \end{equation}
  The former is generally refered to as invariant and the latter as hereditary.
\end{BREAKLMM}
\begin{proof}
  Katsura gives a neat proof based on some factorization result by Lance, more precisely \cite[Lemma 4.4]{LANCE},
  which itself however still requires some rather technical approximation.
  Instead we may verify the equivalence with the following fairly elementary observations:
  Let us first note that both of the right-hand spaces are automatically linear and closed by Cohen--Hewitt:
  \[
    KX = \CLOSESPAN KX\subset X,\TAB XK=\CLOSESPAN XK\subset X.
  \]
  Now we clearly have the forward inclusions since
  \[
    X^*(XK)=(X^*X)K\subset AK\subset K.
  \]
  Conversely, we have using any approximate identity for the ideal:
  \[
    x^*x\in K\implies (1-e)x^*x(1-e)\to0\implies x=\lim_e(xe)\in XK.
  \]
  With the characterization at hand we further obtain the forward direction
  \[
    X^*KX\subset K\implies KX\subset XK:\TAB (X^*K^*)(KX)= X^*KX\subset K.
  \]
  Alternatively, one may verify the forward direction using Blanchard factorization (whose original proof is very neat and elementary, see \cite[Lemma 1.3]{BLANCHARD}):
  \[
    KX = (KX)(KX)^*(KX)= (KX)(X^*KX)\subset(KX)K\subset XK
  \]
  This however implicitely invokes the yet technical Cohen--Hewitt (see above).\\
  If one wishes to refrain from using Cohen--Hewitt altogether, one may also argue in the following elementwise way (formulated in Dirac notation):
  \begin{gather*}
    K\ket{X}\ni k\ket{x} = \ket{y}\braket{y|y}\implies \braket{y|y}^3 = \bra{x}k^*k\ket{x}\in K \\[2\jot]
    \implies \braket{y|y} = \sqrt[3]{\bra{x}k^*k\ket{x}}\in K \implies k\ket{x}=\ket{y}\braket{y|y}\in XK.
  \end{gather*}
  All of these variants for the forward direction have their advantage.\\
  For the converse direction we may simply argue
  \[
    KX\subset XK\implies X^*KX\subset X^*XK\subset AK\subset K.
  \]
  So the notion of invariant and hereditary ideals coincide.
\end{proof}

Let us give an example to illuminate the notion of hereditary ideals:
\begin{BREAKEXM}[Graph correspondences: hereditary ideals]
  \label{GRAPH-CORRESPONDENCES:HEREDITARY-IDEALS}
  Consider a graph correspondence as in Example \ref{GRAPH-CORRESPONDENCES:TENSOR-POWERS}:
  \[
    X=\ell^2\Bigl(~E=\edges~\Bigr),\TAB A=c_0\Bigl(~V=\mathrm{vertices}~\Bigr).
  \]
  Then every ideal corresponds to some collection of vertices
  \[
    K=c_0\Big(S=\mathrm{some~vertices}\Big)\ILEQ c_0(V) = A
  \]
  and hereditary ideals become the hereditary collection of vertices
  \[
    c_0(S)\ell^2(E)\subset\ell^2(E)c_0(S)
    \TAB\iff\TAB SE\subset ES
  \]
  which reads written out in words
  \[
    \range(\mathrm{some~edge})\in S\implies \source(\mathrm{some~edge})\in S
  \]
  and whence \textbf{their name: hereditary ideals.}
\end{BREAKEXM}

In order to understand kernel morphisms (and so in turn also cokernel morphisms) we make the following observation:
Suppose a morphism vanishes on the coefficient algebra, then it does so also on the entire correspondence:
\begin{equation}\label{KERNEL-OBSERVATION}
  (\begin{tikzcd}[column sep=0.5cm]
    A\rar & B
  \end{tikzcd})=0
  \TAB\implies\TAB
  (\begin{tikzcd}[column sep=0.5cm]
    X\rar& Y
  \end{tikzcd})=0
\end{equation}
Indeed one easily verifies (in a way using the \CSTAR-identity)
\[
  \tau(x)=0\implies \phi(x^*x)=\tau(x)^*\tau(x)=0\implies x^*x=0\implies x=0
\]
which is equivalently the commutative diagram (see the previous section)
\[
  \begin{tikzcd}[column sep=4cm]
    \braket{X|X}\drar[phantom]{\tau(\BLANK)=\phi(\BLANK)} \rar\dar &0\dar[equal] \\
    A \rar& 0 \subset B.
  \end{tikzcd}
\]
As such one may already expect the kernel of morphisms to involve the kernel on the coefficient algebra in some crucial way.
With this in mind, we may now give a new intrinsic characterization of kernel and cokernel morphism.

Meanwhile the author would like to note that the idea to consider concrete kernel correspondences (by invariant ideals) dates back to \cite{PIMSNER-1997} and their quotient correspondence to \cite{KAJIWARA-PINZARI-WATATANI} with their representations already appearing in the proof of \cite[theorem~4.3]{KAJIWARA-PINZARI-WATATANI} and as explicitely in \cite{FOWLER-MUHLY-RAEBURN}.\\
The author identified those \textbf{as partial results on} the intrinsic characterisation of categorical \textbf{kernel and cokernel morphisms,} which further expanded and completed provide the \textbf{following entire classification} of kernel and cokernel morphisms in the category of correspondences, which we promote as theorem due to its usefulness (also in upcoming work).

\begin{BREAKTHM}[Classification: kernel and cokernel morphisms]
  \label{THEOREM:KERNEL-COKERNEL-MORPHISMS}
  The category of correspondences has all kernel and all cokernel morphisms.\\
  More precisely, they are all of the special form
  \begin{TIKZCD}[column sep=1cm]
    0\rar& (XK,K)\rar& (X,A)\rar& \left(\frac{X}{XK},\frac{A}{K}\right) \rar&0
  \end{TIKZCD}
  precisely for the invariant (and equivalently hereditary) ideals
  \[
    K\ILEQ A:\TAB\TAB X^*KX\subset K\TAB (\iff KX\subset XK).
  \]
  Conversely, given a morphism its kernel is given by
  \[\hspace{-0.5cm}
    \begin{tikzcd}[column sep=0.8cm]
      (XK,K)\rar{\ker} & (X,A)\rar & (Y,B)
    \end{tikzcd}:\TAB\TAB\TAB
    K=\ker(A\to B)\ILEQ A.
  \]
  As a consequence, its cokernel is given as (within the image)
  \[
    \begin{tikzcd}[column sep=0.8cm]
      (X,A)\rar & (Y,B)\rar{\coker} & \left(\frac{Y}{YL},\frac{B}{L}\right)
    \end{tikzcd}:\TAB
    L = {\sum}_{n\geq0} \bra{Y^n}BAB\ket{Y^{n}}\ILEQ B.
  \]
  Furthermore, the cokernel morphisms are precisely \textbf{the surjective morphisms.}
\end{BREAKTHM}

Before we begin with the proof for the above proposition
let us give a quick clarification on the quotient norm
(and a rather quite simplified proof thereof):
\begin{BREAKLMM}[{compare with \cite[Lemma~2.1]{FOWLER-MUHLY-RAEBURN} and \cite[Lemma~1.5]{KATSURA-GAUGE-IDEALS}}]
  It holds for quotient correspondences: The norm for Hilbert modules agrees with the norm for quotient Banach spaces,
  \[
    \left\|\frac{x_0}{XK}\right\|_\mathrm{Hilbert}^2 = \inf\|\braket{x_0| x_0}+K\|
    = \inf\|x_0+XK\|^2 = \left\|\frac{x_0}{XK}\right\|_\mathrm{Banach}^2.
  \]
  In particular, the quotient correspondence is complete (by Cohen--Hewitt).
\end{BREAKLMM}
\begin{proof}As compared to both references,
  we give a \textbf{simplified proof:}\\
  On the one hand we have the obvious inclusion
  \[
    \|x_0+XK\|^2\subset\|\braket{x_0+XK| x_0+XK}\|\subset \|\braket{x_0| x_0}+K\|
  \]
  and as such the bound (when applying infima)
  \[
   \left\|\frac{x_0}{XK}\right\|_\mathrm{Hilbert}^2 = \inf\|\braket{x_0|x_0}+K\|
   \leq \inf \|x_0+XK\|^2 = \left\|\frac{x_0}{XK}\right\|_\mathrm{Banach}^2.
  \]
  For the converse implication we may invoke the well-known formula for the quotient norm on operator algebras \textbf{(the simplifying trick)}
  \[
    \left\|\frac{\braket{x_0|x_0}}{K}\right\| = \inf\|\braket{x_0|x_0}+K\| =\lim_{e\to1}\|(1-e)\braket{x_0|x_0}(1-e)\|
  \]
  where the limit runs over some or any approximate identity for the ideal.\\
  As such we obtain the converse containment
  \[
    \|\braket{x_0-x_0e|x_0-x_0e}\| = \|x_0-x_0e\|^2\in \|x_0+XK\|^2
  \]
  and so also the desired converse bound:
  \[
    \left\|\frac{x_0}{XK}\right\|_\mathrm{Banach}^2 = \inf \|x_0+XK\|^2
    \leq \lim_{e\to1}\|(1-e)\braket{x_0|x_0}(1-e)\| = \left\|\frac{x_0}{XK}\right\|_\mathrm{Hilbert}^2
  \]
  So the desired Hilbert module norm and Banach space norm coincide.
\end{proof}

Having clarified the completeness for the quotient correspondence,
we may now savely get to the desired kernel and cokernel morphisms.
\begin{proof}[Proof of theorem~\ref{THEOREM:KERNEL-COKERNEL-MORPHISMS}]
  We begin with the statements about kernel morphisms.\\
  We first verify that each morphism admits a (categorical) kernel given by
  \[\hspace{-1cm}
    (XK,K)= \ker(\begin{tikzcd}[column sep=0.5cm]
      (X,A)\rar & (Y,B)
    \end{tikzcd}):\TAB\TAB
    K=\ker(A\to B).
  \]
  Clearly the kernel annihilates the morphism due to our observation \eqref{KERNEL-OBSERVATION}:
  \[
  \Big(\ker(A\to B)\to A\to B\Big)=0\TAB\implies\TAB\Big(X\ker(A\to B)\to X\to Y\Big)=0
  \]
  Meanwhile, let us also note that any such ideal is invariant:
  \begin{gather*}
    \bra{X}\ker(A\to B)\ket{X}\subset\ker(A\to B):\TAB (A\to B)\Big(\bra{X}\ker(A\to B)\ket{X}\Big) = \\[2\jot]
    = (Y\from X)\bra{X}\cdot(A\to B)\ker(A\to B)\cdot(Y\from X)\ket{X}=0.
  \end{gather*}
  Conversely, suppose another morphism annihilates from the left
  \begin{gather*}
    \begin{tikzcd}[column sep=1cm]
        (W,D)\rar & (X,A)
    \end{tikzcd}:\TAB
    (W\to X\to Y)=0,\TAB (D\to A\to B)=0.
  \end{gather*}
  (The annihilation suffices on coefficient algebras due to our observation above.)
  Trivially, we have the desired inclusion at the level of coefficient algebras:
  \[
    (D\to A\to B)=0\TAB\implies\TAB D\to\ker(A\to B)
  \]
  On the other hand, recall that the induced morphism on the support ideal coincides with the morphism at the level of coefficient algebras.
  As such we necessarily also have an inclusion into the kernel,
  \begin{TIKZCD}
    \braket{W|W} \rar[dashed]\dar& \ker(A\to B)\rar\dar[equal] &A\dar[equal] \\
    D \rar& \ker(A\to B)\rar &A.
  \end{TIKZCD}
  We however have the equivalance (using the characterization \eqref{INVARIANCE-IMAGE}):
  \[
    W\to X\ker(A\to B)\iff \braket{W|W}\to \ker(A\to B).
  \]
  As such we obtain also the desired inclusion at the level of correspondences,
  \[
    D\to\ker(A\to B)\TAB\implies\TAB W\to X\ker(A\to B).
  \]
  So there exists also a unique factorization over the kernel correspondence above.
  Every morphism thus admits a kernel given by the above kernel morphism.

  We next verify that each such morphism is indeed some kernel,
  namely the kernel from the short exact sequence in the proposition:
  \begin{gather*}
    \begin{tikzcd}[column sep=1cm]
      0\rar& (XK,K)\rar& (X,A)\rar& \left(\frac{X}{XK},\frac{A}{K}\right) \rar&0
    \end{tikzcd}:\\[2\jot]
    \ker\left(\begin{tikzcd}[column sep=0.5cm]
      (X,A)\rar& \left(\frac{X}{XK},\frac{A}{K}\right)
    \end{tikzcd}\right)=
    \begin{tikzcd}[column sep=0.5cm]
      (XK,K)\rar& (X,A).
    \end{tikzcd}
  \end{gather*}
  For this let us first verify that the quotient gives a well-defined correspondence
  \[
    \left(\tfrac{X}{XK}\right)^*\left(\tfrac{X}{XK}\right)\subset\left(\tfrac{A}{K}\right),\TAB
    \left(\tfrac{X}{XK}\right)\left(\tfrac{A}{K}\right)\subset\left(\tfrac{X}{XK}\right),\TAB
    \left(\tfrac{A}{K}\right)\left(\tfrac{X}{XK}\right)\subset\left(\tfrac{X}{XK}\right).
  \]
  In other words we need to verify the relations
  \[
    (XK)^*X\subset K,\TAB X^*(XK)\subset K,\TAB\ldots,\TAB (K\subset A)X\subset(XK).
  \]
  The only nontrivial one here (which is not guaranteed automatically) is the last one which precisely calls for hereditary ideals (equivalently invariant ideals)
  \[
    (K\subset A)X= KX\subset XK\TAB (\iff KX\subset XK)\hspace{-2cm}
  \]
  and so the quotient gives a well-defined correspondence.\\
  Let us now get to its kernel. We already found the unique (categorical) kernel of morphisms above. As such the desired equality now easily follows from
  \begin{gather*}
    \ker\left(\begin{tikzcd}[column sep=0.5cm]
      (X,A)\rar& \left(\frac{X}{XK},\frac{A}{K}\right)
    \end{tikzcd}\right)=
    \Big(X\ker\left(A\to \tfrac{A}{K}\right),\ker\left(A\to \tfrac{A}{K}\right)\Big)=(XK,K).
  \end{gather*}
  So far about kernel morphisms for correspondences.
  Let us now get to cokernel morphisms. We first verify that each cokernel is also of the special form as above
  \begin{gather*}
    \begin{tikzcd}[column sep=1cm]
      0\rar& (XK,K)\rar& (X,A)\rar& \left(\frac{X}{XK},\frac{A}{K}\right) \rar&0
    \end{tikzcd}:\\[2\jot]
    \coker\left(\begin{tikzcd}[column sep=0.5cm]
      (XK,K)\rar& (X,A)
    \end{tikzcd}\right)=
    \begin{tikzcd}[column sep=0.5cm]
      (X,A)\rar& \left(\frac{X}{XK},\frac{A}{K}\right).
    \end{tikzcd}
  \end{gather*}
  Clearly the quotient annihilates the kernel (using our obersation \eqref{KERNEL-OBSERVATION}):
  \[
    \left(K\to A\to \tfrac{A}{K}\right)=0\TAB\implies\TAB (XK\to X\to \tfrac{X}{XK})=0
  \]
  Conversely, given any other annihilating morphism say
  \[
    \begin{tikzcd}[column sep=1cm]
        (X,A)\rar{(\tau,\phi)} & (Y,B)
    \end{tikzcd}:\TAB(XK\to X\to Y)=0,\TAB (K\to A\to B)=0.
  \]
  Then both morphisms factor uniquely over the quotient (as linear maps):
  \begin{TIKZCD}[column sep=1.4cm]
    0\rar& (XK,K)\rar\dar& (X,A)\rar\dar[swap]{(\tau,\phi)}& \left(\frac{X}{XK},\frac{A}{K}\right) \rar\dar[dashed]&0 \\
    0\rar& 0\rar& (Y,B)\rar[equal]& (Y,B) \rar&0.
  \end{TIKZCD}
  So it remains to verify that the factorization defines a morphism of correspondences.
  That however basically follows from the original morphism (using congruence classes):
  \begin{gather*}
    \phi\Big((x+XK)^*(y+XK)\Big)= \phi(x^*y)=\tau(x)^*\tau(y)=\tau(x+XK)^*\tau(y+XK),\\
    \tau\Big((x+XK)(a+K)\Big) = \tau(xa) = \tau(x)\phi(a) = \tau(x+XK)\phi(a+K).
  \end{gather*}
  Recall here that the coherence with the left action follows automatically,
  \[
    \implies\tau\Big((a+K)(x+XK)\Big) = \phi(a+K)\tau(x+XK).
  \]
  So we have found the desired cokernel:
  \[
  \coker\left(\begin{tikzcd}[column sep=0.5cm]
    (XK,K)\rar& (X,A)
  \end{tikzcd}\right)
  = \left(\tfrac{X}{XK},\tfrac{A}{K}\right).
  \]
  We now wish to find the cokernel for general morphisms:
  \[
    \coker(\begin{tikzcd}[column sep=1cm]
        (X,A)\rar & (Y,B)
    \end{tikzcd})=(?,?)
  \]
  For this we may now make use of the following relation to our advantage:\linebreak
  The kernel and cokernel operator satisfy the Galois connection
  (verbatim from \cite[section VIII.1]{MACLANE} and confer further \cite[chapter~1~and~2]{FREYD-BOOK}):
  \[
    \ker\coker\ker=\ker,\TAB\TAB\ker\ker=0=\coker\coker,\TAB\TAB\coker\ker\coker=\coker.
  \]
  On the other hand, we have already found the form of kernel morphisms:
  \[
    \ker\Bigl(~(Y,B)\to(?,?)~\Bigr) = (YL,L)\to (Y,B).
  \]
  as well as the cokernel for kernel morphisms
  \[
    \coker(\begin{tikzcd}[column sep=0.5cm]
        (YL,L)\rar & (Y,B)
    \end{tikzcd}) =
    \begin{tikzcd}[column sep=0.5cm]
        (Y,B)\rar & \left(\tfrac{Y}{YL},\tfrac{B}{L}\right)
    \end{tikzcd}.
  \]
  We may thus combine these with the Galois connection between the kernel and cokernel operator to find the necessary shape of cokernel morphisms:
  \begin{gather*}
    (?,?)=\coker
    (\begin{tikzcd}[column sep=0.5cm]
        (X,A)\rar & (Y,B)
    \end{tikzcd}) = \\
    = \coker\ker\coker
    (\begin{tikzcd}[column sep=0.5cm]
        (X,A)\rar & (Y,B)
    \end{tikzcd})
    = \left(\tfrac{Y}{YL},\tfrac{B}{L}\right).
  \end{gather*}
  So we are left with finding the invariant ideal which determines the quotient.
  Consider for this the factorization over the image (kernel of cokernel)
  \begin{TIKZCD}[column sep=1.4cm]
    0\rar& (X,A)\rar[equal]\dar[dashed]& (X,A)\rar\dar& 0 \rar\dar&0 \\
    0\rar& (YL,L)\rar[swap]& (Y,B)\rar& \left(\frac{Y}{YL},\frac{B}{L}\right) \rar&0.
  \end{TIKZCD}
  As such the invariant ideal (which determines the kernel) necessarily contains the image of the coefficient algebra
  \[
    \begin{tikzcd}[column sep=0.5cm]
        (X,A)\rar & (YL,L)
    \end{tikzcd}
    \TAB\implies\TAB
    (A\to L\to B)
  \]
  and is necessarily also the smallest such ideal (generated by the image)
  \[
    L = BAB + \bra{Y}BAB\ket{Y}+\ldots = {\sum}_{n\geq0} \bra{Y^n}BAB\ket{Y^{n}}\ILEQ B.
  \]
  Indeed any larger invariant ideal defines nothing but a quotient beyond:
  \begin{TIKZCD}[column sep=1.3cm]
    0\rar& (YL,L)\rar\dar& (Y,B)\rar\dar[equal]& \left(\frac{Y}{YL},\frac{B}{L}\right) \rar\dar&0 \\
    0\rar& (YL',L')\rar& (Y,B)\rar& \left(\frac{Y}{YL'},\frac{B}{L'}\right) \rar&0
  \end{TIKZCD}
  So we have also found the cokernel for general morphisms.
  For the final assertion we note that every cokernel is onto as easily seen from their special form
  \[
    \left(\begin{tikzcd}[column sep=small]
      X\rar&\tfrac{X}{XK}\rar&0
    \end{tikzcd}\right)
    \TAB\text{and}\TAB
    \left(\begin{tikzcd}[column sep=small]
      A\rar&\tfrac{A}{K}\rar&0
    \end{tikzcd}\right).
  \]
  For the converse consider a surjective morphism $(X,A)\to(Y,B):$
  \[
    \left(\begin{tikzcd}[column sep=small]
      X\rar& Y\rar&0
    \end{tikzcd}\right)
    \TAB\text{and}\TAB
    \left(\begin{tikzcd}[column sep=small]
      A\rar& B\rar&0
    \end{tikzcd}\right).
  \]
  Recall from above that the kernel and cokernel operator define a Galois connection
  and so there is no choice for our morphism (to define a cokernel) than to arise as its own coimage (cokernel of kernel)
  \begin{gather*}
    (X,A)\to(Y,B) = \coker\Bigl(~\mathrm{some~morphism}~\Bigr)\\[2\jot]
                  = \coker\ker\coker\Bigl(~\mathrm{some~morphism}~\Bigr) = \coker\ker\Bigl(~(X,A)\to(Y,B)~\Bigr).
  \end{gather*}
  Recall for this the factorisation over its coimage (as outlined before):
  \begin{TIKZCD}[column sep=1.3cm]
    0\rar& (XK,K)\rar{\ker}\dar& (X,A)\rar{\coker\ker}\dar& \left(\frac{X}{XK},\frac{A}{K}\right) \rar\dar[dashed]&0 \\
    0\rar& 0\rar& (Y,B)\rar[equal]& (Y,B) \rar&0
  \end{TIKZCD}
  So we need to verify that the factorisation defines an isomorphism.\\
  For this we obtain as our morphism is onto (on the coefficient algebra):
  \begin{TIKZCD}[column sep=large,row sep=0.7cm]
    0\rar& K\rar\dar[equal]& A\rar\dar[equal]&A/K\rar\dar[equal]&0 \\
    0\rar& K\rar& A\rar&B\rar&0
  \end{TIKZCD}
  So the factorisation defines an isomorphism on the level of coefficient algebras.\\
  As a consequence, the factorisation is also faithful on the correspondence (recall for instance proposition \ref{HIGHER-MIXED}):
  \[
    \ker\left(\SHORTTIKZCD{[column sep=small]\tfrac{A}{K}\rar&B}\right)=0
    \TAB\implies\TAB
    \ker\left(\SHORTTIKZCD{[column sep=small]\tfrac{X}{XK}\rar&Y}\right)=0.
  \]
  On the other hand, the factorisation is also onto (for the correspondence)
  \[
    \im\left(\begin{tikzcd}[column sep=small]
      \tfrac{X}{XK}\rar& Y
    \end{tikzcd}\right) =
    \im\left(\begin{tikzcd}[column sep=small]
      X\rar&\tfrac{X}{XK}\rar& Y
    \end{tikzcd}\right) =
    \im\left(\begin{tikzcd}[column sep=small]
      X\rar& Y
    \end{tikzcd}\right) = Y
  \]
  and as such the factorisation defines an isomrphism as desired.\\
  That is every surjective morphism defines a cokernel morphism as well.
\end{proof}

Let us identify such quotients in the context of directed graphs:
\begin{BREAKEXM}[Graph correspondences: quotient graphs]
  \label{GRAPH-CORRESPONDENCES:QUOTIENT-GRAPHS}
  Consider a graph correspondence as in Example \ref{GRAPH-CORRESPONDENCES:TENSOR-POWERS}
  \[
    X=\ell^2\Bigl(~E=\edges~\Bigr),\TAB A=c_0\Bigl(~\ver=\mathrm{vertices}~\Bigr)
  \]
  and its hereditary ideals given by hereditary collections (as in Example \ref{GRAPH-CORRESPONDENCES:HEREDITARY-IDEALS})
  \begin{gather*}
    K=c_0\Bigl(~S\subset \mathrm{vertices}~\Bigr)\ILEQ A:\TAB\TAB SE\subset ES.
  \end{gather*}
  It is already clear that its quotients themselves arise as a graph,
  simply since their coefficient algebras define discrete direct sums of vertices as in \ref{GRAPH-CORRESPONDENCES:TENSOR-POWERS}:
  \begin{TIKZCD}
    0\rar&[-0.5cm] K=c_0(S)\rar& A=c_0(V)\rar& \frac{A}{K}= c_0\Bigl(~T:=V\setminus S~\Bigr) \rar&[-0.5cm] 0.
  \end{TIKZCD}
  And indeed we may now simply reveal the quotient as
  \begin{TIKZCD}
    0\rar&
    XK=\ell^2(ES)\rar&
    X=\ell^2(E)\rar&
    \frac{X}{XK}=\ell^2(TE)\rar& 0.
  \end{TIKZCD}
  Thus its quotients simply arise as \textbf{the complementary graphs.}
\end{BREAKEXM}

We have now found everything about kernel and cokernel morphisms.\\
With this in our toolbox, we now proceed to covariances.
For this we note that we may render any representation faithful: Indeed we may simply factor any representation over the quotient correspondence
\begin{TIKZCD}[column sep=1.5cm]
  0\rar& (XK,K)\rar{\ker}\dar& (X,A)\rar\dar& \left(\frac{X}{XK},\frac{A}{K}\right) \rar\dar[dashed]&0 \\
  0\rar& 0\rar[swap]& (B,B)\rar[equal]& (B,B) \rar&0.
\end{TIKZCD}
The resulting factorization is faithful (isometric) on the coefficient algebra and as such also on the entire correspondence (see proposition \ref{HIGHER-MIXED}):
\[
  K=\ker(A\to B):\TAB\TAB\ker\left(\tfrac{A}{K}\to B\right)=0\implies\ker\left(\tfrac{X}{XK}\to Y\right)=0.
\]
So we may always first pass to the (unique) quotient correspondence to render any representation faithful.
This is the first step in the classification of \mbox{gauge-equivariant} representations.
The second step is to understand the possibly occuring covariances in nature. This is captured as an important observation by Katsura below.
For this let us recall Katsura's ideal
\begin{equation}\label{KATSURAS-IDEAL}
  \max(X,A) = \ker(A\act X)^\perp\cap XX^*\ILEQ A
\end{equation}
which we denoted as maximal ideal for reasons which will become clear in the result below,
and we will often drop the dependence on the coefficient algebra for simplicity.
Note also that we have already encountered this ideal in different context on Hilbert bimodules (see Proposition \ref{MAXIMAL-FAITHFUL-ACTION} and \ref{HILBERT-BIMODULES}).
With this ideal in mind let us get to Katsura's observation (which we split as a result on faithful representations and further as one on faithful morphisms in general):

\begin{BREAKPROP}[{First part of \cite[Proposition 3.3]{KATSURA-SEMINAL}}]
  \label{FAITHFUL-REPRESENTATIONS:MAXIMAL-COVARIANCE}
  Consider an embedding into some ambient operator algebra
  \[
    \begin{tikzcd}[column sep=0.5cm]
      (X,A)\subset B
    \end{tikzcd}:\TAB\TAB
    A\subset B\TAB(\implies X\subset B).
  \]
  Then its covariance \eqref{COVARIANCE:DEFINITION} lies perpendicular to the kernel
  \[
    \cov(\begin{tikzcd}[column sep=0.5cm]
      (X,A)\subset B
    \end{tikzcd})\perp\ker(A\act X).
  \]
  In particular it holds for general representations (and factored as above)
  \begin{gather*}
    \begin{tikzcd}[column sep=1.2cm,row sep=1cm]
      (X,A)\rar& \left(\frac{X}{XK},\frac{A}{K}\right) \rar[dashed]&B:
    \end{tikzcd}\\[2\jot]
    0\subset\cov\left(\begin{tikzcd}[column sep=small]
      \left(\frac{X}{XK},\frac{A}{K}\right)\rar&B
    \end{tikzcd}\right) \subset \max\left(\tfrac{X}{XK},\tfrac{A}{K}\right).
  \end{gather*}
  So the range of possible covariances is bounded from above by Katsura's ideal, whence its name and notation as maximal ideal in \eqref{KATSURAS-IDEAL}.
\end{BREAKPROP}
\begin{proof}
  We recast the arguments from Katsura in our language.
  For this let us first recall the commutative diagram for the covariance ideal,
  \begin{TIKZCD}[column sep=5cm]
    \cov(X\to B)\subset A\dar\rar & B\dar[equal] \\
    \ketbra{X}{X}\rar & B.
  \end{TIKZCD}
  In our case the representation is also faithful (all the horizontal paths):
  \[
    \hspace{1cm}A\subset B,\TAB\TAB X\subset B,\TAB\TAB XX^*\subset B,\TAB\TAB\ldots
  \]
  On the other hand, the covariance ideal defines an ideal in the coefficient algebra (as we have observed in the previous section). As such we also have
  \[
    \ker(A\act X)\cov(X\to B)\subset \cov(X\to B).
  \]
  So one may place this expression in the commutative diagram above to obtain
  \begin{TIKZCD}[column sep=4cm]
    \ker(A\act X)\cov(X\to B)\dar\rar & B\dar[equal] \\
    \ker(A\act X)\ketbra{X}{X}=0\rar & B
  \end{TIKZCD}
  and simply trace back the desired orthogonality (using the top path).\\
  The remaining points follow from the discussion preceeding the proposition.
\end{proof}

Let us reveal the maximal covariance in the case of graph algebras:
\begin{BREAKEXM}[Graph correspondences: maximal covariance]
  \label{GRAPH-CORRESPONDENCES:MAXIMAL-COVARIANCE}
  Consider a graph correspondence as in Example \ref{GRAPH-CORRESPONDENCES:TENSOR-POWERS}:
  \[
    X=\ell^2\Big(E=\edges\Big),\TAB A=c_0(\mathrm{vertices}).
  \]
  For these the trivially acting portion (the kernel for the left action) and the compactly acting portion correspond to sources and finite receivers:
  \begin{align*}
    \ker(A\act X)= c_0\Big(a:|a\edges|=0\Big)     &=c_0(\sources),\\[2\jot]
    A\cap XX^*   = c_0\Big(a:|a\edges|<\infty\Big)&=c_0(\finRECEIVERS).
  \end{align*}
  As such the orthogonal complement reads together
  \begin{gather*}
    \max(X,A)=\ker(A\act X)^\perp\cap XX^* =\\[2\jot]
    = c_0\Big(a:0<|a\edges|<\infty\Big)=c_0(\regular).
  \end{gather*}
  That is the maximal covariance corresponds to \textbf{the regular vertices}
  and as such any other covariance ideal corresponds to simply some collection of such.
\end{BREAKEXM}

We finish this section with an investigation of covariances for morphisms between correspondences (as opposed to representations into operator algebras).
We begin with faithful morphisms between correspondences. As for representations we have (compare Proposition \ref{HIGHER-MIXED}): Being faithful passes from the coefficient algebra to the correspondence (and to any other power):
\[
  \begin{tikzcd}
    (X,A)\to (Y,B)
  \end{tikzcd}:\TAB\TAB
  A\subset B \TAB\implies\TAB X\subset Y.
\]
So the faithful morphisms may be seen as nothing but a subcorrespondence: That is those whose inner product already lies in the subalgebra and similar for the action from either side,
\begin{equation}\label{SUBCORRESPONDENCE:DEFINITION}
  \begin{gathered}
    \braket{Y|Y}\subset B,\TAB BY\subset Y,\TAB YB\subset Y:\\[2\jot]
    X\subset Y,\TAB A\subset B:\TAB\TAB\braket{X|X}\subset A,\TAB
    AX\subset X,\TAB XA\subset X
  \end{gathered}
\end{equation}
while for comparison mixed expressions only satisfy
\[
  \braket{X|Y}\subset B,\TAB BX\subset Y,\TAB XB\subset Y.
\]
Schematically the inclusions for subcorrespondences may look like
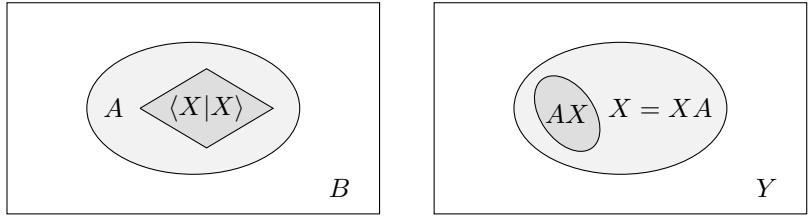
\begin{center}
  \begin{tikzpicture}[scale=0.35]
    \draw (-7,-4) rectangle (7,4);
    \draw[fill=gray!10] (0,0) ellipse [x radius=4,y radius=2.5];
    \draw[fill=gray!30,fill opacity=0.8] (-2,0) to (0.5,-1.5) to (3,0) to (0.5,1.5) to (-2,0);
    \node at (0.5,0) {$\braket{X|X}$};
    \node at (-3,0) {$A$};
    \node at (5.5,-3) {$B$};
  \end{tikzpicture}\TAB\TAB
  \begin{tikzpicture}[scale=0.35]
    \draw (-7,-4) rectangle (7,4);
    \draw[fill=gray!10] (0,0) ellipse [x radius=4,y radius=2.5];
    \draw[fill=gray!30,fill opacity=0.8] (-2,-0.2) ellipse [x radius=1.6,y radius=1,rotate=125];
    \node at (-2,-0.2) {$AX$};
    \node at (1.5,0) {$X=XA$};
    \node at (5.5,-3) {$Y$};
  \end{tikzpicture}
\end{center}
which compare to mixed expressions as possibly only
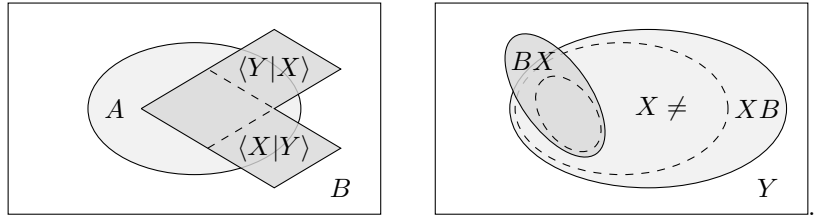
\begin{center}
  \begin{tikzpicture}[scale=0.35]
    \draw (-7,-4) rectangle (7,4);
    \draw[fill=gray!10] (0,0) ellipse [x radius=4,y radius=2.5];
    \draw[fill=gray!30,fill opacity=0.8] (-2,0) to (3,-3) to (5.5,-1.5) to (3,0) to (5.5,1.5) to (3,3) to (-2,0);
    \draw[dashed] (0.5,-1.5) to (3,0) to (0.5,1.5);
    \node at (3,-1.5) {$\braket{X|Y}$};
    \node at (3,1.5) {$\braket{Y|X}$};
    \node at (-3,0) {$A$};
    \node at (5.5,-3) {$B$};
  \end{tikzpicture}\TAB\TAB
  \begin{tikzpicture}[scale=0.35]
    \draw (-7,-4) rectangle (7,4);
    \draw[fill=gray!10] (1.0,0) ellipse [x radius=5.2,y radius=3];
    \draw[fill=gray!10,dashed] (0,0) ellipse [x radius=4,y radius=2.5];
    \draw[fill=gray!30,fill opacity=0.8] (-2.50,0.5) ellipse [x radius=2.7,y radius=1.3,rotate=125];
    \draw[dashed] (-2,-0.2) ellipse [x radius=1.6,y radius=1,rotate=125];
    \node at (-3.3,1.8) {$BX$};
    \node at (1.5,0) {$X\neq$};
    \node at (5.1,0) {$XB$};
    \node at (5.5,-3) {$Y$};
  \end{tikzpicture}.
\end{center}
So one may think of a subcorrespondence as some sort of coherent restriction: Think of global actions and their restrictions to possibly partial actions only.\linebreak
The action of compact operators on the ambient correspondence further reads
\[
  (XX^*)Y\subset X(Y^*Y)\subset XB\subset Y
\]
Altogher, we may thus really think of a subcorrespondence simply as a subspace and subalgebra,
which will be quite convenient also further on.\linebreak
In particular one has (for rather trivial reason)
\[
  (T-T')Y=0\TAB\implies\TAB (T-T')X=0
\]
and as such also for $a\in A$ and $k\in XX^*$,
\begin{equation}
  \label{SUBCORRESPONDENCE:OPERATOR-EQUALITY}
  (a-k)Y=0\TAB\implies\TAB (a-k)X=0.
\end{equation}
This basically trivial implication already verifies Katsura's second observation,
and note how \textbf{thinking in terms of subspaces} and subalgebras paid off:

\begin{BREAKPROP}[{Second part of \cite[Proposition 3.3]{KATSURA-SEMINAL}}]\label{COVARIANCE:SUBCORRESPONDENCES}
  For a subcorrespondence as in \eqref{SUBCORRESPONDENCE:DEFINITION} the covariance arises as pullback
  \begin{gather*}
    \cov\Big(\begin{tikzcd}[column sep=0.5cm]
      X\subset Y
    \end{tikzcd}\Big) =
    \Bigl\{~a\in A~\Big|~\im(a\act Y)\in \im(XX^*\act Y)~\Bigr\} \\[2\jot]
    = (A\act Y)\inv\Bigl(~\im(A\act Y)\cap\im(XX^*\act Y)~\Bigr)
  \end{gather*}
  which we usually abbreviate simply as their common intersection.\\
  As a consequence it follows the characterization for trivial covariance
  \[
    \cov\Big(\begin{tikzcd}[column sep=0.5cm]
      (X,A)\subset B
    \end{tikzcd}\Big)=0\TAB\iff\TAB \im(A\to B)\cap\im\Big(XX^*\to B\Big)=0
  \]
  which reveals the familiar slogan \textbf{for the Toeplitz representation:}\\
  Those whose coefficient algebra has trivial intersection with compact operators.
\end{BREAKPROP}
\begin{proof}This is nothing but the trivial implication \eqref{SUBCORRESPONDENCE:OPERATOR-EQUALITY}.
\end{proof}

We continue with the covariance for kernel and cokernel morphisms.\linebreak
We begin for this with the following fairly standard result due to Kajiwara--Pinzari--Watatani
which will become particularly useful in later context:

\begin{BREAKPROP}[{see \cite[Proposition~4.2]{KAJIWARA-PINZARI-WATATANI}}]
  \label{SHORT-EXACT:ADJOINTABLES}
  A short exact sequence of correspondences (see theorem~\ref{THEOREM:KERNEL-COKERNEL-MORPHISMS})
  \begin{TIKZCD}[column sep=1cm]
    0\rar& (XK,K)\rar& (X,A)\rar& \left(\frac{X}{XK},\frac{A}{K}\right) \rar&0
  \end{TIKZCD}
  induces a morphism (of operator algebras)
  \[
    \begin{tikzcd}[column sep=1cm]
      \ADJOINTS(X)\rar& \ADJOINTS\left(\frac{X}{XK}\right)
    \end{tikzcd}:\TAB\TAB
    T(x+XK):= Tx +XK
  \]
  whose kernel admits the equivalent characterization
  \begin{equation}\label{INVERSE-INVARIANCE}
    X\inv(K) := \Big\{X^*TX\subset K\Big\}
    = \Big\{TX\subset XK\Big\}
    = \ker\Bigl(\ADJOINTS(X)\to \ADJOINTS\left(\tfrac{X}{XK}\right)\Bigr).
  \end{equation}
  Further the morphism commutes with the left action by the coefficient algebra,
  \begin{TIKZCD}[column sep=1.5cm,row sep=0.8cm]
    0\rar& K\rar& A\rar\dar& \frac{A}{AK}\rar\dar&0 \\
    & & \ADJOINTS(X)\rar& \ADJOINTS\left(\tfrac{X}{XK}\right).
  \end{TIKZCD}
\end{BREAKPROP}
\begin{proof}
  In order to verify the induced morphism let us first note the following:
  Adjointable operators restrict to kernel correspondences:
  \[
    T\in\ADJOINTS(X)\implies T\in\ADJOINTS(XK):\TAB\TAB T(XK)\subset(TX)K\subset XK.
  \]
  As such we obtain a commutative diagram and whence the adjointable operator descends also to the quotient,
  \begin{TIKZCD}[column sep=1.5cm]
    0\rar& XK\rar\dar\drar[phantom]{T}& X\rar\dar& \frac{X}{XK} \rar\dar[dashed]&0 \\
    0\rar& XK\rar& X\rar& \frac{X}{XK} \rar&0.
  \end{TIKZCD}
  The resulting operator is adjointable as one easily verifies
  \[
    \braket{T(x+XK)\mid y+XK}=\braket{TX\mid y}=\braket{x\mid T^*y}=\braket{x+XK\mid T^*(y+XK)}.
  \]
  As such we got the desired morphism (of operator algebras)
  \[
    \begin{tikzcd}[column sep=1cm]
      \ADJOINTS(X)\rar& \ADJOINTS\left(\frac{X}{XK},\frac{A}{K}\right)
    \end{tikzcd}:\TAB\TAB
    T(x+XK):= Tx +XK.
  \]
  We meanwhile note that the morphism is generally not onto!\\
  Moreover, the morphism clearly commutes with the left action by the coefficient algebra as one easily verifies (think in terms of subspaces)
  \[
    (a+K)(x+XK)=ax +XK = a(x+XK).
  \]
  The remaining claim about the kernel immediately follows from \eqref{INVARIANCE-IMAGE}.
\end{proof}

We continue with the analogous result for compact operators as for example in Fowler--Muhly--Raeburn
(for which we provide a slightly simplified proof):

\begin{BREAKPROP}[{see \cite[Lemma~2.6]{FOWLER-MUHLY-RAEBURN}}]
  \label{SHORT-EXACT:COMPACTS}
  A short exact sequence of correspondences (see theorem~\ref{THEOREM:KERNEL-COKERNEL-MORPHISMS})
  \begin{TIKZCD}[column sep=1cm]
    0\rar& (XK,K)\rar& (X,A)\rar& \left(\frac{X}{XK},\frac{A}{K}\right) \rar&0
  \end{TIKZCD}
  induces a short exact sequence at the level of compact operators
  \begin{TIKZCD}[column sep=0.8cm,row sep=1cm]
    0\rar& (XK)(XK)^*=XKX^*\rar& XX^*\rar\dar& \left(\frac{X}{XK}\right)\left(\frac{X}{XK}\right)^* \rar\dar&0 \\
    && \ADJOINTS(X)\rar & \ADJOINTS\left(\frac{X}{XK}\right)
  \end{TIKZCD}
  which restricts from the morphism of adjointable operators as indicated.
\end{BREAKPROP}
\begin{proof}Clearly, the induced morphism on compact operators commutes with the morphism of adjointable operators since (think in terms of subspaces)
  \[
    (x+XK)(y+XK)^*(z+XK)= xy^*z +XK = xy^*(z+XK).
  \]
  Next we already know that the kernel correspondence (as a subcorrespondence) defines an embedding at the level of compact operators:
  \[
    (XK,K)\subset(X,A)\TAB\implies\TAB (XK)(XK)^*\subset XX^*.
  \]
  For kernel correspondences these now further define an ideal as for example
  \[
    \Big((XK)(XK)^*=XKX^*\Big)XX^*=XK(X^*XX^*)= XKX^*.
  \]
  On the other hand, the right-hand morphism is also clearly onto since
  \[
    \begin{tikzcd}[column sep=0.5cm]
      X\rar&\frac{X}{XK}\rar&0
    \end{tikzcd}
    \TAB\implies\TAB
    \begin{tikzcd}[column sep=0.5cm]
      XX^*\rar& \left(\frac{X}{XK}\right)\left(\frac{X}{XK}\right)^* \rar&0.
    \end{tikzcd}
  \]
  Regarding exactness in the middle, we may now invoke the characterization of the kernel from the previous proposition, which reads for compact operators
  \[
    X\inv(K)\cap XX^* = \ker\Big(\begin{tikzcd}[column sep=0.5cm]
      XX^*\rar&\left(\frac{X}{XK}\right)\left(\frac{X}{XK}\right)^*
    \end{tikzcd}\Big).
  \]
  While the ideal clearly lies inside the kernel (left as an exercise for the reader)
  the converse inclusion now easily follows from the above description:
  \begin{gather*}
    {X\inv(K)\cap XX^*}
    = XX^*\Big({X\inv(K)\cap XX^*}\Big)XX^* \\
    \subset X\Big(X^*X\inv(K) X\Big)X^*\subset XKX^*.
  \end{gather*}
  So we have found the desired exactness for compact operators.
\end{proof}

We may now easily derive the desired covariance for kernel and cokernel morphisms
(with the previous results in mind)

\begin{PROP}
  \label{COVARIANCE:KERNEL-COKERNEL}
  Kernel and cokernel morphisms have full covariance,
  \begin{gather*}
    \begin{tikzcd}[column sep=1cm]
      0\rar& (XK,K)\rar& (X,A)\rar& \left(\frac{X}{XK},\frac{A}{K}\right) \rar&0
    \end{tikzcd}:\\[2\jot]
    \begin{aligned}
      \cov\Big(\begin{tikzcd}[column sep=1cm]
        (XK,K)\to (X,A)
      \end{tikzcd}\Big)
      &=K\cap XKX^*,\\
      \cov\Big(\begin{tikzcd}[column sep=1cm]
        (X,A)\to \left(\frac{X}{XK},\frac{A}{K}\right)
      \end{tikzcd}\Big)
      &=A\cap XX^*.
    \end{aligned}
  \end{gather*}
  This stands in contrast to the covariance for representations:\\
  The covariance for representations is bounded by Katsura's ideal \eqref{KATSURAS-IDEAL}.
\end{PROP}
\begin{proof}
  Note that kernel morphisms define a particular type of subcorrespondences
  and so we find ourselves in the situation of proposition \ref{COVARIANCE:SUBCORRESPONDENCES}:\linebreak
  The nontrivial converse of implication \eqref{SUBCORRESPONDENCE:OPERATOR-EQUALITY} holds for kernel correspondences,
  \[
    a\in K,\TAB k\in XKX^*:\TAB\TAB (a-k)XKX^*=0\TAB\implies\TAB (a-k)X=0.
  \]
  Note that we may equivalently verify the implication
  \[
    KX^*(a-k)X^*K=0\TAB \implies\TAB X^*(a-k)X=0.
  \]
  For this we note the inclusion (due to invariance):
  \[
    X^*(a-k)X \subset X^*(K+ XKX^*)X
    = (X^*KX)+ (X^*X)K(X^*X)\subset K.
  \]
  As such the above implication holds true since for
  \[
    K\Bigl(~X^*(a-k)X\subset K~\Bigr)K=0\TAB\implies\TAB X^*(a-k)X=0
  \]
  so kernel morphisms have full covariance as desired.\\
  Regarding the cokernel morphism, we may combine proposition \ref{SHORT-EXACT:ADJOINTABLES} and \ref{SHORT-EXACT:COMPACTS} to obtain the desired full covariance
  \begin{TIKZCD}[column sep=1.5cm,row sep=1cm]
    & A\cap XX^*\arrow[rr]\arrow[dd]\arrow[dddl,bend right] && \frac{A}{K}\arrow[dd]\arrow[dddl,bend right,dashed] \\\\
    & \ADJOINTS(X)\arrow[rr] && \ADJOINTS\left(\frac{X}{XK}\right)\\
    XX^*\arrow[rr]\arrow[ru] && \left(\frac{X}{XK}\right)\left(\frac{X}{XK}\right)^*.\arrow[ru]
  \end{TIKZCD}
  Concluding that kernel and cokernel morphisms have full covariance.
\end{proof}

Recall next that we may always render representations faithful
by passing to the induced representation on the quotient
\begin{TIKZCD}[column sep=1.5cm]
  0\rar& (XK,K)\rar{\ker}\dar& (X,A)\rar\dar& \left(\frac{X}{XK},\frac{A}{K}\right) \rar\dar[dashed]&0 \\
  0\rar& 0\rar[swap]& (B,B)\rar[equal]& (B,B) \rar&0.
\end{TIKZCD}
The previous result allows us now to further clarify the relation between the covariance on the quotient correspondence (the relevant one) and the covariance for the original representation (the author would like to note that he found this special instance as an observation made by Katsura in \cite{KATSURA-GAUGE-IDEALS}):

\begin{PROP}
  \label{COVARIANCE:ORIGINAL=QUOTIENT-PULLBACK}
  Consider a cokernel morphism (simply some onto morphism as observed in proposition \ref{THEOREM:KERNEL-COKERNEL-MORPHISMS}) followed by an arbitrary morphism
  \[
    \begin{tikzcd}
      (X,A) \rar& (Y,B)\rar&0
    \end{tikzcd}
    \TAB\text{and}\TAB
    \begin{tikzcd}
      (Y,B)\rar&(Z,C).
    \end{tikzcd}
  \]
  Then the covariance for the composition arises as pullback
  \[
    \cov(X\to Y\to Z)=(A\to B)\inv\cov(Y\to Z)\cap XX^*.
  \]
  In particular the covariance for a representation may be recovered from the induced representation on the quotient correspondence,
  \begin{gather*}
    \begin{tikzcd}
      (X,A)\rar& \left( \frac{X}{XK},\frac{A}{K} \right)\rar[dashed]& B:
    \end{tikzcd}\\[2\jot]
    \cov\bigl(~X\to B~\bigr)
    = \left(A\to\tfrac{A}{K}\right)\inv
    \cov\left(\tfrac{X}{XK}\to B\right)\cap XX^*
  \end{gather*}
  which resambles Katsura's observation \cite[lemma~5.10 statement (v)]{KATSURA-GAUGE-IDEALS}.
\end{PROP}
\begin{proof}The result now easily follows from our previous proposition:
  Indeed as cokernel morphisms have full covariance we obtain a covariance diagram for our cokernel and one for the arbitrary morphism
  \[
    \left(\begin{tikzcd}[row sep=0.7cm]
      A\cap XX^* \rar\dar& B\cap YY^*\dar\\
      XX^*\rar& YY^*
    \end{tikzcd}\right)
    \TAB\text{and}\TAB
    \left(\begin{tikzcd}[row sep=0.7cm]
      \cov(Y\to Z) \rar\dar& C\cap ZZ^*\dar\\
      YY^*\rar& ZZ^*
    \end{tikzcd}\right)
  \]
  and so also a commuting diagram for their composition
  \begin{TIKZCD}[row sep=0.7cm]
      (A\to B)\inv\bigl(~\ldots~\bigr)\cap XX^* \rar\dar& \cov(Y\to Z) \rar\dar& C\cap YY^*\dar\\
      XX^*\rar& YY^*\rar& ZZ^*.
  \end{TIKZCD}
  That is the pullback lies within the covariance
  \[
    (A\to B)\inv\cov(Y\to Z)\cap XX^*\subset \cov(X\to Z).
  \]
  For the converse inclusion it suffices to establish
  \begin{gather*}
    \cov(X\to Z)\subset(A\to B)\inv\cov(Y\to Z) \\[2\jot]
    \hspace{-1.5cm}\iff\TAB (A\to B)\cov(X\to Z)\subset\cov(Y\to Z)
  \end{gather*}
  which is to verify the covariance diagram
  \begin{TIKZCD}[row sep=0.7cm]
    \cov(X\to Z)\rar & (A\to B)\cov(X\to Z)\rar\dar& C\cap ZZ^*\dar\\
    & YY^* \rar& ZZ^*~?
  \end{TIKZCD}
  Indeed this may be easily seen by following the covariance diagram
  \begin{TIKZCD}[row sep=0.7cm]
    \cov(X\to Z)\rar\dar &[1cm] \ldots\rar &[1cm] C\cap ZZ^*\dar\\
    XX^*\rar & YY^* \rar& ZZ^*
  \end{TIKZCD}
  and the full covariance for our cokernel morphisms (once more)
  \begin{TIKZCD}[column sep=2cm,row sep=0.7cm]
    A\cap XX^*\rar\dar& B\cap YY^*\dar\\
    XX^*\rar & YY^* \rar& ZZ^*.
  \end{TIKZCD}
  So the covariance for the composition arises from the pullback as desired.\\
  The remaining statement arises now as special instance from above.
\end{proof}

We finish this section with the following negative result about the covariance for subcorrespondences:
While we have found that kernel and cokernel morphisms have full covariance, this is not the case for subcorrespondences in general.
That is simply the converse of implication \eqref{SUBCORRESPONDENCE:OPERATOR-EQUALITY} fails in general:
\[
  a\in A,\TAB k\in XX^*:\TAB\TAB(a-k)X=0\TAB\NOT\implies\TAB(a-k)Y=0.
\]
For this we consider the following somewhat minimal example:
\begin{BREAKEXM}[Subcorrespondence with zero covariance]
    Consider the direct sum of an operator algebra as both the coefficient algebra and the Hilbert module (which we depict as diagonal operators)
    \[
      Y=\mediumMATRIX{D\\&D}=B\TAB\implies\TAB Y^*Y\subset B,\TAB YB\subset Y\TAB\checkmark
    \]
    but with left action given by the flip automorphism on $B= D\oplus D$:
    \[
      \mediumMATRIX{d_1\\&d_2}\act \mediumMATRIX{y_1\\&y_2}=\left[\mediumMATRIX{&1\\1}\mediumMATRIX{d_1\\&d_2}\mediumMATRIX{&1\\1}\right]\mediumMATRIX{y_1\\&y_2}.
    \]
    Regard the subcorrespondence given by the subalgebra \textbf{(noninvariant ideal!)}
    \[
      \left(~X=\mediumMATRIX{D\\&0}=A~\right)\subset\left(~Y=\mediumMATRIX{D\\&D}=B~\right).
    \]
    Then its left action vanishes identically (and in particular by compact operators):
    \[
      A\act X=\left[\mediumMATRIX{&1\\1}\mediumMATRIX{D\\&0}\mediumMATRIX{&1\\1}\right]\mediumMATRIX{D\\&0} = \mediumMATRIX{0\\&D}\mediumMATRIX{D\\&0}=0
    \]
    On the other hand it never vanishes on the ambient correspondence
    \[
      A\act Y=\left[\mediumMATRIX{&1\\1}\mediumMATRIX{D\\&0}\mediumMATRIX{&1\\1}\right]\mediumMATRIX{D\\&D} = \mediumMATRIX{0\\&D}\mediumMATRIX{D\\&D}=\mediumMATRIX{0\\&D}.
    \]
    As such the covariance diagram is \textbf{maximally noncommuting:}
    \[
      \begin{tikzcd}[column sep=2cm]
        A=A\cap XX^*\rar\dar\drar[phantom]{\neq}
        & B\dar \\
        XX^* \rar
        & YY^*
      \end{tikzcd}:
      \TAB\TAB
      \begin{gathered}
        \ker\Bigl(~A\subset B\subset YY^*~\Bigr)=0,\\[2\jot]
        \ker\Bigl(~A\to XX^*\subset YY^*~\Bigr) = A.
      \end{gathered}
    \]
    So one needs to stay cautious about the covariance \textbf{for subcorrespondences:}\\
    In worst case one needs to verify a particular covariance by hand.
\end{BREAKEXM}

This finishes our section on kernel and cokernel morphisms on one hand, and the possible covariances on the other hand.
With this at hand we may now proceed to the gauge-equivariant representations and their classification\linebreak --- in other words the classification of relative Cuntz--Pimsner algebras.

\section{Relative Cuntz--Pimsner algebras}
\label{sec:RELATIVE-PIMSNER}
We introduce in this section the relative Cuntz--Pimsner algebras
and elaborate why and how these serve to classify the gauge-equivariant representations.

We begin with the Toeplitz algebra: That is the universal representation (more precisely the initial representation)
as such that any other representation uniquely factors via the universal one,
\begin{TIKZCD}[column sep=3cm]
  (X,A)\rar[equal]\dar & (X,A)\dar \\
  \TOEPLITZ (X,A)\rar[dashed] & (B,B).
\end{TIKZCD}
Next recall from the previous section that we may render any representation faithful by passing to the quotient correspondence
\begin{TIKZCD}[column sep=1.5cm]
  0\rar& (XK,K)\rar{\ker}\dar& (X,A)\rar\dar& \left(\frac{X}{XK},\frac{A}{K}\right) \rar\dar[dashed]&0 \\
  0\rar& 0\rar[swap]& (B,B)\rar[equal]& (B,B) \rar&0
\end{TIKZCD}
and recall from theorem \ref{THEOREM:KERNEL-COKERNEL-MORPHISMS} that any such quotient arises precisely for some invariant ideal (equivalently hereditary ideal)
\[
  K=\ker(A\to B)\ILEQ A:\TAB\TAB X^*KX\subset K\TAB (\iff KX\subset XK).
\]
On the other hand, we found in proposition
\ref{FAITHFUL-REPRESENTATIONS:MAXIMAL-COVARIANCE} that faithful representations have their covariance bounded from above by the maximal covariance,
\[
  0\subset\cov\left(\begin{tikzcd}[column sep=0.5cm]
    \left(\frac{X}{XK},\frac{A}{K}\right)\rar&B
  \end{tikzcd}\right) \subset \max\left(\tfrac{X}{XK},\tfrac{A}{K}\right).
\]
As such we may aim to classify representations by pairs of some invariant ideal (as possible covariance) and another bounded ideal (as possible covariance):
\begin{equation}
  \label{DEFINITION:KERNEL-COVARIANCE-PAIRS}
  \Bigl(~~K\ILEQ A:~X^*KX\subset K~~\Big|~~I\ILEQ\tfrac{A}{K}:~I\subset\max\left(\tfrac{X}{XK}\right)~~\Bigr)
\end{equation}
To handle this task, we may consider the class of representations with kernel and covariance at least a given pair of invariant and bounded ideal,
\[
  \begin{tikzcd}[column sep=small]
    (X,A)\rar& (B,B)
  \end{tikzcd}:\TAB\TAB
  K\subset \ker(A\to B),\TAB I\subset\cov\left(\tfrac{X}{XK}\to B\right).
\]
Note that any such class contains at least the trivial representations and furthermore that representations may well have larger kernel and covariance only (which we cannot exclude at this moment as we will explain below).\linebreak
The universal such representation defines now the relative Cuntz--Pimsner algebra for a given kernel--covariance pair as in \eqref{DEFINITION:KERNEL-COVARIANCE-PAIRS}:
\begin{TIKZCD}[column sep=3.5cm,row sep=2cm]
    (X,A)\rar & \left(\frac{X}{XK},\frac{A}{K}\right)
    \rar[equal]\dar\drar[phantom]
    {\begin{gathered}
      \scriptstyle K\subset \ker(A\to B)\\
      \scriptstyle I\subset \cov\left(\tfrac{\scriptscriptstyle X}{\scriptscriptstyle XK}\to B\right)
    \end{gathered}}
    & \left(\frac{X}{XK},\frac{A}{K}\right)\dar \\
    &\PIMSNER(K,I)\rar[dashed] & (B,B).
\end{TIKZCD}
As such we obtain an entire \enquote{2-dimensional lattice} of relative Cuntz--Pimsner algebras (as class of universal representations)
by first following along the lattice of cokernel morphisms (given by invariant ideals)
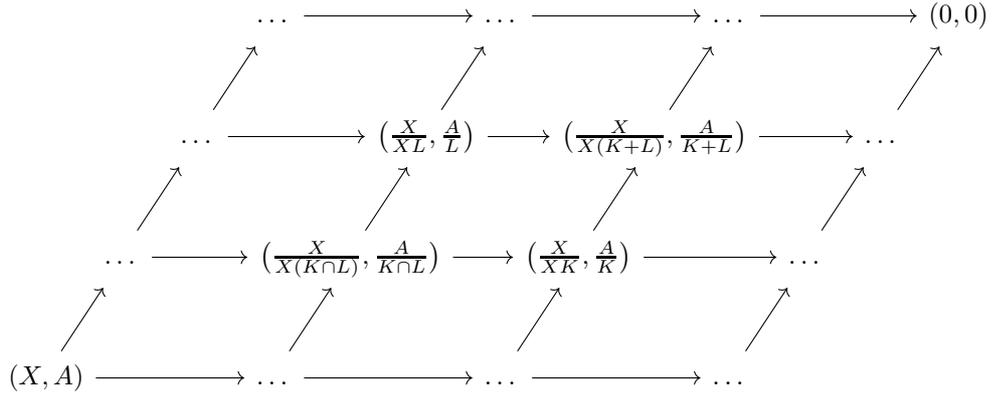
\[\hspace{0cm}\begin{tikzcd}[x=1cm,y=0.8cm]
  \node (1-1) at (0,0) {(X,A) \vphantom{ \big(\frac{A}{(K)}\big) }};
  \node (1-2) at (3,0) {\ldots\vphantom{ \big(\frac{A}{(K)}\big) }};
  \node (1-3) at (6,0) {\ldots\vphantom{ \big(\frac{A}{(K)}\big) }};
  \node (1-4) at (9,0) {\ldots\vphantom{ \big(\frac{A}{(K)}\big) }};
  \node (2-1) at (1,2) {\ldots\vphantom{ \big(\frac{A}{(K)}\big) }};
  \node (2-2) at (4,2) {\big( \frac{X}{X(K\cap L)},   \frac{A}{K\cap L\mathstrut} \big)};
  \node (2-3) at (7,2) {\big( \frac{X}{XK\mathstrut}, \frac{A}{K\mathstrut}       \big)};
  \node (2-4) at (10,2){\ldots\vphantom{ \big(\frac{A}{K\mathstrut}\big) }};
  \node (3-1) at (2,4) {\ldots\vphantom{ \big(\frac{A}{K\mathstrut}\big) }};
  \node (3-2) at (5,4) {\big( \frac{X}{XL\mathstrut}, \frac{A}{L\mathstrut}   \big)};
  \node (3-3) at (8,4) {\big( \frac{X}{X(K+L)},       \frac{A}{K+L\mathstrut} \big)};
  \node (3-4) at (11,4){\ldots\vphantom{ \big(\frac{A}{(K)}\big) }};
  \node (4-1) at (3,6) {\ldots\vphantom{ \big(\frac{A}{(K)}\big) }};
  \node (4-2) at (6,6) {\ldots\vphantom{ \big(\frac{A}{(K)}\big) }};
  \node (4-3) at (9,6) {\ldots\vphantom{ \big(\frac{A}{(K)}\big) }};
  \node (4-4) at (12,6){(0,0) \vphantom{ \big(\frac{A}{(K)}\big) }};
  \draw[->] (1-1.east) to (1-2.west);
  \draw[->] (1-2.east) to (1-3.west);
  \draw[->] (1-3.east) to (1-4.west);
  \draw[->] (2-1.east) to (2-2.west);
  \draw[->] (2-2.east) to (2-3.west);
  \draw[->] (2-3.east) to (2-4.west);
  \draw[->] (3-1.east) to (3-2.west);
  \draw[->] (3-2.east) to (3-3.west);
  \draw[->] (3-3.east) to (3-4.west);
  \draw[->] (4-1.east) to (4-2.west);
  \draw[->] (4-2.east) to (4-3.west);
  \draw[->] (4-3.east) to (4-4.west);
  \draw[->] ($(1-1.north)+(0.2,0)$) to ($(2-1.south)+(-0.25,0)$);
  \draw[->] ($(1-2.north)+(0.2,0)$) to ($(2-2.south)+(-0.25,0)$);
  \draw[->] ($(1-3.north)+(0.2,0)$) to ($(2-3.south)+(-0.25,0)$);
  \draw[->] ($(1-4.north)+(0.2,0)$) to ($(2-4.south)+(-0.25,0)$);
  \draw[->] ($(2-1.north)+(0.2,0)$) to ($(3-1.south)+(-0.25,0)$);
  \draw[->] ($(2-2.north)+(0.2,0)$) to ($(3-2.south)+(-0.25,0)$);
  \draw[->] ($(2-3.north)+(0.2,0)$) to ($(3-3.south)+(-0.25,0)$);
  \draw[->] ($(2-4.north)+(0.2,0)$) to ($(3-4.south)+(-0.25,0)$);
  \draw[->] ($(3-1.north)+(0.2,0)$) to ($(4-1.south)+(-0.25,0)$);
  \draw[->] ($(3-2.north)+(0.2,0)$) to ($(4-2.south)+(-0.25,0)$);
  \draw[->] ($(3-3.north)+(0.2,0)$) to ($(4-3.south)+(-0.25,0)$);
  \draw[->] ($(3-4.north)+(0.2,0)$) to ($(4-4.south)+(-0.25,0)$);
\end{tikzcd}\]
followed by the lattice of covariance ideals (on the chosen cokernel)
\begin{TIKZCD}[column sep=0.5cm,row sep=1cm]
  \ldots\rar&[0.5cm] \left(\frac{X}{XK},\frac{A}{K}\right) \rar\dar& \ldots \\
  &\hspace{-2.5cm}\TOEPLITZ\left(\frac{X}{XK},\frac{A}{K}\right)=\PIMSNER(K,0) \rar\dar&[-1cm]
  \mathstrut\ldots \rar\dar& \mathstrut\ldots    \rar\dar& \mathstrut\ldots \dar \\
  &\mathstrut\ldots \rar\dar& \PIMSNER(K,I\cap J) \rar\dar& \PIMSNER(K,I)   \rar\dar& \mathstrut\ldots \dar\\
  &\mathstrut\ldots \rar\dar& \PIMSNER(K,J)       \rar\dar& \PIMSNER(K,I+J) \rar\dar& \mathstrut\ldots \dar\\
  &\mathstrut\ldots \rar    & \mathstrut\ldots \rar& \mathstrut\ldots \rar&
  \PIMSNER(K,\max)=\PIMSNER\left(\frac{X}{XK},\frac{A}{K}\right)\hspace{-2cm}
\end{TIKZCD}
with Toeplitz algebras and absolute Cuntz--Pimsner algebras as extreme points.

We now explain the precise goals for the classification of relative Cuntz--Pimsner algebras.
We will first note that that relative Cuntz--Pimsner algebra come equipped with gauge-actions rendering the representation gauge-equivariant
(we explain this in more detail in the following section).\linebreak
As such the possible range of relative Cuntz--Pimsner algebras are the gauge-equivariant representations at most.
The surprising first goal of their classification now states that every gauge-equivariant representation indeed arises itself as a relative Cuntz-Pimsner algebra (and in particular every gauge-equivariant representation is itself a universal representation).
More precisely, consider first the induced representation rendering the representation faithful
\[
  \begin{tikzcd}[column sep=1cm]
    (X,A)\rar & \left(\frac{X}{XK},\frac{A}{K}\right) \rar[dashed]& (B,B)
  \end{tikzcd}:\TAB\TAB
  K=\ker(A\to B)
\]
followed by the covariance for the resulting representation
\[
  \begin{tikzcd}[column sep=1cm]
    \left(\frac{X}{XK},\frac{A}{K}\right) \rar& (B,B)
  \end{tikzcd}:\TAB\TAB
  I:=\cov\left(\begin{tikzcd}
    \left(\tfrac{X}{XK},\tfrac{A}{K}\right)\subset B
  \end{tikzcd}\right).
\]
To better illuminate the problem let us restrict the representation to its range, that is the operator algebra generated by (the image of) the correspondence,
\begin{TIKZCD}[column sep=1cm]
    (X,A)\rar & \left(\frac{X}{XK},\frac{A}{K}\right) \rar& \CSTAR(X\cup A)\subset B.
\end{TIKZCD}
With this description the first problem states that
\begin{TIKZCD}[column sep=1cm]
  (X,A)\rar & \left(\frac{X}{XK},\frac{A}{K}\right) \rar& \PIMSNER(K,I)=\CSTAR(X\cup A).
\end{TIKZCD}
Or put in other words, the pairs of invariant ideals as kernel and bounded ideals covariance exhaust the gauge-equivariant representations.
We will solve this problem via the familiar gauge-invariant uniqueness theorem.

The second goal is to determine that in fact every possible kernel and covariance arises itself as an actual kernel and covariance.
More precisely, consider any pair of invariant ideal as kernel and bounded ideal as covariance:
\[
  \left(\TAB
  \begin{gathered}
    K\ILEQ A:\\
    I\ILEQ A/K:
  \end{gathered}\TAB\TAB
  \begin{gathered}
    X^*KX\subset K\\
    I\subset\max(X/XK)
  \end{gathered}\TAB\right)
\]
Then there simply may be no such representation with precisely the given kernel and covariance ideal (so that not every such pair would arise in nature):
\begin{gather*}
  \begin{tikzcd}[column sep=1cm]
    (X,A)\rar & \left(\frac{X}{XK},\frac{A}{K}\right) \rar& (B,B)
  \end{tikzcd}:\\[2\jot]
  \Big(~\ker(A\to B),~\cov\left(\tfrac{X}{XK}\to B\right)~\Big) = (K,I)~?
\end{gather*}
Or put in other words, two possibly different pairs of kernel and covariance ideal could in principle lead to one and the same relative Cuntz--Pimsner algebra:
\[
  \PIMSNER(K,I)=\PIMSNER(K',I')\TAB\implies\TAB(K,I)=(K',I')~?
\]
We will however find the following relations (both nontrivial!):
\begin{gather*}
  \begin{tikzcd}[column sep=1cm]
    (X,A)\rar & \left(\frac{X}{XK},\frac{A}{K}\right) \rar& \PIMSNER(K,I)
  \end{tikzcd}:\\[2\jot]
  \Big(~\ker\left(\tfrac{A}{K} \to\PIMSNER(K,I)\right)=0~
  \Big|~\cov\left(\tfrac{X}{XK}\to\PIMSNER(K,I)\right)=I\Big)
\end{gather*}
and as such also the desired kernel--covariance pair
\begin{equation}\label{RELATIVE-PIMSNER:KERNEL-COVARIANCE}
  \Big(~\ker\left(A\to\PIMSNER(K,I)\right)=K~
  \Big|~\cov\left( \tfrac{X}{XK}\to\PIMSNER(K,I)\right)=I\Big).
\end{equation}
For this the Fock representation will come into play:
Its concrete representation allows us to actually compute its kernel and covariance and so to verify the desired relations.
As such each possible pair of invariant ideal (as kernel) and bounded ideal (as covariance) arises itself as actual kernel and covariance.
Summarizing these goals, the kernel--covariance pairs completely parametrize the entire lattice of gauge-equivariant representations.
In other words, the lattice of relative Cuntz--Pimsner algebras (as schematically given above) classifies the entire lattice of gauge-equivariant representations.

With this at hand, we then further investigate the connecting morphisms between relative Cuntz--Pimsner algebras.
More precisely, we will find that our parametrisation given by kernel--covariance pairs defines a lattice isomorphism in the sense that
\[
  (~K\subset L~|~\mathquote{I\subset J}~) \TAB\iff\TAB \PIMSNER(K,I)\leq\PIMSNER(L,J)
\]
where the latter denotes the factorization (as common notation):
\begin{TIKZCD}[column sep=3cm, row sep=1cm]
  (X,A)\rar[equal]\dar & (X,A)\dar\\
  \PIMSNER(K,I) \rar[dashed] &\PIMSNER(L,J).
\end{TIKZCD}
This has been observed already by Katsura in his 2007 article on gauge-invariant ideals using however instead so-called T-pairs (and further O-pairs):\linebreak
We will unravel these as nothing but transformed versions of our kernel--covariance pairs and as such give a natural interpretation for such pairs.\linebreak
Following we further give precise descriptions for when connecting morphisms exist between cokernel strands (as from where and to where)
\begin{gather*}
  \begin{tikzcd}[column sep=1cm]
    \PIMSNER(K,I=~??)\rar[dashed]& \PIMSNER(L,J)
  \end{tikzcd},\TAB
  \begin{tikzcd}[column sep=1cm]
    \PIMSNER(K,I)\rar[dashed]& \PIMSNER(L,\max)~??
  \end{tikzcd}
\end{gather*}
which generalize results from Katsura.
Together we thus obtain a plethora of connecting morphisms between cokernel strands such as
\begin{TIKZCD}[column sep=2cm,row sep=1cm]
  \rar
  & \left(\frac{X}{XK},\frac{A}{K}\right) \rar\dar
  & \left(\frac{X}{XL},\frac{A}{L}\right) \rar\dar
  & \mathstrut\\
  \rar
  & \TOEPLITZ\left(\frac{X}{XK},\frac{A}{K}\right) \dar\rar
  & \TOEPLITZ\left(\frac{X}{XL},\frac{A}{L}\right) \dar\rar & \mathstrut\\
  \rar[dashed]\drar[dashed]
  & \PIMSNER(K,I)\dar\rar[dashed]
  & \PIMSNER(L,J)\dar\rar[dashed] \drar[dashed] &\mathstrut\\
  & \PIMSNER\left(\frac{X}{XK},\frac{A}{K}\right) \urar[dashed]
  & \PIMSNER\left(\frac{X}{XL},\frac{A}{L}\right) &\mathstrut
\end{TIKZCD}
and note that plenty of connecting morphisms also will be missing.\\
For this we will give some further examples from graph algebras to illuminate the lack of connecting morphism.
With this in mind and without further ado, we now get to the class of gauge-equivariant representations:

\section{Gauge actions: Fourier spaces}
\label{sec:GAUGE-ACTIONS}

We begin with the following observation to motivate gauge-equivariant representations:
For every fixed complex number on the torus we may consider the automorphism which rotates the correspondence
\[
  z\in \TORUS:\TAB\TAB
  \begin{tikzcd}[column sep=0.5cm]
    (X,A)\rar&(X,A)
  \end{tikzcd}:\TAB z\act x=zx,\TAB z\act a=a
\]
which together define a circle action $\TORUS\act (X,A)$.\\
Clearly these do not affect the kernel of representations as for
\[
  \begin{tikzcd}[column sep=1cm]
    (X,A)\rar& B
  \end{tikzcd}:\TAB\TAB
  \ker(\begin{tikzcd}[column sep=0.5cm]
    A\rar{1}& A\rar& B
  \end{tikzcd})
  =\ker(\begin{tikzcd}[column sep=0.5cm]
    A\rar& B
  \end{tikzcd}).
\]
Similarly they do not affect the covariance as they act trivially on compacts,
\begin{gather*}
  zz^*=1:\begin{tikzcd}[column sep=1cm]
    XX^*\rar& XX^*
  \end{tikzcd}:\TAB x(zz^*)y^*= xy^*:\\[2\jot]
  \begin{tikzcd}
    A\cap XX^*\rar[equal]\dar\drar[phantom]{=}& A\cap XX^*\subset A\rar\dar\drar[phantom]{=?}& B\dar[equal]\\
    XX^*\rar[equal]& XX^*\rar& B.
  \end{tikzcd}
\end{gather*}
As such every relative Cuntz--Pimsner algebra (as universal representation for some relations) admits a unique gauge action which renders its representation gauge-equivariant (think in terms of generators and relations):
\[
  \exists\TORUS\act\PIMSNER(K,J):\TAB\TAB
  \begin{tikzcd}[column sep=2cm]
    \Big(\TORUS\act(X,A)\Big)\rar&\Big(\TORUS\act\PIMSNER(K,I))\Big).
  \end{tikzcd}
\]
As such at best we may hope to classify the relative Cuntz--Pimsner algebras amongst those representations which come along with a gauge-action rendering the representation equivariant. More precisely, that is written out
\begin{gather*}
  \begin{tikzcd}[column sep=2cm]
    (\phi,\tau):\Big(\TORUS\act(X,A)\Big)\rar&\Big(\TORUS\act B\Big):
  \end{tikzcd}\\[2\jot]
  \begin{aligned}
    z\act\tau(x)&=\tau(z\act x)=\tau(zx),\\
    z\act\phi(a)&=\phi(z\act a)=\phi(a).
  \end{aligned}
\end{gather*}
Equivalently these are the conventional gauge-equivariant representations of operator algebras (from any of the preceeding relative Cuntz--Pimsner algebras)
\begin{gather*}
  \pi:\begin{tikzcd}[column sep=1cm]
    \Big(\TORUS\act\PIMSNER(K,I)\Big)\rar& \Big(\TORUS\act B\Big)
  \end{tikzcd}:\TAB z\act\pi(\BLANK)=\pi(z\act\BLANK).
\end{gather*}
Consider now the relative Cuntz--Pimsner algebras (as universal representations) which allow factorizations for our gauge-equivariant representation:
\begin{gather*}
  \begin{tikzcd}[column sep=1cm]
    (X,A)\rar&\PIMSNER(K,I)\rar[dashed]& B
  \end{tikzcd}\\[2\jot]
  \hspace{-1cm}\iff\TAB\Big(~K\subset\ker(A\to B)~\Big|~I\subset\cov\left( \tfrac{X}{XK}\to B\right)~\Big)
\end{gather*}
As such --- if the gauge-equivariant representation has a chance of being a universal representation for some kernel--covariance pair --- then certainly the best chance is given by the kernel--covariance pair for the representation itself:
\begin{gather*}
  (K,I)=\Big(~\ker(A\to B)~\Big|~\cov\left( \tfrac{X}{XK}\to B\right)~\Big):\\[3\jot]
  \begin{tikzcd}[column sep=1cm]
    (X,A)\rar&\PIMSNER(K,I)\rar[equal]&\CSTAR(X\cup A)\subset B
  \end{tikzcd}?
\end{gather*}
This is what the gauge-equivariant uniqueness theorem will establish in the next section and so also the first goal in the classification of relative Cuntz--Pimsner algebras. So let us prepare ourselves a bit more for this by taking a closer look at gauge-equivariant representations and their Fourier spaces.

It is well known that every operator algebra equipped with a circle action
(such as our gauge-equivariant representations) comes along with Fourier spaces
\[
  B(n\in\INTEGERS)=\Big\{~b~\Big|~\Big(b(z):=z\act b\Big)=z^nb~\Big\}\subset B
\]
and in particular its fixed point algebra
\[
  B(n=0)=\Big\{~b~\Big|~\Big(b(z):=z\act b\Big)\equiv b~\Big\}\subset B
\]
which define a Fell bundle over the integers
\[
  B(m)B(n)\subset B(m+n),\TAB B(n)^*=B(-n)
\]
together with a conditional expectation onto its fixed point algebra
\[
  E:\begin{tikzcd}[column sep=0.5cm]
    B\rar& B(0)
  \end{tikzcd}:\TAB E(b)=\int_\TORUS \Big(b(z)=z\act b\Big)\dd z
\]
and more generally with projections onto any of its Fourier spaces
\[
  E_n:\begin{tikzcd}[column sep=0.5cm]
    B\rar& B(n)
  \end{tikzcd}:\TAB E_n(b)=\int_\TORUS z^{-n}b(z)\dd z.
\]
Recall that conditional expectations given by averaging are automatically faithful since the state space separates the positive elements
(and since averaging runs as Bochner integral):
\begin{gather*}
  b\geq0\implies \Big(b(z)=z\act b\Big)\geq0\implies SB\Big(b(z)=z\act b\Big)\geq0:\\[3\jot]
  E(b\geq0) = 0 \implies SB(E(b))=SB\left(\int b(z)\dd z\right)=\int SB(b(z))\dd z=0\\
  \implies SB\Big(b(z)=z\act b\Big)\equiv0\implies SB\Big(b=1\act b\Big)=0 \implies b=0.
\end{gather*}
We meanwhile note that the Fourier spaces densely span the operator algebra by \cite[proposition 2.5]{EXEL-CIRCLE-ACTIONS} (which seems a rather intuitive yet nontrivial relation)
\[
  B= \overline{\Big(\ldots + B(-1)+B(0)+B(1)+\ldots\Big)}=\sum_n B(n)
\]
and we encourage the reader to have a look into the beautiful proof by Exel:
It does not require any Cesaro approximations and instead only invokes the elementary isomorphism (as the basic version of Coburn)
\begin{gather*}
  \CSTAR(\INTEGERS) = \CSTAR(u^*u=1=uu^*) = \CONTINUOUS(~\sigma u\mid u^*u=1=uu^*)=\CONTINUOUS(\TORUS)\\[2\jot]
  \implies\CSTAR(\INTEGERS)=\overline{\Big(\ldots+ \COMPLEX u^* +\COMPLEX e +\COMPLEX u +\ldots\Big)}=\CONTINUOUS(\TORUS).
\end{gather*}
\textbf{A short digression:} One may meanwhile wonder how it may be possible that any element may be approximated by Fourier sums while the series of its Fourier coefficients only converges in Cesaro mean (and generally diverges in norm).
The answer to this question becomes evident when replacing for instance the torus with an open disk and the trigonometric with polynomial sums
\[
  \{\ldots, 1/z,1,z,z^2,\ldots\}\subset \CONTINUOUS(\TORUS)\TAB\rightsquigarrow\TAB \{1,z,z^2,\ldots\}\subset \CONTINUOUS(\DISK)
\]
It is well-known here that any continuous function may be approximated in norm by polymoials (by Stone--Weierstrass).
On the other hand, the convergent power series correspond to the class of Taylor analytic functions
\[
  f(z)=\sum_n a_n z^n\iff f\in\CONTINUOUS^\infty(\DISK)
\]
The difference lies in the fact that a polynomial approximation of continuous functions generally changes each of the coefficient in the sequence --- also the previously already set coefficients!
As such one may think of the convergent Fourier series as a \textbf{generalization of analytic functions.}

We now further restrict our attention on the range of the representation (also since the ambient rest does not reflect the correspondence): More precisely, that is the operator algebra generated by (the image of) the correspondence,
\[
  \begin{tikzcd}[column sep=1cm]
    (X,A)\rar& \CSTAR(X\cup A)\subset B
  \end{tikzcd}
  \TAB\rightsquigarrow\TAB
  B=\CSTAR(X\cup A).
\]
Recall that a correspondence (and so its image under the representation) satisfies its defining inclusions as explained in more detail in section \ref{sec:CORRESPONDENCES}:
\[
  X^*X\subset A,\TAB XA\subset X,\TAB AX\subset X.
\]
As such the the range (as generated operator algebra) admits \enquote{a fine structure}
\begin{gather*}
  \CSTAR(X\cup A)=\CLOSESPAN\Big(\ldots + (X^*+ XX^*X^* +\ldots)+\\[2\jot]
                 + (A+ XX^* + XXX^*X^*+\ldots) + (X + XXX^*+ \ldots)+\ldots\Big)
\end{gather*}
which we organised in groups according to the Fourier spaces they belong to.\\
As such we also obtain \enquote{a fine structure} for the fixed point algebra
\begin{gather*}
  B(n=0)= E\Big(B=\CSTAR(X\cup A)\Big)
  = \CLOSESPAN\Big(A + XX^* + XXX^*X^* + \ldots\Big)
\end{gather*}
which allows us to deduce the gauge-invariant uniqueness theorem by induction.\\
Note further that for any correspondence (including possibly degenerate ones!)
\begin{gather*}
  X^*X\subset A,\TAB XA\subset X,\TAB AX\subset X\\[2\jot]
  \implies XA=X,\TAB XXA=XX,\TAB\ldots
\end{gather*}
As such the Fourier spaces further satisfy
\begin{gather*}
  B(n\geq 1)= E_n\Big(B=\CSTAR(X\cup A)\Big)
  = \CLOSESPAN\Big(X^n + X^nXX^* +  \ldots\Big)\\
  = X^n\Big(A + XX^* + XXX^*X^* + \ldots\Big) = X^n B(0)
\end{gather*}
where we left the linear spans and closures implicit for more pointy statements,\\
and by involution also for negative Fourier spaces and as such altogether
\[
  B(n\geq1)=X^n B(0),\TAB\TAB B(n\leq-1) = B(0)X^{-n}.
\]
As a consequence, the Fell bundle (restricted on the range) further satisfies
\begin{gather*}
  B(n\geq1)=B(1)^n:\TAB\TAB B(n)=X^nB(0)\subset B(1)^nB(0)\subset B(1)^n\subset B(n)
\end{gather*}
and similarly for its negative Fourier spaces,
and as such defines altogether a semi-saturated Fell bundle (see for instance \cite[proposition 4.8]{EXEL-CIRCLE-ACTIONS}):
\begin{align*}
  B(m\geq0)B(n\geq0)&=B(m+n),\\
  B(m\leq0)B(n\leq0)&=B(m+n).
\end{align*}
These relations are quite useful for bipartite inflations as introduced in \cite{MUHLY-PASK-TOMFORDE}.\\
With the knowledge about Fourier spaces at hand, we may now get back to the classification problem from above:
\begin{gather*}
  (K,I)=\Big(~\ker(A\to B)~\Big|~\cov\left( \tfrac{X}{XK}\to B\right)~\Big):\\[3\jot]
  \begin{tikzcd}[column sep=1cm]
    (X,A)\rar&\PIMSNER(K,I)\rar[equal]&\CSTAR(X\cup A)\subset B
  \end{tikzcd}?
\end{gather*}
At first we may pass to the quotient correspondence
\begin{TIKZCD}[column sep=1cm]
  (X,A)\rar&\left(\frac{X}{XK},\frac{A}{K}\right)\rar[dashed]& \PIMSNER(K,I)\rar& B
\end{TIKZCD}
to obtain an honest embedding since
\[
  K=\ker(A\to B)\TAB\implies\TAB A/K\subset B\TAB\implies\TAB X/XK\subset B.
\]
As such we may assume that our correspondence lies faithfully in the ambient operator algebra (upon replacing the original with the quotient correspondence)
\begin{equation}\label{REDUCTION:FAITHFUL-REPRESENTATION}
  \left(\tfrac{X}{XK},\tfrac{A}{K}\right)\rightsquigarrow(X,A)
  ~\implies~
  \PIMSNER(K,I)\rightsquigarrow\PIMSNER(0,I)
  ~\implies~
  X\subset B,~ A\subset B
\end{equation}
and at the same time restrict our attention to its range:
\[
  \CSTAR(X\cup A)\subset B\TAB \rightsquigarrow\TAB\CSTAR(X\cup A)=B.
\]
We may thus think of both ambient algebras as faithful completions for the correspondence --- but under a priori possibly different norm topologies:
\begin{TIKZCD}[column sep=1cm]
  \PIMSNER(X;I)=\CSTAR(X\cup A)\rar[hookleftarrow] &X\cup A\rar[hook] &\CSTAR(X\cup A)=B
\end{TIKZCD}
As another valuable perspective, one may also think of the above problem as gauge-equivariant envelopes (comparable to the maximal and minimal \mbox{\CSTAR-envelope} of non-selfadjoint operator algebras):
\[
  \begin{tikzcd}[column sep=0.8cm]
    \TOEPLITZ(X)\rar&\ldots\rar&\PIMSNER(X;I)\rar& Q\rar&\ldots
  \end{tikzcd}:\TAB\TAB (X,A)\subset (\TORUS\act Q).
\]
The author would like to thank Elias Katsoulis for bringing closer this idea.\\
Next as our representation is gauge-equivariant it restricts in particular to fixed point algebras
and further commutes with conditional expectations:
\[
  \begin{tikzcd}
    \PIMSNER(X;I)\rar & B\\
    \PIMSNER(X;I)(0)\rar\uar & B(0)\uar
  \end{tikzcd}
  \TAB\TAB
  \begin{tikzcd}
    \PIMSNER(X;I)\rar\dar & B\dar\\
    \PIMSNER(X;I)(0)\rar& B(0)
  \end{tikzcd}
\]
As such it suffices to verify the above equality on fixed point algebras:
\begin{equation}\label{REDUCTION:FIXED-POINT-ALGEBRAS}
  \pi:\begin{tikzcd}[column sep=1cm]
    \PIMSNER(X;I)(0)\rar[equal]&B(0)
  \end{tikzcd}
  \TAB\implies\TAB
  \pi:\begin{tikzcd}[column sep=1cm]
    \PIMSNER(X;I)\rar[equal]&B
  \end{tikzcd}
\end{equation}
Indeed this follows as standard argument from the faithful conditional expectation (sufficiently for the relative Cuntz--Pimsner algebra):
\begin{gather*}
  \pi(a)=0\implies\pi(a^*a)=0\implies \pi(E(a^*a))=E(\pi(a^*a))=0\\
  \implies E(a^*a)=0\implies a^*a=0\implies a=0.
\end{gather*}
Another interesting less commonly known \textbf{argument due to Exel} basically exploits that continuous functions, whose Fourier coefficients all vanish, already vanish identically themselves (see \cite[proposition 2.5 and 2.9]{EXEL-CIRCLE-ACTIONS}):
\begin{gather*}
  E_n(a)E_n(a)^*\in \PIMSNER(X;I)\Big(n-n=0\Big):\\[2\jot]
  \pi(a)=0\implies E_n(\pi a)E_n(\pi a)^*=\pi\Big(E_n(a)E_n(a)^*\Big)\equiv0\\
  \implies E_n(a)E_n(a)^*\equiv 0 \implies E_n(a)\equiv0\implies  a=0.
\end{gather*}
For the fixed point algebra we however found the \enquote{fine structure}
\[
  \PIMSNER(X;I)(n=0) = \CLOSESPAN\Big(A + XX^* + XXX^*X^* + \ldots\Big)
\]
so the fixed point algebras arises as an increasing union (inductive limit)
\[
  \PIMSNER(X;I)(n=0)=\CLOSURE\left(\bigcup_N\CLOSESPAN\Big(A+ XX^*+\ldots + X^NX^{-N}\Big)\right).
\]
As such we may verify the faithfulness by induction along $n\in\NATURALS$:
\begin{equation}\label{INDUCTION-PROBLEM}
  \begin{gathered}
    \begin{tikzcd}[column sep=1.5cm]
      \PIMSNER(X;I)\supset \CLOSESPAN(A+ XX^*+\ldots + X^nX^{-n})\rar[hook] & B
    \end{tikzcd}?
  \end{gathered}
\end{equation}
The point here is not that each summand embeds separately (which is trivial)
\[
  \begin{tikzcd}[column sep=0.5cm]
    A\rar[hook]& B,
  \end{tikzcd}\TAB
  \begin{tikzcd}[column sep=0.5cm]
    XX^*\rar[hook]& B,
  \end{tikzcd}\TAB
  \begin{tikzcd}[column sep=0.5cm]
    XXX^*X^*\rar[hook]& B,
  \end{tikzcd}\TAB\ldots
\]
but instead that the summands will have plenty of correlation among each other (namely precisely the amount of covariance) which causes lots of sums to collapse within the representation such as the covariance itself,
\begin{gather*}
  \cov(X\to B) =\ker\left[\left(\begin{tikzcd}[column sep=0.5cm,row sep=0.5cm]
    \cov(X\to B)\rar& B\dar[equal]\\& B
  \end{tikzcd}\right) -
  \left(\begin{tikzcd}[column sep=0.5cm,row sep=0.5cm]
    \cov(X\to B)\dar \\ XX^*\rar& B
  \end{tikzcd}\right)\right]\\[3\jot]
  \implies\bigg\{~(a)+(k=-a)\in A+ XX^*~\bigg|~a\in \cov(X\to B)~\bigg\}\subset \ker\left(\begin{tikzcd}[column sep=0.5cm]
    A + XX^*\rar& B
  \end{tikzcd}\right).
\end{gather*}
In fact, proposition \ref{COVARIANCE:SUBCORRESPONDENCES} tells us that the covariance precisely captures this kernel. As an interesting question to the encouraged reader: can you figure out why?

Before we proceed, we would like to note that the idea for the above induction and its proof are due to Evgenios Kakariadis as in \cite{KAKARIADIS-GAUGE-THEOREM}
and note that this simplifies the original proof by Katsura substantially!
We will not go further into how these approaches compare (since this would not benefit our current work) but we would like to note that Katsuras approach analysing cores may be quite valuable after all in the more general context of product systems and so we refer the curious reader to \cite[section 5]{KATSURA-SEMINAL} for comparison.\linebreak
With this we may now proceed to the gauge-invariant uniqueness theorem.

\section{Uniqueness theorem}
\label{sec:GAUGE-THEOREM}

We may now state and proof our version of the gauge-invariant uniqueness theorem for arbitrary gauge-equivariant representations which arise as relative Cuntz--Pimsner algebra along kernel--covariance pairs:

\begin{THM}[Gauge-invariant uniqueness theorem: The general version]
  \label{UNIQUENESS-THEOREM}
  Consider a gauge-equivariant representation (as in the previous section)
  \begin{TIKZCD}[column sep=2cm]
    \Big(\TORUS\act(X,A)\Big)\rar&\Big(\TORUS\act B\Big)
  \end{TIKZCD}
  and choose the kernel--covariance pair for the representation
  \[
    (K,I)=\Big(~\ker(A\to B)~\Big|~\cov\left( \tfrac{X}{XK}\to B\right)~\Big).
  \]
  Then the quotient from the relative Cuntz--Pimsner algebra is faithful:
  \begin{TIKZCD}[column sep=1cm]
    (X,A)\rar&\PIMSNER(K,I)\rar[equal]&\CSTAR(X\cup A)\subset B.
  \end{TIKZCD}
  Thus each gauge-equivariant defines a relative Cuntz--Pimsner algebra.\\
  On the other hand, each relative Cuntz--Pimsner algebra defines itself a gauge-equivariant representations.
  As such the gauge-equivariant representations agree with relative Cuntz--Pimsner algebras (the range of kernel--covariance pairs).
\end{THM}

We have basically already proven the theorem:
Indeed we have already reduced the general version to the case of faithful representations\linebreak(see the previous section)
and the faithful case is a classical result due to Katsura from \cite{KATSURA-SEMINAL} resp.~the simplified version due to Kakariadis in \cite{KAKARIADIS-GAUGE-THEOREM}.\linebreak
For convenience of the reader we review the simplified version:

We begin for this with a couple of observations and useful relations.\linebreak
At first consider the increasing subalgebras from the induction problem \eqref{INDUCTION-PROBLEM} (which exhaust the fixed point algebra):
\[
  \CLOSESPAN(A+ XX^*+\ldots + X^nX^{-n})\subset\PIMSNER(X;I)
\]
In order to simplify their induction we need to \textbf{address a technical detail}:
Note that while the sum of ideals is already closed, this fails in general for the sum of subalgebras,
\[
  A\subset B,\TAB A'\subset B:\TAB\TAB (A+ A')\subset\overline{(A+A')}.
\]
However the sum of an ideal and a subalgebra is closed nevertheless:
\[
  A\subset B,\TAB J\ILEQ B:\TAB\TAB (A+J)=\overline{(A+J)}.
\]
\textbf{Quick proof for an algebra and ideal (as for pairs of ideals):}\\
Consider the short exact sequences of possibly incomplete algebras,
\begin{TIKZCD}[column sep=1.5cm,row sep=1cm]
  0\rar& J\rar\dar[equal]&          (A+J)\rar\dar\drar[phantom]{\pi}&  (A+J)/J \rar\dar&0 \\
  0\rar& J\rar           & \overline{(A+J)}\rar    & \overline{(A+J)}/J \rar&0.
\end{TIKZCD}
Note that the morphism between quotients defines an embedding since
\[
  J\cap(A+J)= J\TAB\implies\TAB (A+J)/J\subset\overline{(A+J)}/J.
\]
Since the range of \STARHOMS\ is always closed we obtain also
\[
  \pi(A+J)=\pi(A)=\overline{\pi(A)}=\overline{\pi(A+J)}=\pi\Big(\overline{A+J}\Big).
\]
As such the quotients agree and so (by basic homological algebra)
\begin{equation}
  \label{CLOSED-SUM:SUBALGEBRA-PLUS-IDEAL}
  (A+J)/J =\overline{(A+J)}/J\TAB\implies\TAB(A+J)=\overline{(A+J)}.
\end{equation}
So the sum of an algebra and ideal \textbf{defines a closed subalgebra.}\\
With this at hand we find in our case (using the obvious inclusion)
\[
  AXX^*\subset XX^* \TAB\implies\TAB \CLOSESPAN(A+ XX^*) = A + \CLOSESPAN(XX^*)
\]
and similarly on larger sums
\[
  \CLOSESPAN\Big(A + XX^* + \ldots + X^nX^{-n}\Big) = A + \CLOSESPAN\Big(XX^* + \ldots + X^nX^{-n}\Big).
\]
This finishes our \textbf{discussion on the technical detail.}\\
With this the induction problem asks for the kernel
\[
  \ker\bigg(\begin{tikzcd}[column sep=1cm]
    \PIMSNER(X;I)\supset A+\CLOSESPAN\Big(XX^* + \ldots + X^nX^{-n}\Big)\rar&B
  \end{tikzcd}\bigg)
\]
and so for the coefficient algebra that intersects compact operators
\[
  A\cap\Big(XX^*+\ldots+X^nX^{-n}\Big)\subset B
\]
where we drop from now on the closed linear span for more pointy statements.\\
We would like to better understand this intersection.
For this we first discuss the following result, which gives an interesting algebraic description for non-commutative Cartan subalgebras from \cite{EXEL-NC-CARTAN} as the nondegenerate ones:
\begin{LMM}\label{NONDEGENERATE-SUBALGEBRAS}
  A subalgebra contains some approximate identity for the ambient algebra if and only if the \textbf{subalgebra is nondegenerate:}
  \[
    A\subset B:\TAB\TAB AB=B\TAB\iff\TAB
    (1-e)B\to0~\text{for}~A\ni e\to1.
  \]
  On the other hand, the nondegeneracy holds equivalently if the subalgebra reaches the entire ambient algebra hereditarily (with implicit Cohen--Hewitt)
  \[
    AB=B\TAB\iff\TAB \her(A)=ABA=B.
  \]
  Further, the subalgebra remains nondegenerate on intermediate subalgebras:
  \[
    A\subset B_0\subset B:\TAB\TAB AB=B\TAB\implies\TAB AB_0=B_0.
  \]
  The analogous statements also hold when replacing the subalgebra and the ambient algebra by any pair of operator algebras (after suitable modifications).
\end{LMM}
\begin{proof}
  The forward direction is obvious since for any approximate unit
  \[
    A\ni e\to1:\TAB\TAB(1-e)B=(1-e)AB\to 0
  \]
  where we implicitly use Cohen--Hewitt as usual.\\
  The converse is also obvious as evidentily
  \[
    (1-e)B\to 0\TAB\implies\TAB B\subset AB.
  \]
  Regarding the hereditary subalgebra we have:
  \begin{gather*}
    AB=B\TAB(\implies B=B^*=(AB)^*=B^*A^*=BA)\\[2\jot]
    \implies\TAB B=(AB)=A(BA)=ABA\TAB\implies\TAB B=ABA\subset AB
  \end{gather*}
  Regarding intermediate algebras we have:
  \begin{TIKZCD}[row sep=small]
    AB=B\TAB \rar[Leftrightarrow]&
    \TAB(1-e)B\to0~\text{for}~A\ni e\to1
    \dar[Rightarrow]\\
    AB_0=B_0\TAB \rar[Leftrightarrow]&
    \TAB(1-e)B_0\to0~\text{for}~A\ni e\to1
  \end{TIKZCD}
  For arbitrary pairs of operator algebras (instead of a subalgebra):\\
  One needs to replace the equality by an inclusion and an additional closure.
\end{proof}
With the nondegeneracy in mind we may now unravel the intersection above:
For this we first note that the compact operators are nondegenerate in the sum of higher order compact operators,
\[
  \Big(XX^* + \ldots + X^nX^{-n}\Big)=XX^*\Big(XX^* + \ldots + X^nX^{-n}\Big).
\]
Indeed we may simply pull out a rabbit from each summand (using Blanchard):
\[
  XX^*=(XX^*)XX^*,\TAB\TAB XXX^*X^*=(XX^*)XXX^*X^*,\TAB\TAB\ldots
\]
As a consequence we obtain the inclusion (similarly as in lemma \ref{NONDEGENERATE-SUBALGEBRAS}):
\[
  A\cap\Big(XX^* + \ldots + X^nX^{-n}\Big)\subset \Big(A\cap(XX^* + \ldots + X^nX^{-n})\Big)XX^*.
\]
The latter however lies inside the algebra of compact operators since
\[
  \Big(A\cap(XX^* + \ldots + X^nX^{-n})\Big)XX^*\subset A XX^*\subset XX^*
\]
and as such inside the covariance for the representation (see proposition \ref{COVARIANCE:SUBCORRESPONDENCES}).\\
As the covariance cannot decrease we obtain the combined relation:
\[
  A\cap\Big(XX^* + \ldots + X^nX^{-n}\Big)=\cov(X\to B) = \cov\Big(X\to\PIMSNER(X;I)\Big).
\]
In particular the kernel arises as sum of compact operators:
\begin{gather*}
  \ker\left(\begin{tikzcd}[column sep=1cm]
  A+ \Big(XX^*+\ldots + XX^{n}X^{-n}X^*\Big)\rar & B
  \end{tikzcd}\right)\\
  \subset XX^* + \Big(XX^*+\ldots + XX^{n}X^{-n}X^*\Big)\subset \PIMSNER(X;I).
\end{gather*}
So far our introductory observations and relations.
With these in mind we may now get to the proof of the gauge-equivariant uniqueness theorem.

\begin{proof}[Proof of theorem \ref{UNIQUENESS-THEOREM} (Gauge-invariant uniqueness theorem):]
  We have already reduced the general version to the case of faithful representations in \eqref{REDUCTION:FAITHFUL-REPRESENTATION}:
  \[
    \left(\tfrac{X}{XK},\tfrac{A}{K}\right)\rightsquigarrow(X,A)\TAB\implies\TAB\PIMSNER(K,I)\rightsquigarrow\PIMSNER(0,I)
    \TAB\implies\TAB X\subset B,\TAB A\subset B
  \]
  Furthermore the problem reduces to fixed point algebras as in \eqref{REDUCTION:FIXED-POINT-ALGEBRAS}
  \[
    \begin{tikzcd}[column sep=1cm]
      \PIMSNER(X;I)(0)\rar[equal]&B(0)
    \end{tikzcd}
    \TAB\implies\TAB
    \begin{tikzcd}[column sep=1cm]
      \PIMSNER(X;I)\rar[equal]&B
    \end{tikzcd}
  \]
  which may be solved as an induction along its fine structure as in \eqref{INDUCTION-PROBLEM}:
  \begin{gather*}
    \begin{tikzcd}[column sep=1.5cm]
      \PIMSNER(X;I)\supset A+ \Big(XX^*+\ldots + X^nX^{-n}\Big)\rar[hook] & B
    \end{tikzcd}?
  \end{gather*}
  With this at hand we may now begin the \textbf{proof by induction from \cite{KAKARIADIS-GAUGE-THEOREM}:}\linebreak
  The base case simply states the inclusion that we are already well aware of,
  \begin{TIKZCD}[column sep=1cm]
    \PIMSNER(X;I)=\CSTAR(X\cup A)\rar[hookleftarrow] &A\rar[hook] &\CSTAR(X\cup A)=B.
  \end{TIKZCD}
  Before we begin with the induction combine the observations from above
  Consider now the induction step (while assuming the induction hypothesis):
  \[
    \ker\left(\begin{tikzcd}
    A + \Bigl(XX^*+\ldots + X^{n+1}X^{-n-1}\Bigr)  \rar&B
    \end{tikzcd}\right)=0~?
  \]
  We may reduce this problem to the induction hypothesis via compression
  \begin{gather*}
    X^*\ker\left(\begin{tikzcd}[column sep=small]
    A + \Bigl(XX^*+\ldots + X^{n+1}X^{-n-1}\Bigr)  \rar&B
    \end{tikzcd}\right)X \\[\jot]
    \subset
    \ker\left(\begin{tikzcd}[column sep=small]
    X^*X+ X^*\Bigl(XX^*+\ldots + X^{n+1}X^{-n-1}\Bigr)X  \rar&B
    \end{tikzcd}\right) \\[\jot]
    \subset
    \ker\left(\begin{tikzcd}[column sep=small]
    A+ \Bigl(XX^*+\ldots + X^{n}X^{-n}\Bigr)  \rar&B
    \end{tikzcd}\right)=0
  \end{gather*}
  from which we infer that the kernel vanishes for the compression
  \[
    XX^*\ker\left(\begin{tikzcd}[column sep=0.5cm]
    A + \Big(XX^*+\ldots + X(X^{n}X^{-n})X^*\Big)\rar & B
    \end{tikzcd}\right)XX^*=0.
  \]
  We however found that the kernel lies inside the sum of compact operators:
  \begin{gather*}
    \ker\left(\begin{tikzcd}[column sep=1cm]
    A+ \Big(XX^*+\ldots + XX^{n}X^{-n}X^*\Big)\rar & B
    \end{tikzcd}\right)\\
    \subset XX^* + \Big(XX^*+\ldots + XX^{n}X^{-n}X^*\Big)\subset \PIMSNER(X;I).
  \end{gather*}
  At the same time, the compact operators define a nondegenerate subalgebra within higher order compact operators and as such have trivial annihilator:
  \begin{gather*}
    \Big(XX^*+\ldots + XX^{n}X^{-n}X^*\Big) = XX^*\Big(XX^*+\ldots + XX^{n}X^{-n}X^*\Big)\\[2\jot]
    \hspace{-1cm}\implies \TAB (XX^*)^\perp\cap\Big(XX^*+\ldots + XX^{n}X^{-n}X^*\Big)=0.
  \end{gather*}
  So the kernel above necessarily also vanishes without compression.
\end{proof}

Concluding the current and previous section, we have achieved the first half in the classification of relative Cuntz--Pimsner algebras: that is we have found that every gauge-equivariant representation arises itself as relative Cuntz--Pimsner algebra for some kernel--covariance pair. In short, that is the kernel--covariance pairs exhaust the gauge-equivariant representations.

We may now proceed with the second half on the classification of relative Cuntz--Pimsner algebras:
The kernel--covariance pairs classify the relative Cuntz--Pimsner algebra and equivalently every possible kernel--covariance pair arises itself in nature.
As such we need to construct sufficiently many representations (as well as any nontrivial one whatsoever). This is where the Fock representation and the quotients thereof will come into play:

\section{Fock representation}
\label{sec:FOCK-SPACE}

We begin with the following problem:
Note that up until now we have not come across any representation whatsoever (put aside sufficiently many) except the trivial representation
\[
  B=0:\TAB\TAB
  \begin{tikzcd}
    X\rar&0,
  \end{tikzcd}\TAB
  \begin{tikzcd}
  A\rar&0
\end{tikzcd}
\]
and as such it \textbf{could be in principle} that the relative Cuntz--Pimsner algebras (for some correspondences) all coincide as the only trivial representation
\begin{TIKZCD}
  (X,A)\rar&\TOEPLITZ(X,A)=\ldots =\PIMSNER(K,I)=\PIMSNER(K',I')=\ldots=0~??
\end{TIKZCD}
As such the question arises whether a correspondence admits always a nontrivial representation (and further also sufficiently many faithful ones).\linebreak
For this we consider a failing attempt which leads us in turn to the Fock representation:
As a first guess we may consider the linking algebra associated to our correspondence (seen as Hilbert module only)
\[
  B=\MATRIX{A\\X}\MATRIX{A^*&X^*}=\MATRIX{A&X^*\\X&XX^*}\subset\ADJOINTS\left[\MATRIX{A\\X}\right]
\]
together with the representation given by the canonical embedding
\[
  \MATRIX{A\\&0}\subset \MATRIX{A&X^*\\X&XX^*}=B,\TAB\TAB \MATRIX{&0\\X}\subset \MATRIX{A&X^*\\X&XX^*}=B.
\]
Now while this representation respects the Hilbert module structure (basically by construction) it does not respect the left module structure:
\begin{gather*}
  \mediumMATRIX{&x^*\\0}\mediumMATRIX{&0\\y}=\mediumMATRIX{x^*y\\&0},\\[2\jot]
  \mediumMATRIX{&0\\x}\mediumMATRIX{a\\&0}=\mediumMATRIX{&0\\xa},\\[2\jot]
  \mediumMATRIX{a\\&0}\mediumMATRIX{&0\\x}=\mediumMATRIX{&0\\0}\neq\mediumMATRIX{&0\\ax}.
\end{gather*}
The solution is to simply extend these down the diagonal to infinity
which brings us straight to the Fock space representation:
More precisely, the Fock space is the infinite direct sum of increasing tensor powers (as in section \ref{sec:CORRESPONDENCES})
\[
  \FOCK(X):= A\oplus X\oplus XX\oplus XXX\oplus\ldots = \MATRIX{A\\X\\XX\\\vdots}
\]
and the representation as diagonal action and as right shift respectively:
\begin{align*}
  A\to\ADJOINTS\left(\FOCK X\right):&\TAB\TAB a\Big(A\oplus X\oplus XX\oplus\ldots\Big):=aA\oplus aX\oplus aXX\oplus\ldots\\[2\jot]
  X\to\ADJOINTS\left(\FOCK X\right):&\TAB\TAB x\Big(A\oplus X\oplus XX\oplus\ldots\Big):=0\oplus xA\oplus xX\oplus xXX\oplus\ldots
\end{align*}
We visualize them in matrix notation (using the formal left and right shift):
\[
  X\tensor R=X\mediumMATRIX{0\\1&0\\&1&0\\&&&\ldots},\TAB
  A\tensor 1=A\mediumMATRIX{1\\&1\\&&1\\&&&\ldots},\TAB
  X^*\tensor L=X^*\mediumMATRIX{0&1\\&0&1\\&&0\\&&&\ldots}.
\]
Further the Fock representation comes with the (inner) circle action
\[
  \TORUS\act\ADJOINTS(\FOCK X):\TAB\TAB (z\act \BLANK):=\mediumMATRIX{1\\&z\\&&z^2\\&&&\ldots}~\BLANK~\mediumMATRIX{1\\&z\\&&z^2\\&&&\ldots}^*
\]
which renders the representation gauge-equivariant:
\begin{gather*}
  \mediumMATRIX{1\\&z\\&&z^2\\&&&\ldots}A\mediumMATRIX{1\\&1\\&&1\\&&&\ldots}\mediumMATRIX{1\\&z\\&&z^2\\&&&\ldots}^*
  =1A\mediumMATRIX{1\\&1\\&&1\\&&&\ldots},\\[2\jot]
  \mediumMATRIX{1\\&z\\&&z^2\\&&&\ldots}X\mediumMATRIX{0\\1&0\\&1&0\\&&&\ldots}\mediumMATRIX{1\\&z\\&&z^2\\&&&\ldots}^*
  =zX\mediumMATRIX{0\\1&0\\&1&0\\&&&\ldots}.
\end{gather*}
Furthermore the representation of compact operators reads
\[
  XX^*~\mapsto~X\mediumMATRIX{0\\1&0\\&1&0\\&&&\ldots}\mediumMATRIX{0&1\\&0&1\\&&0&\\&&&\ldots}X^*= XX^*\mediumMATRIX{0\\&1\\&&1\\&&&\ldots}
\]
With these preliminary computations ready we may now reveal the Fock representation as the universal representation (i.e.~the Toeplitz algebra):

\begin{BREAKPROP}[Fock representation = Toeplitz algebra]
  The Fock representation has trivial kernel and covariance
  \[
    \ker\left(\begin{tikzcd}[column sep=0.5cm]
      A\rar& \ADJOINTS(\FOCK X)
    \end{tikzcd}\right)=0,
    \TAB\TAB
    \cov\left(\begin{tikzcd}[column sep=0.5cm]
      X\rar& \ADJOINTS(\FOCK X)
    \end{tikzcd}\right)=0
  \]
  and as such defines the universal representation (by theorem \ref{UNIQUENESS-THEOREM}).
\end{BREAKPROP}
\begin{proof}
  Clearly the representation is faithful (and so has trivial kernel) as
  \[
    \mediumMATRIX{a\\&a\\&&\ldots}\mediumMATRIX{A\\\ldots\\\ldots}=\mediumMATRIX{aA\\\ldots\\\ldots}=\mediumMATRIX{0\\0\\\ldots}\TAB\implies\TAB a=0.
  \]
  As a consequence, we may further compute the covariance as the intersection within the representation (as the special instance from proposition \ref{COVARIANCE:SUBCORRESPONDENCES})
  \[
    \cov\left(\begin{tikzcd}[column sep=0.5cm]
      X\rar& \ADJOINTS(\FOCK X)
    \end{tikzcd}\right) = A\mediumMATRIX{1\\&1\\&&\ldots}\cap XX^*\mediumMATRIX{0\\&1\\&&\ldots} = 0.
  \]
  So the covariance is trivial as well and the proposition follows.
\end{proof}

Note that more importantly we established the existence of any nontrivial representation whatsoever (equivalently the universal representation is different from the trivial representation) and further also the existence of faithful representations (equivalently the universal representation is faithful).\linebreak
Put in other words, we may now distinguish Toeplitz algebras along the lattice of quotient correspondences (the first dimension of kernel--covariance pairs):
\[
  \TOEPLITZ\left(\tfrac{X}{XK}\right)=\PIMSNER(K,0)~=~\PIMSNER(K',0)=\TOEPLITZ\left(\tfrac{X}{XK'}\right)
  \TAB\implies\TAB
  K=K'
\]
Indeed we have even found the stronger kernel relation from \eqref{RELATIVE-PIMSNER:KERNEL-COVARIANCE}:
\begin{gather*}
  \ker\left( \begin{tikzcd}[column sep=0.5cm]
    A\rar&\TOEPLITZ(X)
  \end{tikzcd} \right)=0
  \TAB\implies\TAB
  \ker\left( \begin{tikzcd}[column sep=0.5cm]
    A\rar&\frac{A}{K}\rar&\TOEPLITZ\left(\frac{X}{XK}\right)
  \end{tikzcd} \right) = K
\end{gather*}
We are still left with the question (which we get to next)
\[
  \PIMSNER(K,I)=\PIMSNER(K',I')\TAB\implies\TAB (~K=K'~|~I=I'~)\TAB?
\]
For this we may now consider the Fock representation as a concrete realization for the Toeplitz algebra and as such further construct every relative Cuntz--Pimsner algebra (along kernel--covariance pairs) as a concrete quotient thereof:
For example given first a quotient correspondence (for some invariant ideal)
\[
  K\ILEQ A:\TAB X^*KX\subset K:\TAB\TAB
  \begin{tikzcd}
    (X,A)\rar&\left(\frac{X}{XK},\frac{A}{K}\right)
  \end{tikzcd}
\]
we may construct the corresponding Toeplitz algebra as quotient by the kernel
(more precisely its ideal generated within the original Toeplitz representation)
\begin{TIKZCD}
  0\rar& (XK,K) \rar\dar& (X,A) \rar\dar& \left(\frac{X}{XK},\frac{A}{K}\right) \rar\dar&0 \\
  0\rar& \TOEPLITZ X (K\subset A)\TOEPLITZ X \rar& \TOEPLITZ X \rar& \TOEPLITZ\left(\frac{X}{XK}\right) \rar&0
\end{TIKZCD}
where we omit as usual the closed linear span for more pointy statements.\\
Before continuing we replace for convenience the original correspondence by the quotient correspondence,
\[
  \left(\tfrac{X}{XK},\tfrac{A}{K}\right)\rightsquigarrow(X,A)
  \TAB\implies\TAB
  \TOEPLITZ\left(\tfrac{X}{XK}\right)\rightsquigarrow\TOEPLITZ (X).
\]
Consider next an ideal (as possible covariance) bounded from above by the maximal covariance (as in proposition \ref{FAITHFUL-REPRESENTATIONS:MAXIMAL-COVARIANCE}),
\[
  I\ILEQ A:\TAB I\subset\max(X,A).
\]
The bound from above is necessary as larger covariances force an additional kernel and as such would factor over some further quotient correspondence.
Similarly as for relative Toeplitz algebras, the relative Cuntz--Pimsner algebra arises now as coequalizer for the chosen covariance
\begin{gather*}
  \PIMSNER(K,I)=
  \coequalizer\left(\begin{tikzcd}[column sep=1cm,row sep=0.5cm]
    I\subset A\cap XX^*\rar\dar& \TOEPLITZ X \dar[equal]\\ XX^*\rar&\TOEPLITZ X
  \end{tikzcd}\right)
\end{gather*}
and as such also as quotient by their difference as in \eqref{DIFFERENCE-MORPHISM}:
\begin{gather*}
  \begin{tikzcd}
    0\rar& \TOEPLITZ X\Big(~(\phi-\tau)I~\Big)\TOEPLITZ X \rar& \TOEPLITZ X \rar& \PIMSNER(X;I) \rar& 0:
  \end{tikzcd}\\[4\jot]
  \hspace{-1cm}(\phi-\tau)=
  \left(\begin{tikzcd}[y=0.5cm]
    \node (A) at (-1,+1) {A\cap XX^*};
    \node (B) at (+1,+1) {\TOEPLITZ X};
    \node (D) at (+1,-1) {\TOEPLITZ X};
    \draw[->] (A.east) to (B.west);
    \draw[double equal sign distance] (B.south) to (D.north);
  \end{tikzcd}\right)-
  \left(\begin{tikzcd}[y=0.5cm]
    \node (A) at (-1,+1) {A\cap XX^*};
    \node (C) at (-1,-1) {XX^*};
    \node (D) at (+1,-1) {\TOEPLITZ X};
    \draw[->] (C.east) to (D.west);
    \draw[->] (A.south) to (C.north);
  \end{tikzcd}\right).
\end{gather*}
We now invoke the Fock representation as Toeplitz algebra:
Recall that we have already found here the concrete embedding for the coefficient algebra as well as the concrete embedding of compact operators (see above)
\[
  \TOEPLITZ X\subset\ADJOINTS\left[ \mediumMATRIX{A\\X\\\ldots} \right]:\TAB\TAB
  A~\mapsto~A\mediumMATRIX{1\\&1\\&&\ldots},\TAB
  XX^*~\mapsto~XX^*\mediumMATRIX{0\\&1\\&&\ldots}
\]
and as such their difference reads
\[
    A\cap XX^*~\mapsto~ A\cap XX^*\mediumMATRIX{1\\&0\\&&\ldots}.
\]
We therefore found the relative Cuntz--Pimsner algebra as quotient
\begin{equation}
  \label{CUNTZ--PIMSNER--SES}
  \begin{tikzcd}
    0\rar& \TOEPLITZ X\mediumMATRIX{I\\&0\\&&\ldots}\TOEPLITZ X \rar& \TOEPLITZ X \rar& \PIMSNER(X;I) \rar& 0.
  \end{tikzcd}
\end{equation}
This was originally established by Muhly and Solel in \cite[theorem~2.19]{MUHLY-SOLEL-TENSOR}.\\
We now wish to verify that the induced quotient representation remains faithful:
That is we note that the quotient \textbf{could in principle} introduce new kernel,
\[
  I\subset A\cap XX^*:\TAB\TAB\ker\left( \begin{tikzcd}
    A\rar&\TOEPLITZ X\rar&\PIMSNER(X;I)
  \end{tikzcd} \right)=0~?
\]
Indeed we note that the quotient \textbf{does introduce new kernel} as soon as the covariance \textbf{exceeds the maximal covariance} (as we have noted also above).
As such we have to make sure this does not happen as long as the covariance lies below the maximal covariance from proposition \ref{FAITHFUL-REPRESENTATIONS:MAXIMAL-COVARIANCE}.
For this we note that the trivial kernel above may be equivalently verified now as the trivial intersection
\begin{equation}
  \label{KERNEL-RELATION:TRIVIAL-INTERSECTION}
  I\subset \max(X,A):\TAB\TAB A\mediumMATRIX{1\\&1\\&&\ldots}\cap\TOEPLITZ X\mediumMATRIX{I\\&0\\&&\ldots}\TOEPLITZ X=0~?
\end{equation}
On the other hand, we wish to also verify the covariance relation from \eqref{RELATIVE-PIMSNER:KERNEL-COVARIANCE}:
\[
  \cov\left( \begin{tikzcd}
    A\rar& \TOEPLITZ X\rar&\PIMSNER(X;I)
  \end{tikzcd} \right) = I~?
\]
The problem here is that the covariance could in principle increase as well.\linebreak
Indeed the construction (and even our very definition) of relative Cuntz--Pimsner algebras guarantees just a least covariance for the provided covariance ideal.
This becomes more evident as follows:
Consider for this the difference morphism which factors by the universal property via the Toeplitz algebra:
\begin{TIKZCD}
  &&A\cap XX^* \rar[equal]\dar\drar[phantom]{\phi-\tau}& A\cap XX^*\dar\\
  0\rar& \TOEPLITZ X\mediumMATRIX{I\\&0\\&&\ldots}\TOEPLITZ X \rar& \TOEPLITZ X \rar& \PIMSNER(X;I) \rar& 0.
\end{TIKZCD}
As such the covariance may be read off from the common intersection as
\begin{equation}
  \label{COVARIANCE-RELATION:TRIVIAL-INTERSECTION}
  \mediumMATRIX{A\cap XX^*\\&0\\&&\ldots}\cap\TOEPLITZ X\mediumMATRIX{I\\&0\\&&\ldots}\TOEPLITZ X = \mediumMATRIX{I\\&0\\&&\ldots}~?
\end{equation}
Note this meanwhile also highlights how the covariance \textbf{could increase: how?}\\
Let us begin to verify that the representation remains faithful along the quotient.
For this we begin with the following well-known relation (which goes all the way back to an observation by Joachim Cuntz made in \cite{CUNTZ-FAMOUS-CUNTZ-ALGEBRAS}):

\begin{PROP}\label{CUNTZ--PIMSNER--IDEAL}
  The ideal generated by the covariance as in \eqref{CUNTZ--PIMSNER--SES} coincides with the ideal of compact operators of the form
  \[
    \TOEPLITZ X\mediumMATRIX{I\\&0\\&&\ldots}\TOEPLITZ X
    = \mediumMATRIX{A\\X\\\ldots}I\mediumMATRIX{A&X^*&\ldots}
    = \COMPACTS\left[\mediumMATRIX{A\\X\\\ldots}I\right]
  \]
  with implicit closed linear spans for more pointy statements.
\end{PROP}
\begin{proof}[Sketch of proof:]
  The result follows most easily using the formal right and left shift operators (as further above) from which the covariance ideal reads
  \[
    \mediumMATRIX{I\\&0\\&&\ldots} = I\left[\mediumMATRIX{1\\&1\\&&\ldots}-\mediumMATRIX{0\\&1\\&&\ldots}\right]= I\tensor(1-RL).
  \]
  Recall however that these satisfy the well-known relation
  \[
    LR=1\TAB\implies\TAB L(1-RL)=0=(1-RL)R.
  \]
  As such we obtain as the only contributions for the ideal
  \[
  \TOEPLITZ X\mediumMATRIX{I\\&0\\&&\ldots}\TOEPLITZ X
  = \sum_{mn}\left(X\mediumMATRIX{0\\1&0\\&&\ldots}\right)^m\mediumMATRIX{I\\&0\\&&\ldots}\left(X^*\mediumMATRIX{0&1\\&0\\&&\ldots}\right)^n.
  \]
  On the other hand they generate the system of matrix units such as
  \[
    \mediumMATRIX{0\\1&0\\&&\ldots}^m\mediumMATRIX{1\\&0\\&&\ldots}\mediumMATRIX{0&1\\&0\\&&\ldots}^n
    = \mediumMATRIX{&0\\0& 1 &0\\&0}
    = \mediumMATRIX{0\\1\\0}\mediumMATRIX{0&1&0}.
  \]
  As such we obtain for the above ideal
  \[
    \TOEPLITZ X\mediumMATRIX{I\\&0\\&&\ldots}\TOEPLITZ X = \ldots = \sum_{mn}\mediumMATRIX{0\\X^n\\0}I\mediumMATRIX{0&X^{-n}&0} = \FOCK(X)\FOCK(X)^*
  \]
  which is the desired relation for the covariance ideal.
\end{proof}

In order to handle the kernel for the induced representation on the quotient
we need to take a closer look into the ideal of compact operators from \ref{CUNTZ--PIMSNER--IDEAL}:\linebreak
We begin for this with the well-known approximation by \enquote{finite rank matrices}.
More precisely, one has for compact operators on Fock space and $S\subset \NATURALS$:
\[
  \begin{tikzcd}[column sep=normal]
    \mediumMATRIX{ 0&& \\ &\COMPACTS\left[\ldots\right]& \\ &&0 }\ni
    \mediumMATRIX{0&\\&k(S)\\&&0}
    \rar{S\to\NATURALS}&
    \mediumMATRIX{k_{00}&k_{01}\\k_{10}&k_{11}\\&&\ldots}
    \in\COMPACTS\left[\mediumMATRIX{A\\X\\\ldots}\right].
  \end{tikzcd}
\]
Indeed one easily verifies this (using Dirac calculus from \ref{DIRAC-CALCULUS}):
\begin{TIKZCD}[column sep=small]
  \mediumMATRIX{\ldots\\&0\\&&1}
  \mediumMATRIX{A\\X\\\ldots}
  \rowMATRIX{A^*&X^*&\ldots}
  \rar&0&\lar
  \mediumMATRIX{A\\X\\\ldots}
  \rowMATRIX{A^*&X^*&\ldots}
  \mediumMATRIX{\ldots\\&0\\&&1}
\end{TIKZCD}
In particular we obtain for the diagonal compact operators as in \ref{CUNTZ--PIMSNER--IDEAL}:
\[
  \COMPACTS\left[\mediumMATRIX{A\\X\\\ldots}I\right]\cap\mediumMATRIX{\ADJOINTS(A)\\&\ADJOINTS(X)\\&&\ldots} =\mediumMATRIX{I\\&XIX^*\\&&\to0}.
\]
Meanwhile the author would like to take a moment to thank
his previous tutor Dominic Enders for highlighting this perspective during personal discussions.

The idea is now to use the previous relation in contrast to the following \textbf{observation by Katsura} which we reformulate in our language:
Consider for this the diagonal operators on Fock space
\[
  \mediumMATRIX{\ADJOINTS(A)\\&\ADJOINTS(X)\\&&\ldots}\subset\ADJOINTS\left[\FOCK X=\mediumMATRIX{A\\X\\\ldots}\right]
\]
and the representation between such diagonal operators:
\begin{TIKZCD}[column sep=0.5cm]
  \mediumMATRIX{0\\&\ADJOINTS(X^n)\\&&0\\&&&0}
  \rar[mapsto]&
  \mediumMATRIX{0\\&0\\&&\ADJOINTS(X^n)\tensor 1\\&&&0}
  \subset
  \mediumMATRIX{0\\&0\\&&\ADJOINTS(X^n\tensor X)\\&&&0}
\end{TIKZCD}
This generally fails to define a faithful representation:
one may for instance consider the graph correspondence for any finite acyclic graph such as
\[
  X=\ell^2\Big( E=\begin{tikzcd}[column sep=0.5cm]
    \bullet \rar&\bullet
  \end{tikzcd} \Big)
  \TAB\implies\TAB
  XX=\ell^2\Big( EE=\emptyset \Big)=0.
\]
Katsura's crucial obervation tells us now that this becomes faithful
when restricted to the subspace of compact operators by our covariance ideal:

\begin{PROP}[{\cite[Lemma~4.7]{KATSURA-SEMINAL}}]
  \label{KERNEL-RELATION:KATSURAS-EMBEDDING}
  The representation above defines an embedding when restricted to the maximal covariance as in proposition \ref{CUNTZ--PIMSNER--IDEAL},
  \begin{gather*}
    \mediumMATRIX{0\\&\ket{X^n}\max(X,A)\bra{X^n}\\&&0\\&&&0}
    \subset
    \mediumMATRIX{0\\&0\\&&\ket{X^n}\max(X,A)\bra{X^n}\tensor 1\\&&&0}.
  \end{gather*}
  As such the representation defines also an embedding for any covariance ideal below the maximal covariance.
\end{PROP}
\begin{proof}[Proof from \cite{KATSURA-SEMINAL}:]
  We note for the kernel (in Dirac braket notation)
  \begin{gather*}
    \bra{X^n}\ker\bigg[\ket{X^n}\max(X,A)\bra{X^n}\act\ket{X^n}\tensor\ket{X}\bigg]\ket{X^n}\\
    \subset\ker\bigg[\braket{X^n|X^n}\max(X,A)\braket{X^n|X^n}\act\ket{X}\bigg]\\
    \subset \ker(A\act X)\cap \ker(A\act X)^\perp=0
  \end{gather*}
  where we have used the obvious inclusion
  \[
    \braket{X^n|X^n}\max(X,A)\braket{X^n|X^n}\subset\max(X,A)\subset\ker(A\act X)^\perp.
  \]
  (Note the intersection reflects also the first level as in proposition \ref{MAXIMAL-FAITHFUL-ACTION}.)\\
  For an ideal such as the kernel above it holds however
  \[
    \bra{X^n}\bigg(\ker[\ldots]=\ker[\ldots]^*\ker[\ldots]\bigg)\ket{X^n}=0
    \TAB\implies\TAB
    \ker[\ldots]\ket{X^n}=0.
  \]
  As such we found the inclusion
  \begin{gather*}
    \ker\bigg[\ket{X^n}\max(X,A)\bra{X^n}\act\ket{X^n}\tensor\ket{X}\bigg]
    \subset\ker\bigg[\ketbra{X^n}{X^n}\act\ket{X^n}\bigg]=0.
  \end{gather*}
  In particular, the same holds true for any covariance below the maximal.
\end{proof}

With Katsura's observation at hand we may now verify the desired kernel and covariance relation as in \eqref{RELATIVE-PIMSNER:KERNEL-COVARIANCE} which will resolve the second half of our classification of relative Cuntz--Pimsner algebras. We note here that the kernel relation is already due to Katsura as established in \cite{KATSURA-SEMINAL}:

\begin{BREAKTHM}[Relative Cuntz--Pimsner algebras: Kernel and Covariance]
  \label{THEOREM:RELATIVE-PIMSNER:KERNEL-COVARIANCE}
  For the relative Cuntz--Pimsner algebra as above it holds
  \begin{gather*}
    \ker\left( \begin{tikzcd}[column sep=small]
      A\to \TOEPLITZ \to \PIMSNER(X;I)
    \end{tikzcd} \right)=0\\
    \cov\left( \begin{tikzcd}[column sep=small]
      X\to \TOEPLITZ \to \PIMSNER(X;I)
    \end{tikzcd} \right)=I
  \end{gather*}
  As a consequence, the kernel--covariance pairs are also classifying.
\end{BREAKTHM}

\begin{proof}
  We begin with \textbf{the kernel relation} due to Katsura from \cite{KATSURA-SEMINAL}:\\
  As in the beginning discussion we need to verify the relation \eqref{KERNEL-RELATION:TRIVIAL-INTERSECTION}:
  \[
    I\subset \max(X,A):\TAB\TAB A\mediumMATRIX{1\\&1\\&&\ldots}\cap\TOEPLITZ X\mediumMATRIX{I\\&0\\&&\ldots}\TOEPLITZ X=0~?
  \]
  For this we first revealed in proposition \ref{CUNTZ--PIMSNER--IDEAL} that the ideal generated by our covariance (within the Toeplitz algebra) agrees with compact operators.\linebreak
  As such the intersection with the coefficient algebra reads
  \[
    A\mediumMATRIX{1\\&1\\&&\ldots}\cap
    \TOEPLITZ X\mediumMATRIX{I\\&0\\&&\ldots}\TOEPLITZ X =
    A\mediumMATRIX{1\\&1\\&&\ldots}\cap
    \mediumMATRIX{I\\&\ket{X}I\bra{X}\\&&\to0}.
  \]
  In contrast, Katsura's observation from proposition \ref{KERNEL-RELATION:KATSURAS-EMBEDDING} tells us that these embed along the diagonal and as such the norm remains also constant throughout,
  \[
    \mediumMATRIX{a\\&a\\&&\ldots}\in \mediumMATRIX{I\\&\ket{X}I\bra{X}\\&&\ldots}:\TAB\TAB
    \|a\| = \|a\act X\| = \ldots = \|a\act X^n\|.
  \]
  Both the vanishing of compact operators along the diagonal and the constant norm
  are only possible for the trivial intersection and as such the trivial kernel.

  We continue with \textbf{the covariance relation} from \eqref{RELATIVE-PIMSNER:KERNEL-COVARIANCE}:
  For this we may now simply verify the common intersection as in \eqref{COVARIANCE-RELATION:TRIVIAL-INTERSECTION} also using proposition \ref{CUNTZ--PIMSNER--IDEAL}:
  \begin{gather*}
    \mediumMATRIX{A\cap XX^*\\&0\\&&\ldots}\cap
    \TOEPLITZ X\mediumMATRIX{I\\&0\\&&\ldots}\TOEPLITZ X
    = \\[2\jot] =
    \mediumMATRIX{A\cap XX^*\\&0\\&&\ldots}\cap
    \mediumMATRIX{I&IX^*\\XI&XIX^*\\&&\ldots} =
    \mediumMATRIX{I\\&0\\&&\ldots}.
  \end{gather*}
  As such the covariance \textbf{does not increase} and the theorem is proven.
\end{proof}

We have thus established also the \textbf{second half in our classification:}\\
More precisely, we have first found that the class of relative Cuntz--Pimsner algebras exhausts the gauge-equivariant representations (which was the content of the gauge-invariant uniqueness theorem). On the other hand we now found that the parametrisation via kernel--covariance pairs \textbf{is also classifying:}
\[
  \PIMSNER(K,I)=\PIMSNER(K',I')\TAB\implies\TAB (K,I)=(K',I')\TAB\checkmark
\]
\textbf{Altogether we have thus found:} the lattice of kernel--covariance pairs parametrises the entire lattice of gauge-equivariant representations (as points).

\section{Lattice structure}

While our discussion (so far) captured the lattice of gauge-equivariant representations \textbf{as individual points} along the lattice,
this still leaves open how the lattice structure of kernel--covariance pairs
\textbf{reflects the lattice structure} of gauge-equivariant representations (among each other) to which we now get:\linebreak
Recall for this that our kernel--covariance pairs encode the covariance for the quotient correspondence (which rendered the representation faithful)
\[
  \begin{tikzcd}[column sep=small]
    (X,A)\rar&B
  \end{tikzcd}:\TAB\TAB
  K=\ker\Bigl(A\to B\Bigr)
  \TAB\implies\TAB
  I=\cov\Bigl(\tfrac{X}{XK}\to B\Bigr)
\]
and its intrinsic characterisation on the quotient (as bounded ideal)
\[
  \Bigl(~K\ILEQ A~\Big|~X^*KX\subset K~\Bigr)
  \TAB\implies\TAB
  \Bigl(~I\ILEQ A/K~\Big|~I\subset\max\left( \tfrac{X}{XK} \right)~\Bigr).
\]
We therefore begin with a translation of our kernel--covariance pairs which lives on the original correspondence.
This allows us to give an intrinsic order on kernel--covariance pairs reflecting the lattice structure of representations.\linebreak
Along this translation we further reveal Katsura' mysterious T-pairs as nothing but our original kernel--covariance pairs (with maximal covariance in disguise).

For our translation we first recall that the covariance for an embedding (faithful representation) may be simply read off as common intersection within the ambient algebra (as in proposition \ref{COVARIANCE:SUBCORRESPONDENCES})
\[
  (X,A)\subset B:\TAB\TAB
  \cov(X\to B) =
  \im(A\to B)\cap
  \im(XX^*\to B).
\]
More precisely, one may take the portion of the coefficient algebra
\[
  \cov(X\to B)=\Bigl\{~a\in A~\Big|~a\in\im(XX^*\to B)~\Bigr\}=A\cap\im(XX^*).
\]
That is however the same amount of information as the actual intersection,\linebreak
as long as one keeps track of the embedding for the coefficient algebra:
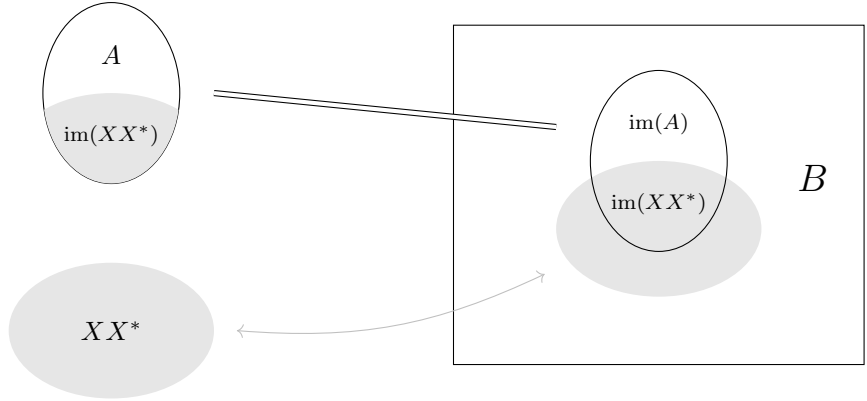
\begin{center}
  \begin{tikzpicture}[scale=0.9]
    \draw (-3,-2) rectangle (3,3);
    \node at (2.25,0.75) {\Large$B$};
    \begin{scope}[shift={(-8,1)}]
      \fill[shift={(0,-2.5)},gray!20]
        (0,0) node[black]{$XX^*$} ellipse (3/2 and 1);
      \draw[shift={(0,1)},clip]
        (0,0) node[shift={(0,0.5)}]{$A$} ellipse (1 and 4/3);
      \fill[shift={(0,0)},gray!20]
        (0,0) node[shift={(0,1/3)},black] {\footnotesize$\im(XX^*)$}  ellipse (3/2 and 1);
    \end{scope}
    \draw[double equal sign distance,shift={(0,0.25)}] (-6.5,1.75) to (-1.5,1.25);
    \draw[<->,shift={(1/3,0)},bend right=15,gray!50] (-6.5,-1.5) to (-2,-2/3);
    \fill[shift={(0,-0)},gray!20] (0,0) node[shift={(0,1/3)},black] {\footnotesize$\im(XX^*)$} ellipse (3/2 and 1);
    \draw[shift={(0,+1)},clip] (0,0) node[shift={(0,0.5)}]{\footnotesize$\im(A)$} ellipse (1 and 4/3);
  \end{tikzpicture}
\end{center}
With this picture in mind, we continue on some general representation
\[
  \begin{tikzcd}
    (X,A)\rar& \left( \frac{X}{XK},\frac{A}{K} \right)\rar[dashed]& B
  \end{tikzcd}:
  \TAB\TAB
  K=\ker(A\to B).
\]
We first note that for a quotient (i.e.~surjective mapping) there is absolutely no loss of generality when pulling back any ideal along the quotient since for
\[
  I\ILEQ A/K\TAB\rightsquigarrow\TAB
  \left( A\to\tfrac{A}{K} \right)\inv I\ILEQ A:\TAB\TAB
  I=\left( A\to\tfrac{A}{K} \right)\left( A\to\tfrac{A}{K} \right)\inv I
\]
and so we may use the equivalent intrinsic definition of kernel--covariance pairs
as those with covariance below the maximal covariance within the quotient:
\begin{equation}
  \label{COVARIANCE-IDEAL:QUOTIENT-PULLBACK}
  (I+K)\ILEQ  A:\TAB\TAB
  I/K\subset\max\left( \tfrac{X}{XK} \right)\TAB\text{!!}
\end{equation}
We wrote the covariance ideal (here and below) as sum with the kernel ideal simply to guarantee that the ideal arises indeed as pullback from the quotient.

On the other hand, note that the amount of covariance (as described above) does not change either in the sense of how much common intersection the coefficient algebra has with compact operators since (see also proposition \ref{SHORT-EXACT:COMPACTS}):
\begin{gather*}
  \im\Bigl( \begin{tikzcd}[column sep=normal]
    A\rar& \frac{A}{K} \rar& B
  \end{tikzcd} \Bigr) =
  \im\Bigl( \begin{tikzcd}[column sep=normal]
    \frac{A}{K} \rar& B
  \end{tikzcd} \Bigr), \\[2\jot]
  \im\Bigl( \begin{tikzcd}[column sep=small]
  XX^*\rar& \left(\frac{X}{XK}\right)\left(\frac{X}{XK}\right)^* \rar& B
  \end{tikzcd} \Bigr) =
  \im\Bigl( \begin{tikzcd}[column sep=small]
    \left(\frac{X}{XK}\right)\left(\frac{X}{XK}\right)^* \rar& B
  \end{tikzcd} \Bigr).
\end{gather*}
So there is really no loss of generality from this perspective either.\\
As such we obtain another equivalent \textbf{extrinsic definition} of kernel--covariance pairs as those with
covariance ideal describing the amount of common intersection between the coefficient algebra and compact operators:
\[
  \begin{tikzcd}[column sep=small]
    (X,A)\rar&B
  \end{tikzcd}:\TAB\TAB
  (I+K) = \im(A\to B)\cap\im(XX^*\to B).
\]
For comparison between kernel--covariance pairs we however take from now on the portion within the coefficient algebra as above, that is
\[
  (I+K)=\Big\{~a\in A~\Big|~a\in\im(XX^*\to B)~\Big\} = A\cap\im(XX^*\to B)
\]
and we note this agrees with our intrinsic definition (somewhat obvious now).\linebreak
Meanwhile we also keep in mind the viewpoint on the covariance as amount of common intersection
as it provides an interesting perspective on representations.

With both these definitions at hand (the intrinsic and the extrinsic) we may now get to the lattice of gauge-equivariant representations.
Given a pair of representations we define the usual order of representations as
\[
  \Bigl(~(X,A)\to B~\Bigr)
  \leq
  \Bigl(~(X,A)\to B'~\Bigr):\TAB\TAB
  \begin{tikzcd}
    (X,A)\rar& B\rar[dashed]& B'.
  \end{tikzcd}
\]
More precisely, that is the representation factors over the other and note that the sole existence of such a factorisation entails a unique such as the representations are all generated as an operator algebra by (the image of) the correspondence:
\[
  B=\CSTAR(A\cup X)
  \TAB\implies\TAB
  \begin{tikzcd}[column sep=small]
    B \rar[dashed]& B'
  \end{tikzcd}~\text{uniquely}\TAB\checkmark
\]
Given a factorisation we now easily infer for their kernel and covariance
\begin{gather*}
  \ker(A\to B)\subset \ker(A\to B\to B'),\\[2\jot]
  A\cap\im(XX^*\to B)\subset A\cap \im(XX^*\to B\to B').
\end{gather*}
Indeed the latter may be easily seen as (somewhat trivially)
\begin{gather*}
  \im(a\to B)\in\im(XX^*\to B)
  ~\implies~
  \im(a\to B\to B')\in\im(XX^*\to B\to B').
\end{gather*}
Schematically the amount of intersection could look something like:
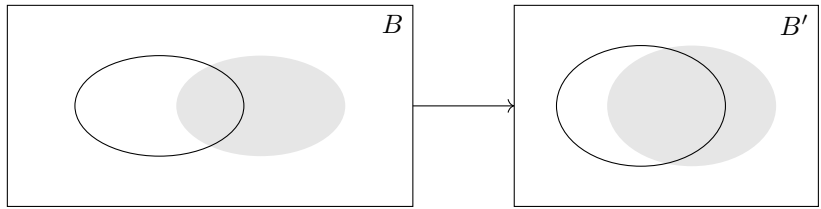
\begin{center}
  \begin{tikzpicture}[scale=2/3]
    \begin{scope}[shift={(-5,0)}]
      \draw (-4,-2) rectangle (4,2) node [anchor=north east] {$B$};
      \fill[shift={(+1,0)},gray!20] (0,0) ellipse (5/3 and 1);
      \draw[shift={(-1,0)}] (0,0) ellipse (5/3 and 1);
    \end{scope}
    \begin{scope}[shift={(+5,0)}]
      \draw (-4,-2) rectangle (2,2) node [anchor=north east] {$B'$};
      \fill[shift={(-0.5,0)},gray!20] (0,0) ellipse (5/3 and 1.2);
      \draw[shift={(-1.5,0)}] (0,0) ellipse (5/3 and 1.2);
    \end{scope}
    \draw[->] ($(4,0)-(5,0)$) to ($(-4,0)+(5,0)$);
  \end{tikzpicture}
\end{center}
So we have found the following converse direction (using theorem \ref{THEOREM:RELATIVE-PIMSNER:KERNEL-COVARIANCE}):
\[
  \Bigl(~K\subset L~\Big|~I+K\subset J+L~\Bigr)
  \TAB\impliedby\TAB
  \PIMSNER(K,I)\leq\PIMSNER(L,J)
  \TAB\checkmark
\]
What about the forward direction? That is assume we have an inclusion of kernel--covariance pairs as above.
As we have an inclusion of kernel ideals we obtain in particular for their quotient correspondence
\begin{TIKZCD}[column sep=2cm]
  (X,A)\rar&
  \left( \frac{X}{XK},\frac{A}{K} \right) \rar\dar&
  \left( \frac{X}{XL},\frac{A}{L} \right) \dar \\
  & \PIMSNER(K,I) & \PIMSNER(L,J)
\end{TIKZCD}
and so we may replace our correspondence as usual by the quotient
\begin{gather*}
  \left( \tfrac{X}{XK},\tfrac{A}{K} \right)\rightsquigarrow(X,A)
  \TAB\implies\TAB
  \PIMSNER(K,I)\rightsquigarrow\PIMSNER(0,I).
\end{gather*}
Recall that the relative Cuntz--Pimsner algebra satisfies (by definition)
\[
  \begin{tikzcd}[column sep=small]
    (X,A)\rar& \PIMSNER(X;I) \rar[dashed]&B
  \end{tikzcd}
  \TAB\iff\TAB
  I\subset \cov(~X\to B~)
\]
and as such we need to verify the least amount of covariance
\[
  I\subset\cov\left( \begin{tikzcd}[column sep=small]
    (X,A)\rar& \left( \tfrac{X}{XL},\tfrac{A}{L} \right) \rar&\PIMSNER(L,J)
  \end{tikzcd} \right)~?
\]
This basically follows now from our study of kernel morphisms and covariance ideals from section \ref{sec:COVARIANCES},
which we recall now for more clarity in our context:\linebreak
At first we found that kernel and cokernel morphisms have full covariance.
That is in our context the fully commutative diagram for the quotient morphism and on the other hand the covariance diagram for the quotient representation:
\[
  \begin{tikzcd}
    A\cap XX^*\rar\dar& A/L\dar \\
    XX^*\rar& \left(\frac{X}{XL}\right)\left(\frac{X}{XL}\right)^*
  \end{tikzcd}
  \TAB\text{and}\TAB
  \begin{tikzcd}
    J/L\rar\dar& \PIMSNER(L,J)\dar[equal]\\
    \left(\frac{X}{XL}\right)\left(\frac{X}{XL}\right)^* \rar&\PIMSNER(L,J).
  \end{tikzcd}
\]
Combing these with our assumption on covariance ideals we obtain
\begin{gather*}
  \begin{tikzcd}[column sep=2cm]
    (I+0)\rar\dar& J/L\rar\dar&\PIMSNER(L,J)\dar[equal]\\
    XX^*\rar& \left(\frac{X}{XL}\right)\left(\frac{X}{XL}\right)^* \rar&\PIMSNER(L,J)
  \end{tikzcd}\\[2\jot]
\end{gather*}
and as such the desired amount of covariance,
\[
  (I+0)\subset (J+L)\TAB\implies\TAB I\subset\cov\Bigl(~(X,A)\to \PIMSNER(L,J)~\Bigr).
\]
As such we also found the forward direction and so the order isomorphism,\linebreak
which is the \textbf{main conclusion of this article:}

\begin{BREAKTHM}[Kernel--covariance pairs: Order isomorphism]
  The kernel--covariance pairs as in \eqref{COVARIANCE-IDEAL:QUOTIENT-PULLBACK}
  define the order isomorphism
  \begin{equation}
    \label{LATTICE-ISOMORPHISM}
    \Bigl(~K\subset L~\Big|~I+K\subset J+L~\Bigr)
    \TAB\iff\TAB
    \PIMSNER(K,I)\leq\PIMSNER(L,J)
  \end{equation}
  and as such the lattice of \textbf{kernel--covariance pairs} with its natural order by inclusion
  describes the entire \textbf{lattice of gauge-equivariant} representations,
  equivalently the entire \textbf{lattice of gauge-invariant ideals.}
\end{BREAKTHM}

Let us give an example of our result for some graph algebra,\\
also to illustrate the ease of working with such kernel--covariance pairs:
\begin{BREAKEXM}[Graph correspondences: gauge-invariant ideals]
  Consider a graph correspondence as in example \ref{GRAPH-CORRESPONDENCES:TENSOR-POWERS},
  \[
    X=\ell^2\Big(E=\edges\Big),\TAB A=c_0(\mathrm{vertices})
  \]
  and recall its quotient graphs as in \ref{GRAPH-CORRESPONDENCES:QUOTIENT-GRAPHS}
  (given by hereditary ideals as in \ref{GRAPH-CORRESPONDENCES:HEREDITARY-IDEALS})\linebreak
  as well as their covariance ideals given by their sets of regular vertices as in \ref{GRAPH-CORRESPONDENCES:MAXIMAL-COVARIANCE}.\\
  As an example consider the following graph and its quotient graphs
  \vspace*{-\baselineskip/2}
  \[\begin{tikzcd}[column sep=1.5cm]
    \begin{gathered}
      K=0:\\
      \left(\begin{tikzpicture}[baseline=8pt]
        \node (a) at (0,0) {$a$} edge [loop] ();
        \node (b) at (1,0) {$b$} edge [loop] ();
        \draw[->] (a)--(b);
      \end{tikzpicture}\right)\\[\baselineskip]
    \end{gathered}
    \rar&
    \begin{gathered}
      K=(a):\\
      \left(\begin{tikzpicture}[baseline=8pt]
        \node (b) at (1,0) {$b$} edge [loop] ();
      \end{tikzpicture}\right)\\[\baselineskip]
    \end{gathered}
    \rar&
    \begin{gathered}
      K=(a\cup b):\\
      \left(\begin{tikzpicture}
        \node {$\emptyset$};
      \end{tikzpicture}\right).\\[1.2\baselineskip]
    \end{gathered}
  \end{tikzcd}
  \vspace*{-1.5\baselineskip}\]
  Then its lattice of gauge-equivariant representations reads (as Hasse diagram)
  
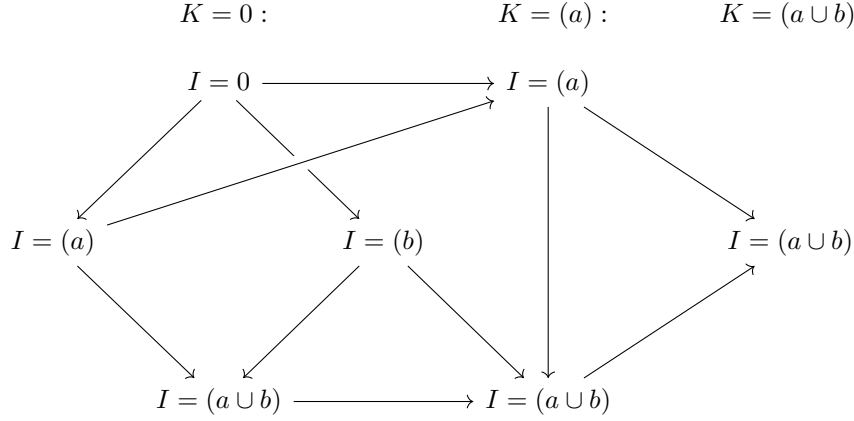
\[\hspace{-1cm}\begin{tikzcd}[column sep=small,row sep=1.5cm,arrows={crossing over}]
    & ~K=0: && ~K=(a): && ~K=(a\cup b): \\[-1.2cm]
    & I=0 \arrow[rr]\dlar\drar && I=(a) \arrow[dd]\arrow[drr] && \\
    I=(a) \drar\arrow[urrr] && I=(b) \dlar\drar &&& I=(a\cup b)\\
    & I=(a\cup b) \arrow[rr] && I=(a\cup b) \arrow[urr]
  \end{tikzcd}\]
  and so equivalently also the lattice of gauge-invariant ideals.\\
  Note how we now easily read off the order via kernel--covariance pairs.
\end{BREAKEXM}

Before we continue let us make a few remarks on the order isomorphism:\\
In particular the following discussion on connecting morphisms will basically cover the notion of suprema and infima as addressed in the following remark:

\begin{BREAKREM}[Lattice isomorphism: suprema and infima]
Note that as both lattices are order isomorphic they will be also lattice isomorphic as unions and intersections (finite or arbitrary) as well as top and bottom elements are determined as suprema and infima respectively
\[
  \mathrm{order~iso}\left(\sup_s a_s\right)=\sup_s\Bigl(\mathrm{order~iso}a_s\Bigr)
\]
which aside their existence depend only on the given order.\\
Put in other words the notion of a lattice is really \textbf{just a pure property} and defines no additional structure
so there is \textbf{really no difference between} the lattice of gauge-equivariant representations and the lattice of kernel--covariance pairs.
We will however later discover that arbitrary suprema and infima of kernel--covariance pairs do not necessarily always arise as intersections and sums of their kernel and covariance ideals, that is we only have
\[
  \inf_s(K_s|I_s)\leq\left(~\bigcap_sK_s~\bigg|~\bigcap_sI_s~\right)
  \TAB\text{and}\TAB
  \sup_s(K_s|I_s)\geq\left(~\sum_sK_s~\bigg|~\sum_s I_s~\right).
\]
Indeed while the intersection und sum of invariant ideals remain invariant
\[
  \bra{X}\left(\bigcap K_s\right)\ket{X}\subset \bigcap K_s
  \TAB\text{and}\TAB
  \bra{X}\left(\sum K_s\right)\ket{X}\subset \sum K_s
\]
the intersection and sum of covariance ideals may not always end up below the maximal covariance,
\[
  \bigcap I_s\subseteq \max\left( \frac{X}{X\left(\bigcap K_s\right)} \right)
  \TAB\text{and}\TAB
  \sum I_s\subseteq \max\left( \frac{X}{X\left(\sum K_s\right)} \right)~?
\]
Schematically that is the next possible kernel--covariance pair may lie just further beyond (as seen within the lattice of any pairs of ideals) such as
\begin{TIKZCD}[row sep=small]
  \Bigl(~K:~\ldots~\Big|~I\subset\max(\frac{X}{XK})~\checkmark~\Bigr) \\
  \Bigl(~(K_1+K_2):~\ldots~\Big|~ (I_1+I_2)\nsubseteq\max\left( \tfrac{X}{X(K_1+K_2)} \right)~\Bigr) \uar \\
  \Bigl(~K_s:~\ldots~\Big|~I_s\subset\max\left(\tfrac{X}{XK_s}\right)~\checkmark~\Bigr). \uar
\end{TIKZCD}
For example the requirement to have at least as much covariance as all the given covariance ideals \textbf{can force larger kernel} than just the sum of given kernel ideals,
or put more drastically there may exist \textbf{no connecting morphism} from each relative Cuntz--Pimsner algebra to the cokernel strand over the sum of kernel ideals. We will see such an example in \ref{EXAMPLE:JOIN-KERNEL-COVARIANCE-PAIRS} below involving already even just a pair of kernel--covariance pairs:
\begin{TIKZCD}[column sep=large]
  (X,A) \rar\dar &
    \left(\frac{X}{XK_1},\frac{A}{K_1}\right) \rar\dar&
    \left(\frac{X}{XK_2},\frac{A}{K_2}\right) \rar\dar&
    \left(\frac{X}{XK},\frac{A}{K}\right) \dar\\
  \TAB\mathstrut\ldots\mathstrut\TAB \dar&
    \PIMSNER(K_1,I_1) \dar\arrow[dr,dashed,lightgray,bend right=15,"\times"description]\arrow[r,dashed,lightgray,"\times"description]&
    \PIMSNER(K_2,I_2) \dar\rar&
    \PIMSNER(K,I) \arrow[from=ll,bend right=20,crossing over]\dar \\
  \TAB\mathstrut\ldots\mathstrut\TAB &
    \TAB\mathstrut\ldots\mathstrut\TAB &
    \TAB\mathstrut\ldots\mathstrut\TAB &
    \TAB\mathstrut\ldots\mathstrut\TAB
\end{TIKZCD}
Furthermore we note that from either description the order defines a partial order (as opposed to just a preorder) as easily seen from
\begin{gather*}
  \hspace{0.5cm}\left(\begin{smalltikzcd}
    (X,A)\rar& \CSTAR(X\cup A)=B
  \end{smalltikzcd}\right)
  \TAB\text{and}\TAB
  \left(\begin{smalltikzcd}
    (X,A)\rar& \CSTAR(X\cup A)=B'
  \end{smalltikzcd}\right):\\[2\jot]
  \left(\begin{tikzcd}[row sep=small]
    (X,A) \rar[equal]\dar& (X,A) \rar[equal]\dar& (X,A) \dar\\
    B \rar& B' \rar& B
  \end{tikzcd}\right) =
  \left(\begin{tikzcd}[row sep=small]
    (X,A) \rar[equal]\dar& (X,A) \dar\\
    B \rar[equal]& B
  \end{tikzcd}\right)
\end{gather*}
so one is a retract of the other and similarly the other way around,
or one may equivalently also argue using their kernel from the Toeplitz algebra
\[
  \left(\begin{tikzcd}
    \TOEPLITZ X\rar& B\rar& B'
  \end{tikzcd}\right)
  \TAB\implies\TAB
  \ker\Bigl(\TOEPLITZ X\to B\Bigr)
  \subset
  \ker\Bigl(\TOEPLITZ X\to B'\Bigr)
\]
from which they had been already the same quotient
\[
  \ker\Bigl(\TOEPLITZ X\to B\Bigr) =
  \ker\Bigl(\TOEPLITZ X\to B'\Bigr)
  \TAB\implies\TAB
  \left(\begin{tikzcd}
    \TOEPLITZ X\rar& B=B'
  \end{tikzcd}\right).
\]
Or one may now also argue using kernel--covariance pairs,
\[
  \Bigl(~K\subset L\subset K~\Big|~I\subset J\subset I~\Bigr)
  \TAB\implies\TAB
  \Bigl(~K=L~\Big|~I=J~\Bigr).
\]
Altogether we remark that the lattice of gauge-equivariant representations coincides entirely with the lattice of kernel--covariance pairs,
while suprema and infima of kernel--covariance pairs may lay only beyond of just the intersection and sum of their kernel and covariance ideals.
\end{BREAKREM}

With this in mind we continue on the remaining questions from section \ref{sec:RELATIVE-PIMSNER}:
More precisely, we wish to find suitable characterisations when morphisms exists between different quotient strands (based on kernel--covariance pairs)
\begin{TIKZCD}[column sep=1.5cm]
  \left(\frac{X}{XK},\frac{A}{K}\right) \rar\dar& \PIMSNER(K,0) \rar\dar& \PIMSNER(K,I) \rar\dar[dashed]& \PIMSNER(K,\max)\\
  \left(\frac{X}{XL},\frac{A}{L}\right) \rar& \PIMSNER(L,0) \rar& \PIMSNER(L,J) \rar& \PIMSNER(L,\max)
\end{TIKZCD}
and we warn ahead that these \textbf{won't always exist.}
Now at first we note that clearly the existence infers the inclusion of kernel ideals and we may for simplicity replace the original correspondence with the quotient
\begin{TIKZCD}[column sep=1cm]
  (X,A):=\left(\frac{X}{XK},\frac{A}{K}\right) \rar& \PIMSNER(K=0,0) \rar& \PIMSNER(K=0,I) \rar& \ldots
\end{TIKZCD}
In particular there always exist connecting morphism from at least the Toeplitz algebra and so also some further relative Cuntz-Pimsner algebras
\begin{TIKZCD}[column sep=1.5cm,row sep=1.5cm]
  \TOEPLITZ\left(\frac{X}{XK}\right) = \PIMSNER(K,0) \rar\dar\arrow[drr,"{\checkmark}"description]&
  \ldots\rar\arrow[dr,"{\checkmark}"description]&
  \PIMSNER(K,I=?) \rar\arrow[d,"{\checkmark}"description]&
  \ldots\arrow[dl,dashed,"{\times}"description]\\
  \TOEPLITZ\left(\frac{X}{XL}\right) = \PIMSNER(L,0) \arrow[rr]&&
  \PIMSNER(L,J) \rar& \PIMSNER(L,\max).
\end{TIKZCD}
As such the first question is to find the smallest relative Cuntz--Pimsner algebra from which connecting morphisms exist.
This may be now easily solved as
\[
  \SHORTTIKZCD{ \PIMSNER(K,I) \rar[dashed]& \PIMSNER(L,J) }\TAB\iff\TAB (I+K)\subset (J+L)
\]
and as such the largest covariance ideal (i.e.~within the maximal covariance) simply arises as intersection with the given covariance from the quotient,
\[
  I= \max\left( \tfrac{X}{XK} \right)\cap\bigl(A/K\to A/L\bigr)\inv J =\max\left( \tfrac{X}{XK} \right)\cap (J+K).
\]
Schematically the intersection can look something like this:
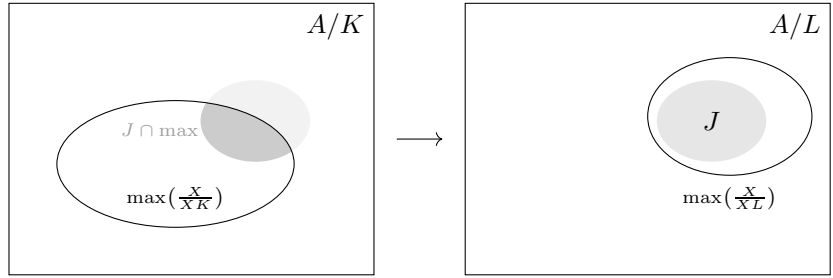
\begin{center}
  \begin{tikzpicture}[scale=0.6]
    \draw (-4,-3) rectangle (4,3) node[anchor=north east]{$A/K$};
    \draw[->] (4.5,0) -- (5.5,0);
    \fill[gray!10] (1.4,0.4) ellipse (1.2 and 0.9);
    \node[gray!75] at (-0.7,0.2) {\scriptsize$J\cap\max$};
    \begin{scope}
      \path[clip] (-0.35,-0.55) ellipse (2.6 and 1.4);
      \node at (-0.35,-1.3) {\scriptsize$\max\left(\tfrac{X}{XK}\right)$};
      \fill[gray!40] (1.4,0.4) ellipse (1.2 and 0.9);
    \end{scope}
    \draw (-0.35,-0.55) ellipse (2.6 and 1.4);
    \begin{scope}[shift={(10,0)}]
      \draw (-4,-3) rectangle (4,3) node[anchor=north east]{$A/L$};
      \draw (1.8,0.5) ellipse (1.8 and 1.3);
      \fill[gray!20] (1.4,0.4) node[black]{$J$} ellipse (1.2 and 0.9);
      \node at (1.8,-1.3) {\scriptsize$\max\left(\tfrac{X}{XL}\right)$};
    \end{scope}
  \end{tikzpicture}
\end{center}
The reader may easily find some examples with (using graph algebras as above):
\[
  \max\left(\tfrac{X}{XK}\right)\neq0:\TAB (J\cap\max)=0~~/~~0\neq(J\cap\max)\neq\max~~/~~(J\cap\max)=\max~~?
\]
Moreover one may now easily guess the \textbf{meet of kernel--covariance pairs:}
\[
  \bigwedge_s(~K_s~|~I_s~)= \left(~K=\bigcap K_s~\bigg|~I=\bigcap I_s\cap \max\left(\tfrac{X}{XK}\right)~\right)
\]
So far about connecting morphisms \textbf{from preceeding cokernel strands.}\\
The other direction however is more interesting:
That is a relative Cuntz--Pimsner does not necessarily connect to every following quotient correspondence
and in there not even beginning at every relative Cuntz-Pimsner algebra either,
\begin{TIKZCD}[column sep=1.5cm]
  \PIMSNER(K,0) \dar\rar& \PIMSNER(K,I)\rar
  \dar[dashed]
  \drar[dashed]&
  \PIMSNER(K,\max) \\
  \PIMSNER(L=?,0) \rar& \PIMSNER(L=?,J=?) \rar& \PIMSNER(L=?,\max).
\end{TIKZCD}
Note this also describes the lattice of gauge-invariant ideals for the given relative Cuntz--Pimsner algebra (simply as each such defines a quotient).\linebreak
We note however that as soon as it connects to another relative Cuntz--Pimsner algebra (in some following quotient correspondence) then it certainly also does so to the absolute Cuntz--Pimsner algebra for that quotient correspondence.\linebreak
As such this introduces an obstruction which may be now handle using \eqref{LATTICE-ISOMORPHISM}:
\[
  \PIMSNER(K,I)\leq \PIMSNER(L,\max)
  \TAB\iff\TAB
  J_{\min}:=\Bigl(A/K\to A/L\Bigr)I\subset\max\left(\tfrac{X}{XL}\right)
\]
This condition fails from time to time (we provide a simple example below).\\
In case this condition is met we obtain \textbf{as smallest solution}
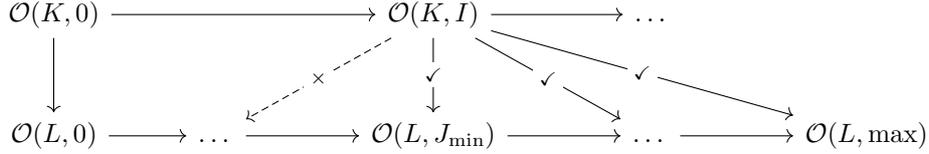
\begin{TIKZCD}[column sep=1.5cm]
  \PIMSNER(K,0) \dar\arrow[rr]&[-0.5cm]& \PIMSNER(K,I)
  \arrow[dl,"{\times}"description,dashed]
  \arrow[d,"{\checkmark}"description]
  \arrow[dr,"{\checkmark}"description]
  \arrow[drr,"{\checkmark}"description]
  \rar&\ldots \\
  \PIMSNER(L,0) \rar&\ldots\mathstrut\rar& \PIMSNER(L,J_{\min}) \rar&\ldots\mathstrut\rar& \PIMSNER(L,\max)
\end{TIKZCD}
while in case the condition fails then there simply is \textbf{no connecting morphism.}\linebreak
As such we also obtain the lattice for any relative Cuntz-Pimsner algebra
which is really \textbf{just our main result restated} (while this also generalizes O-pairs):

\begin{BREAKCOR}[Relative Cuntz-Pimsner algebra: gauge-invariant ideals]
  Consider a relative Cuntz-Pimsner algebra (as described in section \ref{sec:RELATIVE-PIMSNER})
  \begin{TIKZCD}
    (X,A)\rar&\left(\tfrac{X}{XK},\tfrac{A}{K}\right) \rar&\PIMSNER(K,0)\rar&\PIMSNER(K,I)
  \end{TIKZCD}
  for some kernel--covariance pair as in \eqref{COVARIANCE-IDEAL:QUOTIENT-PULLBACK} above.\\
  Then its lattice of gauge-invariant ideals simply runs over pairs as in \eqref{LATTICE-ISOMORPHISM}
  \[
  \Bigl(~K\subset L~\Big|~I+K\subset J+L~\Bigr)
  \iff
  \Bigl(~K\subset L~\Big|~J_{\min}\subset J\subset\max\left(\tfrac{X}{XL}\right)~\Bigr)
  \]
  or in words simply over all larger kernel--covariance pairs.
\end{BREAKCOR}

Let us give an example for when there is no connecting morphism:
\begin{BREAKEXM}[Graph correspondences: no connecting morphism]
  \label{EXAMPLE:JOIN-KERNEL-COVARIANCE-PAIRS}
  Consider as an example the following graph and its quotient graphs
  \vspace*{-\baselineskip/2}
  \[\begin{tikzcd}[column sep=1.5cm]
    \begin{gathered}
      K=0:\\
      \Bigl(~~a\to b~~\Bigr)\\[\baselineskip]
    \end{gathered}
    \rar&
    \begin{gathered}
      K=(a):\\
      \Bigl(~~b~~\Bigr)\\[\baselineskip]
    \end{gathered}
    \rar&
    \begin{gathered}
      K=(a\cup b):\\
      \Bigl(~~\emptyset~~\Bigr).\\[\baselineskip]
    \end{gathered}
  \end{tikzcd}\vspace*{-1.5\baselineskip}\]
  Then there is no connecting morphism for the absolute Cuntz-Pimsner algebras between the first and second quotient (by simply reading off covariance ideals):
  \begin{TIKZCD}[column sep=normal,row sep=1.5cm,arrows={crossing over}]
    \TOEPLITZ\Bigl(~ X=\ell^2(a\to b) ~\Bigr) \dar\rar &[0.5cm] \TOEPLITZ\Bigl(~ X=\ell^2(b) ~\Bigr) \dar[equal]\rar& \TOEPLITZ(X=0)\dar[equal] \\
    \PIMSNER\Bigl(~ X=\ell^2(a\to b) ~\Bigr)
    \arrow[urr]\arrow[r,"{\times}"description,dashed]
    &\PIMSNER\Bigl(~ X=\ell^2(b) ~\Bigr) & \PIMSNER(X=0).
  \end{TIKZCD}
  Indeed the obstruction fails for the covariance ideal:
  \begin{gather*}
    I_1=(b)=\reg\bigl(~a\to b~\bigr)\TAB\text{and}\TAB K_2=(a)\subset\her(a\to b):\\[2\jot]
    (A\to A/K_2) I_1=(b)\nsubseteq \reg\bigl(~\mathrm{quotient~graph}=b~\bigr).
  \end{gather*}
  The issue here is that the hereditary ideal is simply not saturated.
  Alternatively one may note that the first defines the simple algebra of $2{\times}2$ matrices.\\
  In particular, we obtain for the \textbf{join of kernel--covariance pairs}
  \[
    \Bigl(~K_1=0~\Big|~I_1=(b)~\Bigr) \vee
    \Bigl(~K_2=(a)~\Big|~I_2=(a)~\Bigr)
    = \Bigl(~K=(a\cup b)~\Big|~\ldots~\Bigr).
  \]
  In other words, the join as \textbf{next possible kernel--covariance pair} lies only beyond of just the sum of kernel and covariance ideals.
  So we found an example for the issue (mentioned further above) that suprema and infima will be generally \textbf{beyond just intersections and sums of ideals.}
\end{BREAKEXM}

We finish this section with a widely missed identification of Katsura's work:
That is we clarify how \textbf{Katsura's T-pairs (and O-pairs)} are nothing but the pullback version of our kernel--covariance pairs from above.
More precisely, we elaborate \textbf{Katsura's cryptic requirement}
\[
  J(K):=\Bigl\{~a\in A~\Big|~\ldots~\text{and}~ aX\inv(K)\subset K~\Bigr\}:\TAB\TAB K\subset I\subset J(K)
\]
and how this defines a \textbf{translation of the constraint} from proposition \ref{FAITHFUL-REPRESENTATIONS:MAXIMAL-COVARIANCE}:\\
That is any covariance for an embedding into an operator algebra is necessarily orthogonal to the kernel (for its left action) which read in our case
\[
  \cov\left(~\tfrac{X}{XK}\to B~\right)\perp\ker\left(~\tfrac{A}{K}\act \tfrac{X}{XK}~\right)
\]
and as such these cannot exceed the maximal covariance (a.k.a. Katsura's ideal)
\[
  \cov\left(~\tfrac{X}{XK}\to B~\right)\subset
  \left[\left(\tfrac{X}{XK}\right)\left(\tfrac{X}{XK}\right)^*\cap
  \ker\left(\tfrac{A}{K}\act\tfrac{X}{XK}\right)^\perp\right]=
  \max\left(\tfrac{X}{XK}\right).
\]
As such we found instead our kernel--covariance pairs as given \textbf{by invariant ideals as kernel} (which defines some sort of discrete range for kernel ideals)
\[
  K\ILEQ A:\TAB\TAB X^*KX\subset K
\]
together with \textbf{ideals bounded from above} as covariance (which defines an upper bound on the range of covariance ideals)
\[
  I\ILEQ A/K:\TAB\TAB 0\subset I\subset\max\left(\tfrac{X}{XK}\right)
\]
while we found in the \textbf{second half of our classification} that each such kernel--covariance pair indeed arises itself (more precisely theorem~\ref{THEOREM:RELATIVE-PIMSNER:KERNEL-COVARIANCE}).\linebreak
In order to establish Katsura's requirement we first note the obvious
\[
  I=\Bigl(A\to A/K\Bigr)\inv J
  \TAB\iff\TAB
  K\subset I
\]
and one the other hand the inclusion (for quotient maps)
\[
  J\subset \max\left(\tfrac{X}{XK}\right)
  \TAB\iff\TAB
  \Bigl(A\to A/K\Bigr)\inv J\subset\Bigl(A\to A/K\Bigr)\inv\max\left(\tfrac{X}{XK}\right).
\]
As such \textbf{Katsura's condition} simply states (see \cite[lemma~5.2]{KATSURA-GAUGE-IDEALS})
\[
  J(K)=(A\to A/K)\inv \max\left(\tfrac{X}{XK}\right)
\]
for which it further suffices to verify (see also \cite[lemma~5.2]{KATSURA-GAUGE-IDEALS})
\[
  \Bigl\{~aX\inv(K)\subset K~\Bigr\}=\Bigl(A\to A/K\Bigr)\inv\ker\left(\tfrac{A}{K}\act\tfrac{X}{XK}\right)^\perp
\]
since the dotted condition represents nothing but compactly acting coefficients.
This is however now easily verified:
Consider for this the pullback (which contains the same information)
\begin{gather*}
  \ker\left(A\act\tfrac{X}{XK}\right) =
  \Bigl(A\to A/K\Bigr)\inv\ker\left(\tfrac{A}{K}\act\tfrac{X}{XK}\right):\\[3\jot]
  \Bigl(A\to A/K\Bigr)\Bigl(A\to A/K\Bigr)\inv(~\ldots~)=(~\ldots~)
\end{gather*}
and which further reads (as in proposition \ref{SHORT-EXACT:ADJOINTABLES})
\[
  \ker\left(A\act\tfrac{X}{XK}\right) = \left\{~a\tfrac{X}{XK}=0~\right\} = \{aX\subset XK\} = \{X^* aX\subset K\} = X\inv(K).
\]
Put together we obtain the desired relation for the orthogonal complement.\\
We note that the relation has been worked out by Katsura in \cite[lemma~5.2]{KATSURA-GAUGE-IDEALS}\linebreak
which however has been not continued further on: Katsura chose to work with the cryptic requirement instead of pursuing their \mbox{kernel--covariance} counterpart.
Possibly because they got \textbf{only partially recognized} as covariance ideals.

Finally the author notes that the results here arose from a more detailed study of \cite{KATSURA-GAUGE-IDEALS} which builds on \cite{FOWLER-MUHLY-RAEBURN} and further \cite{KAJIWARA-PINZARI-WATATANI} and \cite{PIMSNER-1997}.
More precisely, the author realized the relations drawn in \cite[lemma~5.10]{KATSURA-GAUGE-IDEALS} (which extend \cite[lemma~2.9]{FOWLER-MUHLY-RAEBURN})
as a partial result on \textbf{categorical kernel} and cokernel morphisms
which led to \textbf{their intrinsic characterization}
(in theorem \ref{THEOREM:KERNEL-COKERNEL-MORPHISMS})
and so also on the \textbf{range of possible kernel ideals.}\\
On the other hand the author realized the first observation made in \cite[proposition~3.3]{KATSURA-SEMINAL} as an intrinsic characterisation on the range of \textbf{possible covariance ideals} for the induced representation on the quotient.\linebreak
These let the author to \textbf{systematically employ} such kernel--covariance pairs,\linebreak
which allowed on one hand to handle \textbf{the general version} of the gauge-invariant uniqueness-theorem \textbf{by reduction to the faithful case} which follows from the \textbf{sleek and simplifying proof} by Evgenios Kakariadis in \cite{KAKARIADIS-GAUGE-THEOREM} (which draws from the second observation made in \cite[proposition~3.3]{KATSURA-SEMINAL}) and on the other hand the \textbf{critical observation} made by Takeshi Katsura in \cite[lemma~4.7]{KATSURA-SEMINAL} in his seminal paper from 2004,
which led the author to \textbf{retrieve kernel--covariance pairs}
from their relative Cuntz-Pimsner algebra (in theorem \ref{THEOREM:RELATIVE-PIMSNER:KERNEL-COVARIANCE}).
\textbf{The main difference} however is that we didn't need to build any ad-hoc semi-kind-of categorical pushout for correspondences as was handled in \cite{KATSURA-GAUGE-IDEALS}. Instead it is all based on \textbf{the simple idea of reduction} to faithful representations using \textbf{kernel and cokernel morphisms.}

\section{Pimsner dilations}
\label{sec:PIMSNER-DILATIONS}

This final section introduces the notion of dilations and verifies the existence of the maximal dilation \textbf{as Hilbert bimodule.}
We further reveal Katsura's construction as a \textbf{particular nonmaximal dilation} and illustrate the lack of minimal dilations.
Meanwhile, the author would like to take this opportunity \textbf{to thank Ralf Meyer} for sharing his enlightening perspective on the Pimsner dilation \textbf{as maximal dilation.}

We begin with the concept of dilations:
More precisely that is any \textbf{gauge-equivariant} factorisation over some \textbf{intermediate correspondence} such as
\begin{TIKZCD}[column sep=large]
  & (Y,B)\arrow[dr,dashed] \\
  (X,A)\arrow[ur,dashed]\arrow[rr]&&\PIMSNER(K,I)
\end{TIKZCD}
where the gauge-equivariance boils down to simply
\[
  \begin{smalltikzcd}
    Y\rar&\PIMSNER(K,I)(1)
  \end{smalltikzcd}
  \TAB\text{and}\TAB
  \begin{smalltikzcd}
    B\rar&\PIMSNER(K,I)(0).
  \end{smalltikzcd}
\]
As the original correspondence generates the relative Cuntz--Pimsner algebra, so does also the intermediate one
\[
  \CSTAR(X\cup A)=\PIMSNER(K,I)\TAB\implies\TAB\CSTAR(Y\cup B)=\PIMSNER(K,I)
\]
whence the factorisation defines a relative Cuntz--Pimsner algebra itself:
\[
  \begin{tikzcd}
    (Y,B)\rar&\PIMSNER(K,I)=\PIMSNER\Bigl(Y,B\Big|~L=?~J=?~\Bigr)
  \end{tikzcd}
\]
As such the task is now to find dilations which generate the relative Cuntz--Pimsner algebra as an \textbf{absolute Cuntz--Pimsner algebra:}
\begin{TIKZCD}[column sep=0.7cm]
  (X,A) \rar[dashed]& \Bigl(~Y=?,B=?~\Bigr) \rar[dashed]& \PIMSNER(K,I)=\PIMSNER\Bigl(Y,B~\Big|~L=0,J=\max~\Bigr)~?
\end{TIKZCD}
As the kernel should be trivial we have no choice than to look within the relative Cuntz--Pimsner algebra itself. Also we may as usual assume that our original correspondence embeds itself (simply by replacing our original correspondence by its quotient correspondence).
As such our intermediate correspondence necessarily arises \textbf{as an intermediate subspace}
\[
  X\subset Y\subset \PIMSNER(K=0,I)(1)\TAB\text{and}\TAB A\subset B\subset\PIMSNER(K=0,I)(0).
\]
On the other hand we found (as a well-known description) that the absolute Cuntz--Pimsner algebra arises as the smallest gauge-equivariant quotient for which the coefficient algebra faithfully embeds into (and so also the correspondence) or in other words the coefficient algebra detects the gauge-invariant ideals within the absolute Cuntz--Pimsner algebra. So we aim to find an \textbf{intermediate subspace which detects} the remaining gauge-invariant ideals
\[
  J\ILEQ\PIMSNER(K=0,I):\TAB\TAB B\cap J=0\implies J=0~?
\]
By our main result (theorem \ref{THEOREM:RELATIVE-PIMSNER:KERNEL-COVARIANCE}) we found a parametrisation for the entire lattice of gauge-equivariant representations
\begin{TIKZCD}
  \PIMSNER(K=0,I=0)\dar\rar       & \PIMSNER(K\neq0,I=0)  \dar\rar         & \mathstrut\ldots\mathstrut \dar\\
  \PIMSNER(K=0,I)\dar\rar[dashed] & \mathstrut\ldots\mathstrut    \dar\rar[dashed] & \mathstrut\ldots\mathstrut \dar\\
  \PIMSNER(K=0,\max)\rar[dashed]  & \mathstrut\ldots\mathstrut \rar[dashed]     & \mathstrut\ldots\mathstrut
\end{TIKZCD}
given by kernel--covariance pairs
\[
  \Bigl(~~K\ILEQ A:~X^*KX\subset K~~\Big|~~I\ILEQ\tfrac{A}{K}:~I\subset\max\left(\tfrac{X}{XK}\right)~~\Bigr).
\]
and so also of gauge-invariant ideals (within the Toeplitz algebra) whence also for the relative Cuntz-Pimsner algebra.
Our original coefficient algebra however already detects the kernel component:
\[
  A\cap \TOEPLITZ(X,A)(K\subset A)\TOEPLITZ(X,A)=0\TAB\implies\TAB K=0\TAB\checkmark
\]
As such the only gauge-invariant ideals which our original coefficient algebra cannot detect are precisely the covariance ideals (with trivial kernel component)
\[
  I\subset\max(X,A):\TAB\TAB A\cap\TOEPLITZ(X,A)\mediumMATRIX{I\\&0}\TOEPLITZ(X,A)=0.
\]
Furthermore we have already taken a quotient for some covariance (which brought us to our relative Cuntz--Pimsner algebra) and so the remaining gauge-invariant ideals arise as remaining quotients
\begin{TIKZCD}
  \PIMSNER(K=0;I=0)\rar&\PIMSNER(K=0,I)\rar&~\ldots~\rar&\PIMSNER(K=0,\max).
\end{TIKZCD}
As such we need to find an intermediate coefficient algebra which detects all of the remaining covariance ideals beyond the given one $I\subset J\subset\max(X,A)$:
\begin{equation}\label{DILATIONS:COVARIANCE-DETECTION}
  A\subset B\subset\PIMSNER(K,I)(0):\TAB\TAB B\cap\PIMSNER(K,I)\mediumMATRIX{J\\&0}\PIMSNER(K,I)\neq0
\end{equation}
with an intermediate subspace as correspondence (as described in section \ref{sec:CORRESPONDENCES})
\[
  X\subset Y\subset\PIMSNER(K,I)(1):\TAB\TAB Y^*Y\subset B,\TAB BY\subset Y,\TAB YB\subset Y.
\]
As such one may first choose an intermediate subalgebra which detects the remaining covariance ideals (and if desired also a chosen subspace) and from there simply enlarge the chosen subalgebra to form a correspondence, for instance as the smallest correspondence above:
\[
  \bigcap\Bigl\{~\big(Y_0\subset Y\big|B_0\subset B\big)~\Big|~Y^*Y\subset B,~BY\subset Y,~YB\subset Y~\Bigr\}.
\]
Indeed the intersection of any class of correspondences forms a correspondence:
\[
  \hspace{-1cm}\left(~Y=\bigcap Y_n~\Big|~ B=\bigcap B_n~\right):\TAB\TAB Y_n^*Y_n\subset B_n,~\ldots~\implies~Y^*Y\subset B,~\ldots
\]
We meanwhile need to verify that there always exists one above:
For this we simply consider the maximal dilation \textbf{(also known as Pimsner dilation)}
\[
  \Bigl(~Y=\PIMSNER(K,I)(1)~\Big|~B=\PIMSNER(K,I)(0)~\Bigr).
\]
Indeed the maximal dilation defines \textbf{even a Hilbert bimodule} and so also a correspondence simply as Fourier spaces define Fell bundles (confer section \ref{sec:GAUGE-ACTIONS}):
\begin{align*}
  \PIMSNER(K,I)(-1)\PIMSNER(K,I)(1)&\subset\PIMSNER(K,I)(-1+1=0),\\
  \PIMSNER(K,I)(0)\PIMSNER(K,I)(1) &\subset\PIMSNER(K,I)(0+1=1),\TAB\ldots
\end{align*}
On the other hand the maximal dilation is also easily seen to detect all of the remaining covariance simply as each is generated from the fixed point algebra:
\begin{equation}\label{MAXIMAL-DILATION:COVARIANCE-DETECTION}
  \begin{gathered}
    \mediumMATRIX{J\\&0}\subset A\mediumMATRIX{1\\&1} + XX^*\mediumMATRIX{0\\&1}\subset \PIMSNER(K,I)(0)\\[2\jot]
    \hspace{-1cm}\implies\TAB
    \mediumMATRIX{J\\&0}\subset\PIMSNER(K,I)(0)\cap\biggl[~\PIMSNER(K,I)\mediumMATRIX{J\\&0}\PIMSNER(K,I)~\biggr]
  \end{gathered}
\end{equation}
\textbf{In fact one may argue more generally:}
Given any operator algebra with a given circle action and consider its fixed point algebra (as in section \ref{sec:GAUGE-ACTIONS})
\[
  \TORUS\act B:\TAB\TAB B(n=0)=\Bigl\{~b~\Big|~\bigl(b(z):=z\act b~\bigr)\equiv b~\Bigr\}.
\]
Then its fixed point algebra detects every gauge-invariant subalgebra:
\[
  \Bigl(~A=\TORUS\act A~\Bigr)\subset B:\TAB\TAB A\cap B(0)=0\implies A=0.
\]
Indeed this follows using the conditional expectation (confer section \ref{sec:GAUGE-ACTIONS})
\[
  \TORUS\act B:\TAB E(b)=\int_\TORUS \Big(b(z)=z\act b\Big)\dd z.
\]
As the subalgebra is gauge-invariant the conditional expectation does not leave the subalgebra (using its construction as Bochner integral)
\[
  \Bigl(~A=\TORUS\act A~\Bigr) \TAB\implies\TAB
  E(A) = \int_\TORUS \Bigl(~z\act A\subset A~\Bigr)\dd z \subset A.
\]
On the other hand, every operator algebra is spanned by its positive portion,
\[
  A=\pos(A) - \pos(A) + i\pos(A) - i\pos(A),\TAB\TAB\pos(A):=\{0\leq a\in A\}.\hspace{-1cm}
\]
The conditional expectation (given by averaging) is however faithful on the positive portion and as such we have have found the detection
\[
  E(\pos A)\subset A\cap B(0)=0
  \TAB\implies\TAB
  \pos(A)=0
  \TAB\implies\TAB
  A=0.
\]
This well-known technique is \textbf{quite worthwhile in other context.}\\
As such we have found the following familiar result with ease
(note there was basically nothing left to prove anymore):

\begin{BREAKTHM}[Maximal dilation: absolute Cuntz--Pimsner algebra]
  The maximal dilation realises relative Cuntz--Pimsner algebras as absolute one
  \begin{TIKZCD}[column sep=large]
    \PIMSNER(K,I)=\PIMSNER\Bigl(~Y=\PIMSNER(K,I)(1)~\Big|~ B=\PIMSNER(K,I)(0)~\Bigr).
  \end{TIKZCD}
  and further defines the maximal Hilbert bimodule.
\end{BREAKTHM}
\begin{proof}
  This is now an immediate consequence from \eqref{MAXIMAL-DILATION:COVARIANCE-DETECTION} satisfying \eqref{DILATIONS:COVARIANCE-DETECTION}.
\end{proof}

This dilation however is rather large in the sense that there is not much control over its behavior (besides its universal description). For instance one may think of the maximal dilation similar to maximal Furstenberg boundary.
Instead we therefore seek for some dilation small enough \textbf{to be tractable combinatorially} while large enough to detect covariance.
As explained above, one may for this simply begin with a small subalgebra which detects covariance and simply enlarge the chosen subalgebra to form a correspondence.
In practice one may for instance \textbf{attempt to run the algorithm}
\begin{equation}
  \label{DILATION:ALGORITHM}
  \begin{gathered}
    Y= BX + X+ XB +BXB
    \TAB\implies\TAB
    B'= B + Y^*Y\\[3\jot]
    \implies\TAB
    Y'= B'Y + Y+ YB'
    \TAB\implies\TAB\ldots\hspace{1cm}
  \end{gathered}
\end{equation}
with implicit closed linear span as usual.\\
While there always exist a smallest correspondence above (as we found above) \textbf{this process may never halt} and whence leave us clueless about its combinatorial behavior.
\textbf{In good cases however,} the algorithm halts and thus allows for its combinatorial description.
This happens in particular for the canonical subalgebra \textbf{given by the maximal covariance itself:}
\begin{equation}
  \label{KATSURA-DILATION:COVARIANCE-DETECTION}
  A\subset\biggl[~B=A+\mediumMATRIX{\max(X,A)\\&0}~\biggr]\subset \PIMSNER(K,I)(0)
\end{equation}
whose sum defines a subalgebra since
\[
  \mediumMATRIX{\max(X,A)\\&0}A\mediumMATRIX{1\\&1}\subset\mediumMATRIX{\max(X,A)A\\&0}\subset \mediumMATRIX{\max(X,A)\\&0}
\]
which is nothin but the relation (from proposition \ref{HIGHER-MIXED})
\[
  (\phi-\tau)\max(X,A)\cdot\phi(A) = (\phi-\tau)\Bigl(\max(X,A)A\Bigr)\subset(\phi-\tau)\max(X,A).
\]
On the other hand its left action \textbf{keeps the space invariant}
\begin{equation}
  \label{KATSURA-DILATION:INVARIANT-SPACE}
  \mediumMATRIX{\max(X,A)\\&0}X\subset(A + XX^*)X\subset X
\end{equation}
and as such the algorithm halts right after the first round,
\[
  Y= X + X\mediumMATRIX{\max(X,A)\\&0}=XB:\TAB\TAB Y^*Y=BX^*XB\subset B
\]
and as such we got the \textbf{canonical dilation as a combinatorial object.}\\
Using the left and right shift (as in section \ref{sec:FOCK-SPACE}) we further even note
\begin{equation}
  \label{KATSURA-DILATION:LEFT-ACTION=TRIVIAL}
  \mediumMATRIX{\max(X,A)\\&0} \act X = \max(X,A)(1-RL)\cdot XR = 0
\end{equation}
which is nothing but the obvious relation (see also proposition \ref{HIGHER-MIXED})
\[
  \tau(XX^*)\tau(X)=\tau(XX^* X)\TAB\implies\TAB(\phi-\tau)\Bigl(A\cap XX^*\Bigr)\tau(X)=0.
\]
This resambles Katsura's construction from \cite{KATSURA-SEMINAL}
and so we refer to the canonical dilation given by the maximal covariance as \textbf{Katsura dilation}
(and note also here that there was basically nothing left to prove anymore):

\begin{BREAKTHM}[Katsura dilation: absolute Cuntz--Pimsner algebra]\label{THEOREM:KATSURA-DILATION=COVARIANCE-DETECTION}
  The canonical dilation given by the maximal covariance realises a relative Cuntz--Pimsner algebra as an absolute Cuntz--Pimsner algebra
  \[
    \PIMSNER(K=0,I)=\PIMSNER\biggl(~Y = X + X\mediumMATRIX{\max(X,A)\\&0}~\Big|~B=A+\mediumMATRIX{\max(X,A)\\&0}~\biggr)
  \]
  and the analogous dilation for kernel--covariance pairs with kernel ideal.\\
  This dilation may well fail to define a minimal dilation (detecting covariance)
  and even if minimal, it generally fails to be the only minimal dilation.
\end{BREAKTHM}
\begin{proof}
  This is now an immediate consequence from \eqref{KATSURA-DILATION:COVARIANCE-DETECTION} satisfying \eqref{DILATIONS:COVARIANCE-DETECTION}.\\
  We further provide examples for the failure of minimality in example \ref{EXAMPLE:MINIMAL-DILATIONS:FAILURE}.
\end{proof}

\begin{BREAKCOR}[Katsura dilation: intrinsic description]
  \label{THEOREM:KATSURA-DILATION:INTRINSIC-DESCRIPTION}
  The canonical dilation given by the maximal covariance allows the intrinsic description
  as the operator algebra freely generated by their abstract copies
  \[
    \biggl(~A=A\mediumMATRIX{1\\&1}~\biggr)
    \cup
    \biggl(~\max(X,A)/I=\mediumMATRIX{\max(X,A)\\&0}~\biggr)
  \]
  with multiplication given by
  \[
    A\mediumMATRIX{1\\&1}\cdot\mediumMATRIX{\max(X,A)\\&0}
    \subset
    \mediumMATRIX{A\max(X,A)\\&0}
    \subset
    \mediumMATRIX{\max(X,A)\\&0}
  \]
  and similarly for the correspondence itself.\\
  The analogous expression holds for kernel--covariance pairs with kernel ideal.\\
  This further recovers the particular description from \cite[definition~6.1]{KATSURA-GAUGE-IDEALS}.
\end{BREAKCOR}
\begin{proof}
  By part of our main result (the nontrivial part of theorem \ref{THEOREM:RELATIVE-PIMSNER:KERNEL-COVARIANCE}) we found that the relative Cuntz--Pimsner algebra does not introduce additional kernel which is the faithful copy of the coefficient algebra (as a familiar result):
  \[
    \begin{smalltikzcd}
      (X,A)\rar&\PIMSNER(K=0,I):
    \end{smalltikzcd}\TAB\TAB
    A= A\mediumMATRIX{1\\&1}.
  \]
  On the other hand the relative Cuntz--Pimsner algebra also does not introduce additional covariance (asides the already given covariance) which reads
  \[
    \mediumMATRIX{A\cap XX^*\\&0}\cap\TOEPLITZ(X,A)\mediumMATRIX{I\\&0}\TOEPLITZ(X,A)=\mediumMATRIX{I\\&0}.
  \]
  and as such the added maximal covariance defines a faithful copy up to
  \[
    \max(X,A)/I = \mediumMATRIX{\max(X,A)\\&0}\subset\PIMSNER(K=0,I).
  \]
  Note that the maximal covariance absorbs the coefficient algebra and as such their sum already defines a closed and thus complete operator algebra
  (see the quick proof \eqref{CLOSED-SUM:SUBALGEBRA-PLUS-IDEAL} on the sum of an algebra and ideal from section \ref{sec:GAUGE-THEOREM}):
  \[
    A\mediumMATRIX{1\\&1} + \mediumMATRIX{\max(X,A)\\&0} = \overline{A\mediumMATRIX{1\\&1} + \mediumMATRIX{\max(X,A)\\&0}}\subset\PIMSNER(K=0,I).
  \]
  In fact this holds in any representation as also their universal:
  \begin{gather*}
    \CSTAR\Bigl(~ A\cup\max(X,A)/I~\Big|~A\max(X,A)/I\subset\max(X,A)/I~\Bigr)\\[2\jot]
     = A + \max(X,A)/I = \overline{A + \max(X,A)/I}.
  \end{gather*}
  As such any concrete representation which provides faithful disjoint copies for the coefficient algebra and the maximal covariance (mod covariance ideal) defines a faithful representation for their freely generated operator algebra:\\
  \[
    A\cap\max(X,A)/I=0\subset B\TAB\implies\TAB \CSTAR(A\cup\max(X,A)/I)\subset B
  \]
  Indeed this simply follows by some basic linear algebra.
  This holds in particular for their copy in the relative Cuntz--Pimsner algebra
  \[
    A\mediumMATRIX{1\\&1}\cap\mediumMATRIX{\max(X,A)\\&0} = \mediumMATRIX{\ker(A\act X)\\&0}\cap \mediumMATRIX{\max(X,A)\\&0} = 0
  \]
  where we have used their trivial intersection
  \[
    \ker(A\act X)\cap \max(X,A)\subset\ker(A\act X)\cap \ker(A\act X)^\perp=0
  \]
  and as such the dilation arises as universal representation
  \begin{gather*}
    \CSTAR\Bigl(~ A\cup\max(X,A)/I~\Big|~A\max(X,A)/I\subset\max(X,A)/I ~\Bigr) \\[2\jot]
    = A\mediumMATRIX{1\\&1} + \mediumMATRIX{\max(X,A)\\&0} \subset\PIMSNER(K=0,I).
  \end{gather*}
  Note this defines a quite general argument which applies also in other context.\\
  Finally Katsura's description is nothing but the isomorphism
  \[
    \CSTAR\Bigl(~A\cup M~\Big|~AM\subset M~\Bigr) = \Bigl\{~a\oplus(\im a+m)\in A\oplus (\im A+M\subset B)~\Bigr\}
  \]
  which simply enforces faithful disjoint copies for any $A\to B$:
  \[
    A = A(1\oplus1),\TAB
    A(1\oplus1)\cap (0\oplus M) = 0,\TAB
    M = 0\oplus(M\subset B).
  \]
  There is however nothing special about this choice of formal description.\\
  Instead it is the universal description as freely generated copies with one absorbing which captures its properties:
  \begin{gather*}
    \CSTAR\Bigl(~ A\cup\max(X,A)/I~\Big|~A\max(X,A)/I\subset\max(X,A)/I ~\Bigr) \\[2\jot]
    = A\mediumMATRIX{1\\&1} + \mediumMATRIX{\max(X,A)\\&0} \subset\PIMSNER(K=0,I).
  \end{gather*}
  The reader may now similarly argue for the correspondence.
\end{proof}

We may now easily recover the classical result that
any gauge-equivariant quotient for some (possibly relative) graph algebra \textbf{remains a graph algebra.}\linebreak
For this we first recall that any quotient correspondence (as kernel component)
arises as a quotient graph (confer example \ref{GRAPH-CORRESPONDENCES:QUOTIENT-GRAPHS})
\begin{gather*}
  \hspace{1cm}A/K=c_0(V)/c_0\Bigl(~H=\hereditary~\Bigr)=c_0\Bigl(~W=V\setminus H~\Bigr)\\[2\jot]
  X/XK=\ell^2(E)/\ell^2(EH)=\ell^2\Bigl(~F:=WE~\Bigr)
\end{gather*}
and as such we may replace the original graph by the quotient graph
\[
  X=\ell^2(E:=F),\TAB A=c_0(V:=W)\TAB\implies\TAB\PIMSNER(K=0,I).\hspace{-1cm}
\]
On the other hand recall that any covariance ideal for a graph (in our case the quotient graph)
arises simply as a regular set of vertices (confer example \ref{GRAPH-CORRESPONDENCES:MAXIMAL-COVARIANCE}):
\[
  \max(X,A)=c_0\Bigl(~\regular~\Bigr)\TAB\implies\TAB I=c_0\Bigl(~R\subset\regular~\Bigr).
\]
With this notation in mind we may now find the canonical dilation as a graph.
We note for this that as it was given \textbf{by the algorithm above} we may recover the canonical dilation as a combinatorial object \textbf{from the original data,}\linebreak
which boils down in our case to the canonical dilation arising as a graph:

\begin{BREAKCOR}[Katsura dilation: graph correspondences]
  \label{THEOREM:KATSURA-DILATION:GRAPH-CORRESPONDENCE}
  The canonical dilation given by the maximal covariance as in \eqref{KATSURA-DILATION:COVARIANCE-DETECTION}
  \[
    \biggl(~XB = X + X\mediumMATRIX{\max(X,A)\\&0}~\Big|~B=A+\mediumMATRIX{\max(X,A)\\&0}~\biggr)
  \]
  arises as the following canonical graph (with notation from above):\\
  Its coefficient algebra arises as the orthogonal sum of vertices
  \[
    W=\singular\mediumMATRIX{1\\&1} \cup \mediumMATRIX{0\\&\regular} \cup \mediumMATRIX{\regular\setminus R\\&0}
  \]
  together with the correspondence given by the graph
  \[
    EW=E\biggl(~\singular\mediumMATRIX{1\\&1} \cup \mediumMATRIX{0\\&\regular} \cup \mediumMATRIX{\regular\setminus R\\&0}~\biggr)
  \]
  and its left action given by (whence defining the range of edges)
  \[
  \mediumMATRIX{a\\&a}EW=(aE)W
    \TAB\text{and}\TAB
    \mediumMATRIX{0\\&b}EW=\mediumMATRIX{b\\&b}EW=(bE)W
  \]
  while trivially acting for the left over last summand.\\
  As such the graph reads in more classical terms
  \begin{gather*}
    W=\Bigl(~\mathrm{all~vertices}=\singular\cup\regular~\Bigr)\amalg\Bigl(~\regular\setminus R~\Bigr)\\[2\jot]
    EW = E\Bigl(~\mathrm{all~vertices}~\Bigr)\amalg E\Bigl(~\regular\setminus R~\Bigr)
  \end{gather*}
  with range and source map given by
  \begin{gather*}
    s(\BLANK\amalg\emptyset)=s(\BLANK)\amalg\emptyset
    \TAB\text{and}\TAB
    s(\emptyset\amalg\BLANK)=\emptyset\amalg s(\BLANK),\\[2\jot]
    r(\BLANK\amalg\emptyset)=r(\BLANK)\amalg\emptyset= r(\emptyset\amalg\BLANK).
  \end{gather*}
  Note that the combined summands basically recover the original graph
  whereas the last provides an additional copy to make up for the maximal covariance.\linebreak
  On the other hand the edges \textbf{all point into the original copy of vertices}.
  As such this recovers the familiar construction for graph algebras:
  Any gauge-equivariant quotient arises as a graph algebra itself.
\end{BREAKCOR}
\begin{proof}In order to find the canonical dilation as a graph it suffices to recover its coefficient algebra as an \textbf{orthogonal sum of vertices} (see example \ref{GRAPH-CORRESPONDENCES:TENSOR-POWERS}).\linebreak
  In our case the coefficient algebra already reads as a sum of vertices
  \[
    A\mediumMATRIX{1\\&1} + \mediumMATRIX{\max(X,A)\\&0} = c_0\bigl(~\vertices~\bigr)\mediumMATRIX{1\\&1} + \mediumMATRIX{c_0\bigl(~\regular~\bigr)\\&0}
  \]
  which we may decompose now further into an orthogonal sum:\\
  First the singular vertices (that is the nonregular ones) are trivially disjoint from the regular ones and as such define an orthogonal summand,
  \[
    \vertices=\singular\cup\regular:\TAB\TAB\singular\mediumMATRIX{1\\&1}\perp\mediumMATRIX{\regular\\&\regular}.
  \]
  On the other hand the sum on regular vertices may be also taken as
  \[
    \regular\mediumMATRIX{1\\&1} + \mediumMATRIX{\regular\\&0}
    = \mediumMATRIX{0\\&\regular} + \mediumMATRIX{\regular\\&0}
  \]
  whose summands belong to the relative Cuntz--Pimsner algebra since
  \[
    \tau(XX^*)=\mediumMATRIX{0\\&XX^*}\TAB\text{and}\TAB (\phi-\tau)\bigl(A\cap XX^*\bigr) = \mediumMATRIX{A\cap XX^*\\&0}
  \]
  which is available only for the compactly acting portion!
  The latter summand vanishes precisely for the given covariance (see theorem \ref{THEOREM:RELATIVE-PIMSNER:KERNEL-COVARIANCE} or corollary \ref{THEOREM:KATSURA-DILATION:INTRINSIC-DESCRIPTION})
  \[
    \mediumMATRIX{\regular\setminus R\\&0}\neq0\TAB\text{and}\TAB \mediumMATRIX{R\\&0} = 0
  \]
  and as such the sum reduces to the \textbf{non-zero vertices}
  \[
    \mediumMATRIX{0\\&\regular} + \mediumMATRIX{\regular\\&0}
    = \mediumMATRIX{0\\&\regular} + \mediumMATRIX{\regular\setminus R\\&0}.
  \]
  As such the coefficient algebra for our dilation decomposes into
  \begin{gather*}
    \vertices\mediumMATRIX{1\\&1} + \mediumMATRIX{\regular\\&0}
    = \singular\mediumMATRIX{1\\&1} + \mediumMATRIX{0\\&\regular} + \mediumMATRIX{\regular\setminus R\\&0}
  \end{gather*}
  and so we have found the \textbf{vertices for our graph correspondence.}\\
  We may now simply read off the edges from our correspondence as
  \begin{gather*}
    X + X\mediumMATRIX{\max(X,A)\\&0}
    = \ell^2(E)\biggl(~\singular\mediumMATRIX{1\\&1} + \mediumMATRIX{0\\&\regular} + \mediumMATRIX{\regular\setminus R\\&0}~\biggr)
  \end{gather*}
  with the source of edges as evident.
  On the other hand its left action reads (using the induced morphism on compact operators from \ref{HIGHER-MIXED})
  \begin{gather*}
    \mediumMATRIX{a\\&a}EW=\phi(a)\tau(E)W=\tau(aE)W=(aE)W \\[2\jot]
    \mediumMATRIX{0\\&b}EW=\tau(b)\tau(E)W=\tau(bE)W=(bE)W \\[2\jot]
    \mediumMATRIX{c\\&0}EW=(\phi-\tau)c\tau(E)W=0
  \end{gather*}
  and so we have found the desired graph.
\end{proof}

We illustrate the canonical dilation for the following prominent graph:

\begin{BREAKEXM}[Katsura dilation: Toeplitz graph]
  \label{EXAMPLE:KATSURA-DILATION:GRAPH-CORRESPONDENCE}
  Consider the correspondence given by the single loop
  \[
    \left(~X=\ell^2\left(\begin{tikzpicture}[baseline=8pt,every loop/.style={ distance=5mm, out=120, in=60, -> }]
      \node (v) at (1,0) {$a$} edge [loop] node[above]{$x$} ();
    \end{tikzpicture}\right)=\COMPLEX x~\Bigg|~ A=c_0\Bigl(\vertices\Bigr)=\COMPLEX a~\right).
  \]
  Its Toeplitz algebra recovers the traditional Toeplitz algebra since
  \[
    \TOEPLITZ(X,A)=\CSTAR\Bigl(~x\cup a~\Big|~x^*x=a,~ax=x~\Bigr) = \CSTAR\Bigl(~x^*x=1~\Bigr)=\TOEPLITZ
  \]
  and as such its \textbf{suggestive name as Toeplitz graph.} On the other hand its absolute Cuntz--Pimsner algebra recovers the traditional circle algebra
  \[
    \PIMSNER(X,A)=\CSTAR\Bigl(~x\cup a~\Big|~x^*x=a=xx^*~\Bigr) = \CSTAR\Bigl(~x^*x=1=xx^*~\Bigr)=\CONTINUOUS(\TORUS).
  \]
  These already define \textbf{all its gauge-equivariant representations:}\\
  Indeed there is only a single covariance ideal given by the single vertex
  \[
    \max(X,A)=c_0\Bigl(~\regular=a=\vertices~\Bigr) = A
  \]
  and no further quotient graph (except the trivial one).\\
  As such we found the entire lattice of gauge-equivariant representations.\\
  Further we may now compute \textbf{the graph for the canonical dilation:}\\
  For this we may now simply read off the graph as (confer corollary \ref{THEOREM:KATSURA-DILATION:GRAPH-CORRESPONDENCE})
  \begin{gather*}
    \Bigl(~W=(a)\amalg(a\setminus\emptyset)= a\amalg a
    ~\Big|~
    EW= (xa)\amalg (xa) = x\amalg x~\Bigr):\\[2\jot]
    \left(~~\begin{gathered}
      s(x\amalg\emptyset)=a\amalg\emptyset,~s(\emptyset\amalg x)=\emptyset\amalg a\\[2\jot]
      r(x\amalg\emptyset)=a\amalg\emptyset=s(\emptyset\amalg x)
    \end{gathered}~~\right)
    \implies
    \left(~\begin{tikzpicture}[x=3cm,baseline=13pt]
      \node (ORIGINAL-VERTEX) at (0,0) {$a\amalg\emptyset$} edge [loop] node[above]{$x\amalg\emptyset$} ();
      \node (COPY-VERTEX) at (1,0) {$\emptyset\amalg a$};
      \draw[->] (COPY-VERTEX) -- node[above]{$\emptyset\amalg x$} (ORIGINAL-VERTEX);
    \end{tikzpicture}~\right).
  \end{gather*}
  Put together we found the \textbf{realisation for the Toeplitz algebra:}
  \[
    \TOEPLITZ\left(\begin{tikzpicture}[baseline=10pt]
      \node (v) at (0,0) {$\bullet$} edge [loop] ();
    \end{tikzpicture}\right) =
    \PIMSNER\left(\begin{tikzpicture}[baseline=10pt]
      \node (ORIGINAL-VERTEX) at (0,0) {$\bullet$} edge [loop] ();
      \node (COPY-VERTEX) at (1,0) {$\bullet$} edge [loop,opacity=0] ();
      \draw[->] (COPY-VERTEX) -- (ORIGINAL-VERTEX);
    \end{tikzpicture}\right)
  \]
  The latter appears sometimes as well under the name Toeplitz graph.
\end{BREAKEXM}

We now continue with the issue about the existence of minimal dilations.
For this we begin with the following positive result for relative graph algebras:
\begin{COR}[Katsura dilation: minimal dilation]
  For graph correspondences the canonical dilation is minimal.
  That is roughly speaking, there is no smaller graph which realises the relative graph algebra as an absolute one.
\end{COR}
\begin{proof}
  Consider a subalgebra detecting each covariance $B\subset c_0(\regular)$:
  \[
    B\cap c_0\Bigl(S\subset\regular\Bigr)=0\TAB\implies\TAB c_0\Bigl(S\subset\regular\Bigr)=0.
  \]
  Then the subalgebra necessarily contains each summand and so also
  \[
    c_0\Bigl(S=\{r\}\Bigr)=\COMPLEX(r)\subset B
    \TAB\implies\TAB
    B={\bigoplus}_r\COMPLEX(r)=c_0\Bigl(\regular\Bigr).
  \]
  As such the maximal covariance defines a minimal dilation.
\end{proof}

The previous observation suggests now the following result for correspondences over spaces (as coefficient algebra)
and in particular for integer actions:

\begin{BREAKCOR}[Katsura dilation: correspondences over spaces]
  It holds for correspondences with abelian maximal covariance (and so also for correspondences over spaces):
  The canonical dilation given by the maximal covariance defines a minimal dilation if and only if the maximal covariance has discrete spectrum
  (resp.~the maximal covariance defines a discrete subspace).\linebreak
  Maximal covariances with totally disconnected spectra are not sufficient.
\end{BREAKCOR}
\begin{proof}Note that the entire problem deals with an abelian operator algebra.\linebreak
  By Stone--Weierstrass it thus suffices to focus on any ideal over some point in the spectrum and any subalgebra over some pair of points of the form
  \[
    \ker\Bigl(\omega: B\to\COMPLEX\Bigr)=\Bigl\{b(\omega)=0\Bigr\}
    \TAB\text{and}\TAB
    \mathrm{eq}\Bigl(\omega_i:B\rightrightarrows\COMPLEX\Bigr)=\Bigl\{b(\omega_1)=b(\omega_2) \Bigr\}.
  \]
  For the ideal we first observe for a pair of closed subspaces $E,F\subset X=\Gamma B$:
  \[
    \ker(E)\cap\ker(F)=\ker(E\cup F)=0\TAB\iff\TAB E\cup F = X.
  \]
  As such the only ideals some non-isolated point cannot detect are
  \[
    \ker(\omega)\cap\ker(F)=0:\TAB \ker\Bigr(F= \overline{X-\omega}=X\Bigr)=0
  \]
  whence the non-isolated point would already detect every ideal.\\
  Suppose on the other hand that the spectrum defines a discrete space.\\
  Then every subalgebra as above fails detect even two ideals:
  \[
    \mathrm{eq}\bigl(\omega_1,\omega_2\bigr)\cap \ker(X-\omega_i)=\Bigl\{b(X-\omega_i)=0=b(\omega_i)\Bigr\}=0
  \]
  The corollary now follows from combining these two observations.\\
  We leave the final statement about totally disconnected spaces to the reader.
\end{proof}

With the previous corollary in mind, we may now illustrate some negative examples on the existence of minimal dilations altogether:

\begin{EXM}[Minimal dilations: failures of existence]
  \label{EXAMPLE:MINIMAL-DILATIONS:FAILURE}
  For this we may now simply consider any integer action as our correspondence,
  \[
    \INTEGERS\act C_0(\SPACE)\TAB\implies\TAB X=C_0(\SPACE)=A.
  \]
  Note that any integer action defines a regular correspondence and as such the coefficient algebra defines the maximal covariance.
  Consider now some continuous space such as for instance the real line for which we thus obtain:\\
  a) The canonical dilation \textbf{fails to be minimal} as for instance the following ideals already detect each covariance
  \[
    C_0(\REALS-r)=\ker(r) \subset C_0(\REALS):\TAB\TAB C_0(\REALS-r)\cap I=0\TAB\implies\TAB I=0.
  \]
  b) On the other hand none of those is minimal either since furthermore detection of covariance already happens for those with discrete complement such as for
  \[
    C_0\Bigl(\REALS-\{r_1,\ldots,r_n\}\Bigr)=\ker\Bigl(r_1\cup\ldots\cup r_n\Bigr)\subset C_0(\REALS).
  \]
  Consider now some enumeration of rational numbers $\{q_1,q_2,\ldots\}=\RATIONALS$.\\
  As such we obtain some decreasing sequence of ideals with detection:
  \[
    \ker(q_1)\supset \ker(q_1\cup q_2)\supset \ldots\supset\ker(q_1\cup\ldots\cup q_n)\supset\ldots
  \]
  Their intersection however fails to detect any covariance since
  \[
    \ker\Bigl(\RATIONALS=\{q_1,q_2,\ldots\}\Bigr)=\ker\Bigl(\overline{\{q_1,q_2,\ldots\}}=\REALS\Bigr)=0.
  \]
  As such the given sequence of ideals admits no minimal dilation.\\
  In fact we have even found that \textbf{the axiom of choice fails to apply!}
\end{EXM}

\section*{Acknowledgements}

The author would like to thank his supervisor S{\o}ren Eilers for his kind support and encouragements,
as well as Evgenios Kakariadis for pointing out to \cite{KATSURA-GAUGE-IDEALS} as a valuable follow-up article by Takeshi Katsura, which inspired the author on finding the missing components.
Moreover, the author acknowledges the support under the Marie–Curie Doctoral~Fellowship~No.~801199.

\appendix

\bibliography{Bibliography}
\bibliographystyle{alpha-all}

\end{document}